\documentclass[a4paper,12pt,reqno]{article}
\usepackage{amsthm,amsfonts,amssymb,euscript}
\usepackage{multicol, fancybox}
\usepackage{graphicx}
\usepackage{color}
\usepackage{amsmath, amsthm, amssymb, bm}
\usepackage{epstopdf}
\usepackage{caption}
\usepackage{psfrag}
\usepackage{fullpage}
\usepackage{authblk}
\usepackage{soul}

\newtheorem{theorem}{Theorem}[section]
\newtheorem{lemma}[theorem]{Lemma}
\newtheorem{proposition}[theorem]{Proposition}
\newtheorem{corollary}[theorem]{Corollary}
\newtheorem{definition}[theorem]{Definition}
\newtheorem{remark}[theorem]{Remark}

\numberwithin{equation}{section}

\usepackage{hyperref}

\newcommand{\non}{\nonumber}

\def\a{{\alpha}}

\def\ga{\gamma}

\def\de{\delta}
\def\De{\Delta}

\def\la{\lambda}
\def\La{\Lambda}
\def\si{\sigma}
\def\Si{\Sigma}
\def\om{\omega}
\def\Om{\Omega}

\def\th{\theta}

\def\nab{\nabla}

\def\AA{{math\cal A}}
\def\BB{{\mathcal B}}
\def\CC{{\mathcal C}}
\def\MM{{\mathcal M}}

\def\LL{{\mathcal L}}

\def\HH{{\mathcal H}}

\def\WW{{\mathcal W}}
\def\VV{{\mathcal V}}

\def\OO{{\mathcal O}}

\def\RR{{\mathcal R}}
\def\QQ{{\mathcal Q}}
\def\AA{{\mathcal A}}

\def\HH{{\mathcal H}}

\def\C{{\bf C}}

\def\E{{\bf E}}

\def\J{{\bf J}}

\def\P{{\bf P}}

\def\T{{\bf T}}
\def\g{{\bf g}}
\def\m{{\bf m}}

\def\u{{\bf u}}
\def\p{{\bf p}}

\def\RRR{{\Bbb R}}

\def\f12{{\frac 1 2}}

\def\div{\mathrm{div}}
\def\curl{\mathrm{curl}}

\def\tr{\mbox{tr}}

\def\f12{\frac 1 2}
\def\bsplit{\begin{split}}

\DeclareFontFamily{U}{mathx}{\hyphenchar\font45}
\DeclareFontShape{U}{mathx}{m}{n}{
      <5> <6> <7> <8> <9> <10>
      <10.95> <12> <14.4> <17.28> <20.74> <24.88>
      mathx10
      }{}
\DeclareSymbolFont{mathx}{U}{mathx}{m}{n}
\DeclareFontSubstitution{U}{mathx}{m}{n}
\DeclareMathAccent{\widecheck}{0}{mathx}{"71}

\newcommand{\half}{\frac{1}{2}}

\usepackage{mathtools}
\renewcommand{\d}{\mathrm{d}}
\newcommand{\dr}{\partial}
\newcommand{\ffi}{\varphi}
\newcommand{\GO}[1]{O\left( #1 \right)}

\DeclareFontFamily{U}{mathx}{\hyphenchar\font45}
\DeclareFontShape{U}{mathx}{m}{n}{
      <5> <6> <7> <8> <9> <10>
      <10.95> <12> <14.4> <17.28> <20.74> <24.88>
      mathx10
      }{}
\DeclareSymbolFont{mathx}{U}{mathx}{m}{n}
\DeclareFontSubstitution{U}{mathx}{m}{n}
\DeclareMathAccent{\widecheck}{0}{mathx}{"71}

\newcommand{\wc}[1]{\widecheck{ #1 }}

\renewcommand{\tr}{\mathrm{tr}}

\newcommand{\pth}[1]{\left( #1 \right)}

\newcommand{\e}{\varepsilon}
\renewcommand{\l}{\left\|}
\renewcommand{\r}{\right\|}
\newcommand{\enstq}[2]{\left\{#1~\middle|~#2\right\}}

\newcommand{\Hess}{\mathrm{Hess}}
\usepackage{bbm}
\usepackage{bm}
\renewcommand{\u}{\mathfrak{u}}
\newcommand{\X}{\mathfrak{X}}
\renewcommand{\d}{\mathrm{d}}
\renewcommand{\p}{\mathfrak{p}}
\renewcommand{\k}{\mathbf{k}}
\newcommand{\coker}{\mathrm{coker}}

\begin{document}

\title{Initial data for Minkowski stability with arbitrary decay}

\author{Allen Juntao Fang\footnote{Universit\"at M\"unster, M\"unster, Deutschland (\href{mailto:allen.juntao.fang@uni-muenster.de}{allen.juntao.fang@uni-muenster.de})}, 
Jérémie Szeftel\footnote{CNRS \& Laboratoire Jacques-Louis Lions, Sorbonne Université, Paris, France (\href{mailto:jeremie.szeftel@sorbonne-universite.fr}{jeremie.szeftel@sorbonne-universite.fr})}  
\,and Arthur Touati\footnote{CNRS \& Institut de Mathématiques de Bordeaux, France (\href{mailto:arthur.touati@math.u-bordeaux.fr}{arthur.touati@math.u-bordeaux.fr})}
}

\maketitle

\begin{abstract}
We construct and parametrize solutions to the constraint equations of general relativity in a neighborhood of Minkowski spacetime with arbitrary prescribed decay properties at infinity. We thus provide a large class of initial data for the results on stability of Minkowski which include a mass term in the asymptotics. Due to the symmetries of Minkowski, a naive linear perturbation fails. Our construction is based on a simplified conformal method, a reduction to transverse traceless perturbations and a nonlinear fixed point argument where we face linear obstructions coming from the cokernels of both the linearized constraint operator and the Laplace operator. To tackle these obstructions, we introduce a well-chosen truncated black hole around which to perturb. The control of the parameters of the truncated black hole is the most technical part of the proof, since its center of mass and angular momentum could be arbitrarily large.
\end{abstract}

\tableofcontents

\section{Introduction}

In this article, we construct rapidly decaying initial data for the spacelike initial value problem in general relativity.

\subsection{The initial value problem in general relativity}

At the heart of general relativity are the Einstein vacuum equations, which are given by
\begin{align}\label{EVE}
\mathrm{Ricci}(\g) & = 0,
\end{align}
where $\g$ is a Lorentzian metric defined on a 4-dimensional manifold $M$, and $\mathrm{Ricci}(\g)$ is the Ricci tensor of $\g$. We refer the reader to \cite{ChoquetBruhat2009} for a rich presentation of the mathematical study of the equation \eqref{EVE} and its equivalent with matter sources. 

The geometric equation \eqref{EVE} in fact corresponds to a Cauchy problem. The Cauchy data, i.e. initial data, are $(\Si,g,k)$, where $\Si$ is a 3-dimensional manifold, $g$ is a Riemannian metric on $\Si$ and $k$ is a symmetric 2-tensor on $\Si$. Solving \eqref{EVE} with the data $(\Si,g,k)$ then amounts to constructing a 4-dimensional Lorentzian manifold $(M,\g)$ with vanishing Ricci tensor and such that $\Si$ is a spacelike hypersurface in $M$ with first and second fundamental forms given by $(g,k)$. This last requirement simply means that $g$ is the induced metric on $\Si$ by $\g$ and that $k = -\half \LL_\T \g$, where $\T$ is the future-directed unit normal to $\Si$ for $\g$ and $\LL$ is the Lie derivative. In her seminal article \cite{ChoquetBruhat1952}, Choquet-Bruhat showed that the Cauchy problem described above is locally solvable if and only if the data satisfy the so-called constraint equations
\begin{equation}\label{C1}
\left\{
\begin{aligned}
R(g) - |k|^2_g + (\tr_g k)^2 & = 0,
\\ \div_g k - \d \tr_g k & = 0,
\end{aligned}
\right.
\end{equation}
on the manifold $\Si$, where $R(g)$ is the scalar curvature of $g$, $|k|^2_g=g^{ij}g^{ab}k_{ia}k_{jb}$, $\tr_g k = g^{ij} k_{ij}$ and $\div_g k_\ell=g^{ij}D_i k_{j\ell}$, where $D$ is the covariant derivative associated to $g$. This nonlinear system for $(g,k)$ is of elliptic nature and a wealth of literature has been produced on the subject, see the review \cite{Carlotto2021}.

The task of solving \eqref{EVE} thus factorizes nicely into an elliptic problem for the data and a hyperbolic problem for the evolution. We focus in this paper on the elliptic side, but construct initial data designed to match the requirements of an important question from the hyperbolic side, i.e. the stability of Minkowski spacetime.

\subsection{Stability of Minkowski spacetime}

The Minkowski spacetime $(\RRR^{1+3}, \m)$ with $\m$ given in Cartesian coordinates by
\begin{align*}
\m = - (\d t)^2 + (\d x^1)^2 + (\d x^2)^2 + (\d x^3)^2,
\end{align*}
is the simplest solution to \eqref{EVE} and corresponds to the framework of special relativity. As a stationary solution of \eqref{EVE}, the question of its nonlinear asymptotic stability is of great importance. More precisely, if one considers asymptotically flat initial data close to the Minkowski data $(\RRR^3,e,0)$, can one show that the resulting spacetime $(M,\g)$ is geodesically complete and converges back to Minkowski in a particular gauge? 

This question has been answered positively first in the seminal work \cite{Christodoulou1993}, and has been revisited several times since. We would like to classify these various results depending on the decay of $(g,k)-(e,0)$ at infinity. More precisely, the data are such that\footnote{We denote by $\mathbb{N}$ the set of nonnegative integers and by $\mathbb{N}^*$ the set of positive integers.} 
\begin{align}
| g - e | + r |k| \lesssim \frac{\e}{(1+r)^{q+\de}}, \qquad q\in\mathbb{N}, \quad 0<\de<1,\label{AF data}
\end{align}
with obvious extension to derivatives of $g$ and $k$. If $q\geq 1$, a mass term decaying like $r^{-1}$ needs to be included in \eqref{AF data}; we will be very precise on this issue later in the introduction but for now we simply divide into two categories a (non-exhaustive) list of results on the stability of Minkowski. First, there are the results with $q=0$:
\begin{itemize}
\item \cite{Bieri2010} deals with $\de =\half$, 
\item \cite{LeFloch2022} deals with $\half<\de<1$ for self-gravitating massive fields,
\item \cite{Ionescu2022} deals with $\de=1-\de'$ (with $\de'>0$ small) for the Einstein-Klein-Gordon,
\item \cite{Shen2023a} deals with the full range $0<\de<1$.
\end{itemize}
Then come the results with $q\geq 1$:
\begin{itemize}
\item \cite{Christodoulou1993} deals with $q=1$ and $\de=\half$,
\item \cite{Lindblad2010}, \cite{Hintz2020} and \cite{Hintz2023} deal with $q=1$ and the full range $0<\de<1$.
\end{itemize}
While the above results concern the global stability of Minkowski, i.e. the geodesic completeness of the maximal globally hyperbolic development of the initial data, we also mention three results, still in the case $q\geq 1$, proving stability of Minkowski restricted to the outside of a far-away outgoing light cone:
\begin{itemize}
\item \cite{Klainerman2003} deals with $q=1$ and $\de=\half$,
\item \cite{Klainerman2003a} deals with $q\geq 3$, $0<\de<1$ and applications to peeling estimates and the smoothness of null infinity,
\item \cite{Shen2023} deals with $q\geq 1$ and $0<\de<1$.
\end{itemize}
Since hyperbolic estimates propagate decay in the evolution, stronger spatial decay on the data implies stronger spacetime decay of the perturbations, which in turn provides better control of nonlinear terms. Therefore, for the evolution problem, the more decay in $r$ in the initial data, the easier it is to prove stability of Minkowski. 

On the other hand, for the elliptic problem, the opposite occurs. Constructing initial data is easier the lower the rate of $r$-decay specified. This is related to the idea that the construction of initial data requires the inversion of elliptic operators in weighted Sobolev spaces. These operators are invertible when $q=0$, and initial data for the articles corresponding to $q=0$ can easily be constructed (see for example \cite{ChoquetBruhat2000a}). The goal of this article is to construct and parametrize solutions of the constraint equations near Minkowski spacetime with $q\geq 1$ and $0<\de<1$ both arbitrary\footnote{The case $\de=0$ cannot be considered here since we rely on Frehdolm properties of the Laplacian which are well-known to fail on weighted Sobolev spaces with integer power decay.}. In particular, this constructs the initial data needed for the stability of Minkowski proved in \cite{Christodoulou1993,Klainerman2003,Klainerman2003a,Lindblad2010,Hintz2020,Hintz2023,Shen2023} and it is nontrivial as an elliptic problem.

\subsection{Different approaches for the constraint equations}

We now focus on the constraint equations \eqref{C1}. The constraint equations form an underdetermined system of four equations for twelve unknowns, and thus the difficulty is not to find \textit{some} solutions but rather to parametrize the space of \textit{all} solutions. The most successful strategy is the conformal method, first introduced in \cite{Lichnerowicz1944}. In its most basic form, it considers solutions of the constraint equations of the form
\begin{align*}
(g,k) = \pth{ \ffi^4 \bar{g}, \ffi^{-2} \pth{ \si + K_{\bar{g}}W} + \frac{\tau}{3} \ffi^4 \bar{g} },
\end{align*}
where $\bar{g}$ is a Riemannian metric, $\si$ is a traceless and divergence-free 2-tensor and $\tau$ is the mean curvature, and the unknowns are $\ffi$ and the vector field $W$ (with $K_{\bar{g}}W = \LL_W\bar{g} - \frac{2}{3} (\div_{\bar{g}}W)\bar{g} $). The system \eqref{C1} is then recast as a semilinear elliptic system for $(\ffi,W)$. This strategy has proved to be very efficient in constructing solutions on compact manifolds with constant mean curvature (CMC) or not (see for instance \cite{Isenberg1995} and \cite{Maxwell2009}) or on asymptotically flat manifolds (see for instance \cite{ChoquetBruhat2000a} and \cite{Dilts2014}). We also mention the recent drift formulation from \cite{Maxwell2021} aimed at exploring far from CMC regions of the space of solutions.

As mentionned at the end of the previous section, from the point of view of stability of Minkowski or more generally of perturbative regimes in the case $q\geq 1$, the conformal method and its refinements do not produce appropriate initial data, since it always relies on $\De_{\bar{g}}$ being an isomorphism on fields decaying like $r^{-\de}$ for $0<\de<1$ (thus corresponding to $q=0$ in \eqref{AF data}), a fact demonstrated for asymptotically flat metric $\bar{g}$ in the seminal \cite{Bartnik1986}. If one is interested in stronger and stronger decay, one must take into account the larger and larger cokernel of $\De_{\bar{g}}$.

Focusing now on the asymptotically flat case, another successful approach has been the various gluing techniques. Starting with the seminal paper \cite{Corvino2000} followed by \cite{Chrusciel2003} and \cite{Corvino2006}, the idea is to glue two solutions of \eqref{C1}, and thus create an interesting third one, equal to an arbitrary solution when restricted to a compact set and to a member of a family of reference solutions outside of a larger compact set. The gluing \textit{per se} happens in a far away annulus, and the family of reference solutions is often taken as the family of Kerr initial data sets. Because of the nontrivial cokernel of the linearized constraint operator around the Minkowski data (denoted $D\Phi[e,0]$ later in this article), this gluing strategy faces linear obstructions with the following consequence: one cannot choose the Kerr initial data set to which one glues. Recently, \cite{Czimek2022a} and \cite{Mao2023} have been able to produce obstruction-free gluing results, i.e. they show how one can choose the Kerr initial data set to which one glues.

From the point of view of perturbative regimes, solutions produced by gluing techniques are too restrictive since they coincide with an exact solution outside of a compact set. Thanks to the finite speed of propagation, the corresponding stability problem is then localized inside an outgoing cone (see for example \cite{Lindblad2005}). Note that in \cite{Chrusciel2003} stationnary spacetimes outside of a compact set are constructed via gluing, they are more general than exact black holes solutions but aren't generic perturbations.

In order to construct nontrivial initial data for the stability of Minkowski problem, i.e. data close to $(e,0)$, arbitrarily decaying and not restricted to a compact set, we employ a simplified version of the conformal method and also face linear obstructions which now come from the intersection $\coker(\De)\cap\coker(D\Phi[e,0])$. The fact that this intersection is not trivial forces us to include and control truncated black holes of small masses in the asymptotics, very much as in gluing constructions. While the truncated black hole solutions that we consider have small mass, we will show that they will have potentially large center of mass and angular momentum. Controlling these potentially large parameters for the truncated black hole solutions is a primary complication in the analysis.

\subsection{First version of the main result}

In this section we present a first version of our main result, see Theorem \ref{maintheorem} for the precise statement. First, we reformulate the constraint equations \eqref{C1} by introducing the reduced second fundamental form
\begin{align*}
\pi & \vcentcolon= k - (\tr_g k)g.
\end{align*}
Note that this relation is invertible, i.e. $k=\pi - \half (\tr_g\pi)g$. The constraint equations for the unknown $(g,\pi)$ read
\begin{equation}\label{C2}
\left\{
\begin{aligned}
R(g)  + \half (\tr_g \pi)^2 - |\pi|^2_g & = 0,
\\ \div_g \pi & = 0.
\end{aligned}
\right.
\end{equation}
For further use, we introduce the notations
\begin{align*}
\HH(g,\pi) & \vcentcolon= R(g)  + \half (\tr_g \pi)^2 - |\pi|^2_g,
\\ \MM(g,\pi) & \vcentcolon= \div_g \pi,
\\ \Phi(g,\pi) & \vcentcolon=(\HH(g,\pi),\MM(g,\pi)),
\end{align*}
so that \eqref{C2} rewrites $\Phi(g,\pi)=0$. For $X$ a vector field on $\RRR^3$ and $g$ a Riemannian metric we define 
\begin{align*}
L_g X \vcentcolon= \LL_X g - (\div_g X)g.
\end{align*}
The symmetric 2-tensors defined on $\RRR^3$ which are traceless and divergence free with respect to the Euclidean metric $e$ play a major role in this article, and such a tensor is said to be TT (for traceless and transverse). When needed, we will also denote by $TT$ the set of TT tensors.

\begin{theorem}[Main result, version 1]\label{rough theorem}
Let $q\in\mathbb{N}^*$, $0<\de<1$ and $\e>0$. Let $(\wc{g},\wc{\pi})$ be TT tensors with
\begin{align}\label{rough decay}
\sup_{k=0,1}\pth{ r^k|\nab^k\wc{g}| + r^{k+1} |\nab^k\wc{\pi}|} \leq \frac{\e}{(1+r)^{q+\de}},
\end{align}
and set $\eta^2 \vcentcolon = \frac{1}{4} \l \nab \wc{g} \r_{L^2}^2 + \l \wc{\pi}\r_{L^2}^2$. If $\eta>0$ and if $\e$ is small enough, then there exists a solution of \eqref{C2} on $\RRR^3$ which for large $r$ is of the form
\begin{equation}\label{rough expansion}
\begin{aligned}
g & = g_{\vec{p}} + \wc{g} + 4\wc{u} e + \GO{\e r^{-q-\de-1}},
\\ \pi & = \pi_{\vec{p}} + \wc{\pi} + L_e \wc{X} + \GO{\e r^{-q-\de-2}}.
\end{aligned}
\end{equation}
Here, $(g_{\vec{p}},\pi_{\vec{p}})$ is a Kerr initial data set of mass $m$, center of mass $y$ and angular momentum $a$ such that
\begin{align}\label{rough p}
\eta^2 \lesssim m \lesssim \eta^2, \qquad |y| + |a| \lesssim \pth{\frac{\e}{\eta}}^{\frac{2}{2q+2\de-1}},
\end{align}
and the scalar function $\wc{u}$ and the vector field $\wc{X}$ satisfy
\begin{align*}
|\wc{u}| + \left| \wc{X}\right| & \lesssim \frac{\eta\e}{(1+r)^{q+\de}}.
\end{align*}
\end{theorem}

Some comments are in order. 
\begin{itemize}
\item In Theorem \ref{rough theorem} we actually construct a map from the "seed" decaying TT tensors $(\wc{g},\wc{\pi})$ to the space of solutions to \eqref{C2}, i.e. both the Kerr initial data set $(g_{\vec{p}},\pi_{\vec{p}})$ and the correctors $\pth{\wc{u},\wc{X}}$ are constructed out of $(\wc{g},\wc{\pi})$. In Proposition \ref{prop uniqueness}, we will show that this map is injective so that formally we construct a graph of solutions over the set of decaying TT tensors. This cannot be seen on the expansion \eqref{rough expansion} and requires the exact statement of Theorem \ref{maintheorem}. The assumption that our seed tensors $(\wc{g},\wc{\pi})$ are TT might seem restrictive but we will actually show that thanks to the diffeomorphism invariance it is generic in a neighborhood of Minkowski, see Section \ref{section param TT}.
\item As announced at the end of the previous section, the center of mass $y$ and the angular momentum $a$ of the constructed Kerr initial data set $(g_{\vec{p}},\pi_{\vec{p}})$ are large. Indeed, the decay assumption \eqref{rough decay} implies that $\eta\lesssim \e$, so that $\pth{\frac{\e}{\eta}}^{\frac{2}{2q+2\de-1}}$ appearing in \eqref{rough p} is potentially large. The reason for this will be explained in the next section where we will in particular explain how $(g_{\vec{p}},\pi_{\vec{p}})$ is constructed. We can already say that the largeness of $y$ and $a$ is the main technical difficulty in our proof, where we will often crucially need to extract the small constant $\eta$ instead of the small (but larger) constant $\e$. Note that the largeness of $\frac{|a|}{m}$ does not contradict the Penrose inequality since there are no apparent horizons in the resulting manifold constructed by Theorem \ref{rough theorem} (we will introduce cutoffs such that we don't see the black hole's horizon).
\item The solutions constructed in Theorem \ref{rough theorem} can be glued to interior solutions of the constraint equations using the obstruction-free gluing results of \cite{Czimek2022a} and \cite{Mao2023}, thus providing a large number of solutions that asymptote to trivial initial data at arbitrary rate. 
\item While in this paper we focus on the case $q\in\mathbb{N}^*$, the case $q=0$ is well-known and significantly simpler thanks to the good inversion properties of the Laplacian in this range. See for instance \cite{Huang2010}, where the author constructs perturbations around any asymptotically flat solution of the constraint equations for $q=0$ and $\de\in\pth{\frac{1}{2},1}$.
\end{itemize}

Of independent interest is the following corollary (see Corollary \ref{coro TS} for a precise statement) where solutions to $R(g)=0$ are constructed in a neighborhood of Minkowski.

\begin{corollary}[Time-symmetric case, version 1]\label{rough coro TS}
Let $q\in\mathbb{N}^*$, $0<\de<1$ and $\e>0$. Let $\wc{g}$ be a TT tensor with
\begin{align*}
\sup_{k=0,1}r^k|\nab^k\wc{g}|  \leq \frac{\e}{(1+r)^{q+\de}},
\end{align*}
and set $\eta^2 \vcentcolon =  \l \nab \wc{g} \r_{L^2}^2 $. If $\eta>0$ and if $\e$ is small enough, then there exists a solution of $R(g)=0$ on $\RRR^3$ which for large $r$ is of the form
\begin{equation*}
\begin{aligned}
g & = g_{\vec{p}} + \wc{g} + 4\wc{u} e + \GO{\e r^{-q-\de-1}}.
\end{aligned}
\end{equation*}
Here, $g_{\vec{p}}$ is a Schwarzschild initial data set of mass $m$ and center of mass $y$ such that
\begin{align}\label{rough p}
\eta^2 \lesssim m \lesssim \eta^2, \qquad |y|  \lesssim \pth{\frac{\e}{\eta}}^{\frac{2}{2q+2\de-1}},
\end{align}
and the scalar function $\wc{u}$ satisfies
\begin{align*}
|\wc{u}|  & \lesssim \frac{\eta\e}{(1+r)^{q+\de}}.
\end{align*}
\end{corollary}

\subsection{Sketch of proof}

In this section, we intend to present the main features of our proof. We first consider so-called "seed" initial data in the form of $(\wc{g},\wc{\pi})\in S^2T^*\Si\times S^2T^*\Si$ a pair of symmetric 2-tensors over the initial slice $\Si$, and search for solutions of \eqref{C2} of the form 
\begin{align}\label{rough ansatz}
(g,\pi) = \pth{ u^4 (e + \wc{g}), \wc{\pi} + L_e X}.
\end{align}
This is a simplification of the ansatz proposed in \cite{Corvino2006}, where $\pi$ would be given by $\pi = u^2 (\wc{\pi} + L_{e+\wc{g}}X)$. The constraint equations \eqref{C2} then reads
\begin{align}\label{rough eq u X}
\De \pth{ u - 1, -X} = \Phi(e+\wc{g},\wc{\pi}) + \text{lower order terms},
\end{align}
where $\De$ is the standard Euclidean Laplacian operator. In this discussion we completely neglect the lower order terms. If $\wc{g}$ and $\wc{\pi}$ satisfy \eqref{rough decay} then thanks to standard properties of $\De$ in weighted Sobolev spaces we can invert \eqref{rough eq u X} and obtain 
\begin{align}\label{rough expansion}
\pth{ u - 1, -X} & = \sum_{j=1}^q \sum_{\ell=-(j-1)}^{j-1} \left\langle \Phi(e+\wc{g},\wc{\pi}) , \WW_{j,\ell} \right\rangle \VV_{j,\ell} + \pth{\tilde{u},-\tilde{X}},
\end{align}
where $\WW_{j,\ell}$ denotes roughly a basis of homogeneous polynomials\footnote{These are in fact harmonic polynomials and vectorial harmonic polynomials, see Definition \ref{def tool laplace}.} of degree $j-1$ and $\VV_{j,\ell}$ decays like $r^{-j}$, $\tilde{u}$ and $\tilde{X}$ are small remainders decaying like $r^{-q-\de}$ and $\langle f , h \rangle = \int_{\RRR^3} fh$. We actually have 
\begin{align}\label{coker}
\coker\pth{ \De } = \mathrm{Span}\pth{ \WW_{j,\ell}, 1\leq j \leq q, -(j-1)\leq \ell \leq j-1 },
\end{align}
if $\De$ denotes the operator\footnote{Here, $H^2_{-q-\de}$ and $L^2_{-q-\de-2}$ denote weighted Sobolev spaces, see Definition \ref{def weighted sobolev spaces}.}  $\De : H^2_{-q-\de}\longrightarrow L^2_{-q-\de-2}$ (acting on couples formed by a scalar function and a vector field expressed in Cartesian coordinates). See Section \ref{section elliptic} for complete details on the inversion of an equation like \eqref{rough eq u X}. In order to prove our theorem, we need to cancel the scalar products appearing in \eqref{rough expansion}. If we expand the constraint operator $\Phi$ we obtain
\begin{align}\label{scalar product}
\left\langle \Phi(e+\wc{g},\wc{\pi}) , \WW_{j,\ell} \right\rangle & = \left\langle (\wc{g},\wc{\pi}) , D\Phi[e,0]^*\pth{\WW_{j,\ell}} \right\rangle + \text{nonlinear terms}.
\end{align}
where $D\Phi[e,0]$ is the first variation of the constraint operator at $(e,0)$ and $D\Phi[e,0]^*$ denotes the $L^2$-formal adjoint operator. A naive way to generate solutions to the constraint equations that decay like $r^{-q-\de}$ near infinity would then be to find compact perturbations $(\breve{g},\breve{\pi})$ such that
\begin{align*}
\left\langle (\breve{g},\breve{\pi}), D\Phi[e,0]^*(\mathcal{W}_{j,\ell}) \right\rangle = - \left\langle (\wc{g},\wc{\pi}), D\Phi[e,0]^*(\mathcal{W}_{j,\ell}) \right\rangle
\end{align*}
to linearly cancel out the scalar product on the RHS of \eqref{scalar product} and linearly absorb \[\left\langle (\wc{g},\wc{\pi}) , D\Phi[e,0]^*\pth{\WW_{j,\ell}} \right\rangle.\] This however requires $ D\Phi[e,0]^*\pth{\WW_{j,\ell}}$ to be non-zero. Thus, according to \eqref{coker}, we need to understand the intersection
\begin{align}\label{intersection}
\coker\pth{ \De } \cap \coker\pth{D\Phi[e,0] }.
\end{align}
The famous result of \cite{Moncrief1975} tells us that elements of $\coker\pth{D\Phi[g,\pi] }$ (usually called KIDS in the literature, for Killing Initial Data Sets) are identified with the projections on the initial hypersurface under consideration of the Killing fields of the spacetime emanating from the initial data set $(g,\pi)$. The Killing fields of Minkowski are the four translations $\dr_\mu$, the three rotations $x_i\dr_j - x_j\dr_i$ and the three boosts $t\dr_i+x_i\dr_t$. It is a straightforward exercise to check that the associated Killing initial data sets of the Killing fields of Minkowski in fact all belong to $\coker\pth{ \De }$ so that \[\coker\pth{ \De } \cap \coker\pth{D\Phi[e,0] }= \coker\pth{D\Phi[e,0] }\] forms a 10-dimensional space of linear obstructions.

\begin{remark}
In this regard, the Minkowski spacetime is the hardest case to treat. First, it is maximally symmetric, i.e. the space of Killing vector fields has the largest dimension possible (10 for a 3+1 spacetime) and, thanks to \cite{Moncrief1975}, so does $\coker\pth{D\Phi[g,\pi] }$. Second, the intersection \eqref{intersection} is maximal since $\coker\pth{D\Phi[e,0] }\subset\coker\pth{\De }$. 
\end{remark}

We modify the ansatz \eqref{rough ansatz} in order to deal with the scalar products in \eqref{rough expansion}: we introduce compactly supported perturbations $\breve{g}$ and $\breve{\pi}$, and a truncated black hole initial data set $(\chi_{\vec{p}}(g_{\vec{p}}-e),\chi_{\vec{p}}\pi_{\vec{p}})$ and consider the now complete ansatz
\begin{align}\label{rough ansatz 2}
(g,\pi) = \pth{ u^4 (e + \chi_{\vec{p}}(g_{\vec{p}}-e) + \wc{g} + \breve{g}), \chi_{\vec{p}}\pi_{\vec{p}} +  \wc{\pi} + \breve{\pi} + L_e X}.
\end{align}
The family of $(\chi_{\vec{p}}(g_{\vec{p}}-e),\chi_{\vec{p}}\pi_{\vec{p}})$ are constructed in Section \ref{section kerr family} following \cite{Chrusciel2003}. They are parametrized by the black hole parameter $\vec{p}=(m,y,a,v)\in\RRR_+\times \RRR^3\times \RRR^3\times B_{\RRR^3}(0,1)$ where $m$ is the mass of the black hole, $y$ is the center of mass of the black hole, $a$ is the angular momentum of the black hole and $v$ corresponds to a Lorentz boost. The cutoff $\chi_{\vec{p}}$ is equal to 0 if $r\leq \la$ and to 1 if $r\geq 2\la$ where $\la=\max(2|y|,2|a|,3)$. They precisely form a 10-dimensional family of asymptotics. The goal will now be to choose the correct black hole parameter $\vec{p}$ and compact perturbation $(\breve{g},\breve{\pi})$ so that $u$ and $X$ in \eqref{rough ansatz 2} decay like $r^{-q-\de}$.

\subsubsection{The system for the black hole parameter}

The 10-dimensional family of $(\chi_{\vec{p}}(g_{\vec{p}}-e),\chi_{\vec{p}}\pi_{\vec{p}})$ are added to cancel the scalar products $\left\langle \Phi(e+\wc{g},\wc{\pi}) , \WW_{j,\ell} \right\rangle$ in the case where $\WW_{j,\ell}\in\coker\pth{D\Phi[e,0] }$. Recall that this is precisely the case where the leading term in \eqref{scalar product} are the quadratic terms. Thus, the black hole needs to solve
\begin{align}\label{scalar p}
\left\langle D\Phi[e,0](\chi_{\vec{p}}(g_{\vec{p}}-e),\chi_{\vec{p}}\pi_{\vec{p}}), \WW_{j,\ell} \right\rangle & = - \left\langle \half D^2\Phi[e,0]((\wc{g},\wc{\pi}),(\wc{g},\wc{\pi})) , \WW_{j,\ell} \right\rangle + \mathrm{l.o.t} ,
\end{align}
where $D^2\Phi[e,0]$ is the second variation of the constraint operator. As we will show using the well-known asymptotics charges (see for example \cite{Chrusciel2003}), the expressions on the LHS of \eqref{scalar p} actually precisely recover the black hole parameters (see Proposition \ref{prop kerr}). 

\paragraph{Lower bound on the mass.} If $\WW_{j,\ell}$ in \eqref{scalar p} corresponds to the time translation, we get the equation for the mass
\begin{align}\label{rough eq m}
16\pi \ga m & = \frac{1}{4} \l \nab\wc{g}\r_{L^2}^2 + \l \wc{\pi} \r_{L^2}^2 - \half \l \div\wc{g} \r_{L^2}^2 + \frac{1}{4} \l \nab\tr \wc{g} \r_{L^2}^2 - \half \l \tr\wc{\pi} \r_{L^2}^2 +  \mathrm{l.o.t} ,
\end{align}
where $\ga=\frac{1}{\sqrt{1-|v|^2}}$ is the Lorentz factor. Note that $m$ is not only the mass of the black hole $(g_{\vec{p}},\pi_{\vec{p}})$ but also the mass of our solution \eqref{rough ansatz 2}. The positive mass theorem, proved in a neighborhood of Minkowski in \cite{ChoquetBruhat1976} (where only the time-symmetric case is studied) and then in full generality in \cite{Schoen1979} and \cite{Witten1981}, implies that if $m=0$ then $(g,\pi)$ is isometric to $(e,0)$. Therefore, to ensure a good parametrization of a neighborhood of $(e,0)$, we need coercivity of the quadratic form appearing in \eqref{rough eq m}. This is achieved by assuming that $(\wc{g},\wc{\pi})$ are TT tensors (in \cite{ChoquetBruhat1976} harmonic coordinates are used, see also \cite{ChoquetBruhat1979}), and we show in Section \ref{section param TT} that this is generic in a neighborhood of Minkowski up to a choice of diffeomorphism, i.e. up to a choice of a different initial hypersurface. Therefore, \eqref{rough eq m} becomes
\begin{align}\label{rough eq m 2}
16\pi \ga m & = \frac{1}{4} \l \nab\wc{g}\r_{L^2}^2 + \l \wc{\pi} \r_{L^2}^2 + \mathrm{l.o.t}.
\end{align}
In particular the mass cannot be zero if $(\wc{g},\wc{\pi})$ are not trivial.

\paragraph{Coupling with the boost.} Looking at \eqref{rough eq m 2}, we see that a lower bound on the mass is equivalent to an upper bound on $\ga$. From \eqref{scalar p} in the case where $\WW_{j,\ell}$ corresponds to space translations, we can get an equation for $v$ which, plugged into \eqref{rough eq m 2}, finally implies
\begin{align}\label{rough eq m 3}
16\pi m & = \sqrt{1 - J(\wc{g},\wc{\pi})^2} \pth{ \frac{1}{4} \l \nab\wc{g}\r_{L^2}^2 + \l \wc{\pi} \r_{L^2}^2 } + \mathrm{l.o.t},
\end{align}
where the functional $J$ is defined by
\begin{align*}
J(\wc{g},\wc{\pi}) & = \frac{\sqrt{ \sum_{k=1,2,3}\pth{\int_{\RRR^3}\wc{\pi}^{ij}\dr_k \wc{g}_{ij} }^2}}{\frac{1}{4} \l \nabla \wc{g} \r_{L^2}^2 + \l \wc{\pi} \r_{L^2}^2 }.
\end{align*}
In Section \ref{section restriction to a cone} we study the functional $J$ and show that it is not uniformly far from 1 on the set of TT tensors. Therefore our construction might degenerate if $(\wc{g},\wc{\pi})$ are such that $J(\wc{g},\wc{\pi}) $ goes close to 1, and we need to perform our construction over the set of TT tensors $(\wc{g},\wc{\pi})$ such that $J(\wc{g},\wc{\pi})\leq 1 - \a$ for some $\a$ arbitrary in $(0,1]$ (note that this is a cone in the space of TT tensors). Since $\a$ is a universal constant, \eqref{rough eq m 3} finally implies the lower bound $m \gtrsim \eta^2$ (where $\eta$ is defined in Theorem \ref{rough theorem}).

\paragraph{Large center of mass and angular momentum.} From \eqref{scalar p} when $\WW_{j,\ell}$ corresponds to the three rotations and the three boosts and from the lower bound on the mass, we then obtain the following bound for the center of mass and angular momentum
\begin{align}\label{rough y a}
|y| + |a| \lesssim \eta^{-2} \int_{\RRR^3} \pth{ |\nab\wc{g}|^2 + |\wc{\pi}|^2} r + \mathrm{l.o.t}.
\end{align}
Because of the $r$ weight in \eqref{rough y a} we need to use the decay assumption \eqref{rough decay}, which means that we cannot fully compensate the $\eta^{-2}$ coming from the mass. The best we can do is to interpolate sharply between weighted and non-weighted norms and get
\begin{align}\label{rough y a 2}
|y| + |a| \lesssim \pth{\frac{\e}{\eta}}^{\frac{2}{2q+2\de-1}}.
\end{align}
The largeness of $y$ and $a$, coming from the two different small constants $\eta$ and $\e$, is at the origin of most of the difficulties we face. Getting $\eta$ instead of $\e$ in nonlinear estimates will often be crucial, for instance for the contraction estimates in the fixed point arguments. In particular, the largeness of $y$ and $a$ implies that the cutoffs $\chi_{\vec{p}}$ are localized in far regions and we need to be careful when estimating weighted norms (the power in \eqref{rough y a 2} will be crucial).

\subsubsection{The compact perturbations} 

The tensors $(\breve{g},\breve{\pi})$ are added to cancel the scalar products $\left\langle \Phi(e+\wc{g},\wc{\pi}) , \WW_{j,\ell} \right\rangle$ in the case where $\WW_{j,\ell}\notin\coker\pth{D\Phi[e,0] }$. Therefore, at leading order they need to solve a finite dimensional and linear system of the form
\begin{equation}\label{rough eq breve}
\left\langle (\breve{g},\breve{\pi}) , D\Phi[e,0]^*(\WW_{j,\ell}) \right\rangle = \QQ_{j,\ell},
\end{equation}
for some numbers $\QQ_{j,\ell}$ depending on $(\wc{g},\wc{\pi})$ and for every $\WW_{j,\ell}\notin\coker\pth{D\Phi[e,0] }$. We solve for $(\breve{g},\breve{\pi})\in \CC_q$ with $\CC_q$ a finite dimensional space of compactly supported 2-tensors. Since the family 
\begin{align*}
\pth{D\Phi[e,0]^*(\WW_{j,\ell})}_{\WW_{j,\ell}\in\coker\pth{ \De } \setminus \coker\pth{D\Phi[e,0] }}
\end{align*}
is linearly independent, a definition of $\CC_q$ ensuring the invertibility of \eqref{rough eq breve} is not difficult. However, as first explained in the remarks following Theorem \ref{rough theorem}, we want a uniqueness statement i.e. that the map $(\wc{g},\wc{\pi})\longmapsto (g,\pi)$ defined by \eqref{rough ansatz 2} is injective. Therefore, we need  
\begin{align}\label{complementary}
\CC_q \cap \pth{ \pth{ TT + \mathrm{Span}(e)}\times \pth{ TT + \mathrm{ran}(L_e)} } = \{ 0 \}
\end{align}
where $TT$ denotes the space of TT tensors and $\mathrm{ran}(L_e)=\left\{ L_e X \middle| X \in T\Si\right\}$. The construction of $\CC_q$ satisfying \eqref{complementary} is performed in Section \ref{section constraint preli}.

\subsection{Outline of the article}

We conclude this introduction by giving an outline of the rest of the article.
\begin{itemize}
\item In Section \ref{section préliminaires}, we introduce the different tools and results necessary to our construction. This includes standard notations, elliptic theory on $\RRR^3$, properties of the constraint operators and the family of Kerr initial data sets.
\item In Section \ref{section main results}, we introduce our simplified conformal formulation of the constraint equations, reduce the space of seed tensors and then state our results.
\item In Section \ref{section main proof} we prove our main result, Theorem \ref{maintheorem}.
\end{itemize}
The article is concluded by two appendices: Appendix \ref{appendix section 2} contains all the proofs of Section \ref{section préliminaires} and Appendix \ref{appendix proof section 3} contains all the proofs of Section \ref{section main results}.

\subsection{Acknowledgments}

The first author acknowledges support from NSF award DMS-2303241 and support through Germany’s Excellence Strategy EXC 2044 390685587, Mathematics Münster: Dynamics–Geometry–Structure. The second author is supported by the ERC grant ERC-2023 AdG 101141855 BlaHSt.

\section{Preliminaries}\label{section préliminaires}

\subsection{Notations}

We will work on the manifold $\RRR^3$ with Cartesian coordinates $(x^1,x^2,x^3)$ and $r=|x|$ with
\begin{align*}
|x|=\sqrt{(x^1)^2+(x^2)^2+(x^3)^2}.
\end{align*} 
The standard Euclidean metric will be denoted by $e$, i.e.
\begin{align*}
e =  (\d x^1)^2 + (\d x^2)^2 + (\d x^3)^2.
\end{align*}
Latin indexes will always refer to Cartesian coordinates and will be raised with respect to $e$. Moreover, the Einstein summation convention will be used for couples of repeted indexes with one up and one down.

If $f$ is a scalar function on $\RRR^3$ we define its gradient by $\nab f = (\dr_1 f,\dr_2 f,\dr_3 f)$ and its Hessian matrix by $(\Hess f)_{ij} = \dr_i \dr_j f$. If $X$ is a vector field on $\RRR^3$ its components in Cartesian coordinates are denoted $X_i$ and if $T$ is a symmetric 2-tensor on $\RRR^3$ its components in Cartesian coordinates are denoted $T_{ij}$. Moreover we use the notations
\begin{align*}
|\nab X|^2 &  = \dr^i X^j \dr_i X_j, & \tr T &  = \de^{ij} T_{ij},
\\ \div X &  = \dr^i X_i, & | \nab T |^2 &  = \dr^kT^{ij} \dr_kT_{ij},
\\ \nab\otimes X_{ij} &  = \dr_i X_j + \dr_j X_i, & \div T_i &  = \dr^j T_{ji}.
\end{align*}
We recall the following standard definition of a TT tensor.
\begin{definition}
A symmetric 2-tensor $S$ is called TT when $\tr S=0$ and $\div S=0$.
\end{definition}

We employ the following standard notations for inequalities: $A\lesssim B$ means that $A\leq C B$ for $C$ a universal constant; $C(c_1,\dots,c_n)$ will denote any polynomial expression (without constant coefficient) of the real numbers $c_i>0$; $A\lesssim_{c_1,\dots,c_n} B$ means that $A\leq C(c_1,\dots,c_n) B$. 

Finally, if $E$ is a normed vector space, then $B_E(x,r)$ denotes the open ball of center $x\in E$ and radius $r\geq 0$. 

\subsection{Weighted elliptic theory on $\RRR^3$}\label{section elliptic}

We define the usual weighted Sobolev and Hölder spaces on $\RRR^3$. Note that if not otherwise specified, every function spaces (such as the usual Lebesgue or Sobolev spaces) will be defined on $\RRR^3$.

\begin{definition}\label{def weighted sobolev spaces}
Let $k\in\mathbb{N}$ and $\de\in\RRR$. 
\begin{itemize}
\item[(i)] We define the weighted Sobolev spaces $H^k_\de$ as the completion of the space of smooth and compactly supported functions for the norm
\begin{align*}
\l u \r_{H^k_\de} = \sum_{|\a|\leq k} \l (1+r)^{-\de-\frac{3}{2} +|\a|}\nabla^\a u \r_{L^2}.
\end{align*}
We use the standard notation $L^2_\de=H^0_\de$.
\item[(ii)] We define the space $C^k_\de$ as the completion of the space of smooth and compactly supported functions for the norm
\begin{align*}
\l u \r_{C^k_\de} = \sum_{|\a|\leq k} \l (1+r)^{-\de+|\a|}\nabla^\a u \r_{L^\infty}.
\end{align*}
\end{itemize}
\end{definition}

These definition are easily extended to tensors of all kinds by summing over all components in Cartesian coordinates. The spaces $H^k_\de$ were first introduced in \cite{Cantor1975} and their properties studied in \cite{ChoquetBruhat1981} (see also the first appendix of \cite{ChoquetBruhat2009}). Note that we use Bartnik's convention from \cite{Bartnik1986} in order to ease the reading of the decay rates, as shown by next lemma's first part.

\begin{lemma} \label{lem plongement}
We have the following properties:
\begin{itemize}
\item[(i)]If $k>m+\frac{3}{2}$, then we have the continuous embedding $H^k_\de\subset C^m_\de$.
\item[(ii)]  If $k\leq \min(k_1,k_2)$, $k<k_1+k_2 - \frac{3}{2}$ and $\de>\de_1+\de_2$ we have the continuous embedding
\begin{align*}
H^{k_1}_{\de_1}\times H^{k_2}_{\de_2} \subset H^{k}_{\de}.
\end{align*}
\end{itemize}
\end{lemma}

\begin{proof}
See \cite{ChoquetBruhat2009}.
\end{proof}

The following definition contains all the various functions and vector fields required to invert the Laplace operator $\De\vcentcolon=\dr^i \dr_i$ on $\RRR^3$.

\begin{definition}\label{def tool laplace}
We introduce the following notations:
\begin{itemize}
\item[(i)] We consider the real spherical harmonics $Y_{j,\ell}$ for $j\geq 0$ and $-j\leq \ell \leq j$ with the following normalization
\begin{align}\label{orthogonality}
\int_{\mathbb{S}^2}Y_{j,\ell}Y_{j',\ell'} & = 4\pi \de_{jj'}\de_{\ell\ell'}.
\end{align}
\item[(ii)] For $j\geq 1$ and $-(j-1)\leq \ell \leq j-1$, we define two families of functions
\begin{align*}
w_{j,\ell} & = r^{j-1} Y_{j-1,\ell},
\\ v_{j,\ell} & = \frac{(-1)^{\ell+1}  Y_{j-1,-\ell} }{4\pi(2j-1)r^{j}}.
\end{align*}
\item[(iii)] For $j\geq1$, $-(j-1)\leq \ell \leq j-1$ and $k=1,2,3$ we define the vector fields
\begin{align*}
W_{j,\ell,k} & = w_{j,\ell}\dr_k,
\\ V_{j,\ell,k} & = v_{j,\ell}\dr_k.
\end{align*}
For $a,b=1,2,3$, we also define
\begin{align*}
\Om_{ab}^\pm & = x^a \dr_b \pm x^b \dr_a.
\end{align*}
\item[(iv)] We also fix a smooth cutoff function $\chi:\RRR_+\longrightarrow[0,1]$ such that $\chi(r)=0$ for $r\leq 1$ and $\chi(r)=1$ for $r\geq 2$.
\end{itemize}
\end{definition}

The following theorem from \cite{McOwen1979} is our main tool to invert the flat Laplacian in weighted Sobolev spaces on $\RRR^3$.

\begin{theorem}\label{theo macowen}
Let $p\in\mathbb{N}^*$ and $-p-1 < \rho < -p$. Then $\De:H^2_{\rho}\longrightarrow L^2_{\rho-2}$ is an injection with closed range equal to
\begin{align*}
L^2_{\rho-2}\cap\pth{ \bigcup_{i=0}^{p-1}\mathfrak{H}^i}^{\perp_{L^2}},
\end{align*}
where $\mathfrak{H}^i$ is the set of harmonic polynomials defined on $\RRR^3$ which are homogeneous of degree $i$. Moreover we have $\l u \r_{H^2_\rho}\lesssim \l \De u \r_{L^2_{\rho-2}}$.
\end{theorem}

\begin{remark}\label{remark iso}
Also from \cite{McOwen1979} we have: if $-1<\rho<0$, then $\De:H^2_{\rho}\longrightarrow L^2_{\rho-2}$ is an isomorphism and moreover $\l u \r_{H^2_\rho}\lesssim \l \De u \r_{L^2_{\rho-2}}$. 
\end{remark}

We deduce the following proposition, its proof being postponed to Appendix \ref{appendix prop laplace}.

\begin{proposition}\label{prop laplace}
Let $q\in\mathbb{N}^*$ and $0 < \de < 1$. Then we have the following:
\begin{itemize}
\item[(i)] If $f$ is a scalar function satisfying $f\in L^2_{-q-\de-2}$, then there exists a unique $\tilde{u}\in H^2_{-q-\de}$ such that $u$ defined by
\begin{align*}
u=  \chi(r) \sum_{j=1}^q \sum_{\ell=-(j-1)}^{j-1} \langle f , w_{j,\ell} \rangle  v_{j,\ell}  + \tilde{u} 
\end{align*}
is a solution to $\De u =f$. Moreover, we have $\l \tilde{u} \r_{H^2_{-q-\de}}\lesssim \l f \r_{L^2_{-q-\de-2}}$. 
\item[(ii)] If $Y$ is a vector field satisfying $Y\in L^2_{-q-\de-2}$, then there exists a unique vector field $\tilde{X}\in H^2_{-q-\de}$ such that $X$ defined by
\begin{align*}
X & =  \chi(r)  \sum_{k=1,2,3} \langle Y , W_{1,0,k} \rangle  V_{1,0,k} 
\\&\quad + \frac{\sqrt{3}}{2} \chi(r)  \Big( \langle Y ,  \Om^-_{12} \rangle \pth{- V_{2,-1,1} + V_{2,1,2}} +  \langle Y , \Om^-_{13} \rangle\pth{ - V_{2,0,1} + V_{2,1,3}} 
\\&\quad\hspace{2cm} + \langle Y , \Om^-_{23} \rangle \pth{ - V_{2,0,2} + V_{2,-1,3}  } +  \langle Y , \Om^+_{23}  \rangle \pth{ V_{2,0,2} +  V_{2,-1,3} }
\\&\quad\hspace{2cm} + \langle Y , \Om^+_{12}  \rangle \pth{ V_{2,-1,1}  +  V_{2,1,2}} + \langle Y , \Om^+_{13}  \rangle \pth{ V_{2,0,1}  +  V_{2,1,3} } 
\\&\quad\hspace{2cm} + \langle Y , \Om^+_{11} \rangle V_{2,1,1} + \langle Y , \Om^+_{22} \rangle V_{2,-1,2}  + \langle Y , \Om^+_{33} \rangle V_{2,0,3} \Big)
\\&\quad + \chi(r) \sum_{j= 3}^q \sum_{\ell=-(j-1)}^{j-1}\sum_{k=1,2,3} \langle Y , W_{j,\ell,k} \rangle  V_{j,\ell,k}  + \tilde{X} 
\end{align*}
is a solution to $\De X =Y$. Moreover, we have $\l \tilde{X} \r_{H^2_{-q-\de}}\lesssim \l Y \r_{L^2_{-q-\de-2}}$. 
\end{itemize}
\end{proposition}

\begin{remark}
In the previous proposition, we performed a change of basis for the space of vector fields with components given by harmonic polynomials homogeneous of degree 1. Instead of the standard basis $(x^i \dr_j)_{1\leq i,j\leq 3}$, we prefered the basis
\begin{align*}
\pth{ \Om^+_{11} , \Om^+_{22} , \Om^+_{33} , \Om^+_{12} , \Om^+_{13} , \Om^+_{23} , \Om^-_{12} , \Om^-_{13}, \Om^-_{23} }.
\end{align*}
This is motivated by the properties of the linearized momentum constraint operator, see Proposition \ref{prop KIDS}.
\end{remark}

\begin{remark}
We can improve Proposition \ref{prop laplace}: if the scalar function $f$ or the vector field $Y$ are in $H^k_{-q-\de-2}$ for $k\in\mathbb{N}$, then $\tilde{u}$ or $\tilde{X}$ are actually in $H^{k+2}_{-q-\de}$ with the obvious bounds. This also applies for the isomorphism case mentioned in Remark \ref{remark iso}. This will only be used (without mention) in Appendix \ref{section proof lem reduction}  and \ref{section proof reduction}.
\end{remark}

\subsection{The constraint operator}\label{section constraint preli}

If $\bar{g}$ is a Riemannian metric and $\bar{\pi}$ a symmetric two-tensor, we denote by $D\HH[\bar{g},\bar{\pi}]$ and $D^2\HH[\bar{g},\bar{\pi}]$ (resp. $D\MM[\bar{g},\bar{\pi}]$ and $D^2\MM[\bar{g},\bar{\pi}]$) the first and second variation of the operator $\HH$ (resp. $\MM$) at $(\bar{g},\bar{\pi})$, and recall that $\Phi(\bar{g},\bar{\pi})=\pth{ \HH(\bar{g},\bar{\pi}),\MM(\bar{g},\bar{\pi})}$. Note that in this article, and in particular in Lemma \ref{lem Phi} below, we do not distinguish between coefficients of $\bar{g}$ and coefficients of its inverse $\bar{g}^{-1}$ and thus simply write $\bar{g}$ in the schematic notation for error terms. 

\begin{lemma}\label{lem Phi}
We have
\begin{align}
\Phi(\bar{g} + h ,\bar{\pi} + \pi) & = \Phi(\bar{g},\bar{\pi}) + D \Phi [\bar{g},\bar{\pi}](h,\pi) + \half  D^2\Phi [\bar{g},\bar{\pi}]((h,\pi),(h,\pi)) + \mathrm{Err}\Phi[\bar{g},\bar{\pi}]( h,\pi),
\end{align}
where schematically
\begin{align}
D\HH[\bar{g},\bar{\pi}](h,\pi) & = \bar{g}\nabla^2 h + h \nabla^2 \bar{g}  + \bar{g}^2 \nabla \bar{g} \nabla h + \bar{g} h (\nab \bar{g})^2  + \bar{g}^2\bar{\pi}\pi + \bar{g}\bar{\pi}^2h \label{DH bis},
\\ D\MM[\bar{g},\bar{\pi}](h,\pi) & =  \bar{g}\nab\pi + h\nab\bar{\pi} + \bar{g}^2\nab\bar{g}\pi + \bar{g}^2\bar{\pi}\nab h + \bar{g} \bar{\pi} h \nab\bar{g} \label{DM bis},
\\ D^2\HH [\bar{g},\bar{\pi}]((h,\pi),(h,\pi)) & = h\nab^2 h +\bar{g} h \nab\bar{g} \nab h + (\bar{g}\nab h)^2  + (h\nab\bar{g})^2  \label{D2H bis}
\\&\quad +  \bar{g}\bar{\pi}h\pi + (\bar{g}\pi)^2 + (\bar{\pi}h)^2,
\\ D^2\MM [\bar{g},\bar{\pi}]((h,\pi),(h,\pi)) & = \bar{g}^2 h\nab\pi + \bar{g} h \pi \nab\bar{g} + \bar{g}\bar{\pi} h \nab h  + h^2 \bar{\pi} \nab\bar{g} \label{D2M bis},
\end{align}
and
\begin{align}
\mathrm{Err}\HH[\bar{g},\bar{\pi}]( h,\pi) & = \bar{g} h(\nab h)^2 + h^2 \nab\bar{g}\nab h + (h\nab h)^2 + \bar{\pi}h^2\pi + (h\pi)^2\label{ErrH},
\\ \mathrm{Err}\MM[\bar{g},\bar{\pi}]( h,\pi) & =  \bar{g} h \pi\nab h + h^2\pi \nab\bar{g} + h^2\bar{\pi}\nab h + h^2 \pi\nab h\label{ErrM}.
\end{align}
\end{lemma}

\begin{proof}
Straightforward computations from the schematic expressions
\begin{align*}
\HH(\bar{g},\bar{\pi}) & =  \bar{g}\nab^2\bar{g} + (\bar{g}\nab\bar{g})^2 +  (\bar{g}\bar{\pi})^2,
\\ \MM(\bar{g},\bar{\pi}) & =  \bar{g}\pth{\nab\bar{\pi} + \bar{g} \bar{\pi}\nab\bar{g} }.
\end{align*}
See also \cite{Fischer1980} for the exact expressions of $D \Phi [\bar{g},\bar{\pi}]$ and $D^2 \Phi [\bar{g},\bar{\pi}]$.
\end{proof}

Of particular importance is the case of the Minkowski initial data set $(\bar{g},\bar{\pi})=(e,0)$, for which we give the exact expressions of both the first and second variations of the constraint operator.

\begin{lemma}
We have\footnote{Note that $\pi$ doesn't appear in $D \HH[e,0](h,\pi)$ and $h$ doesn't appear in $D \MM[e,0](h,\pi)$ so we simply denote them by $D \HH[e,0](h)$ and $D \MM[e,0](\pi)$ throughout the paper.}
\begin{align}
D \HH[e,0](h) & =-\De\tr h + \div\div h, \label{DH}
\\ \half D^2 \HH[e,0]((h,\pi),(h,\pi)) & =  -  2   h^{ij}\dr_i\div h_{j}   +  h^{ij} \De h_{ij} + h^{ij}  \dr_i \dr_j \tr h - |\div h|^2 \label{D2H}
\\&\quad + \div h\cdot \nabla \tr h   - \frac{1}{4} | \nabla \tr h|^2  + \frac{3}{4}  | \nabla h|^2 \nonumber
\\&\quad   - \frac{1}{2}    \dr^b h^{jk}  \dr_{j} h_{kb}  + \half (\tr \pi)^2 - |\pi|^2,\nonumber
\end{align}
and
\begin{align}
D \MM[e,0](\pi) & = \div \pi, \label{DM}
\\ \half D^2 \MM[e,0]((h,\pi),(h,\pi))_i & =  - h^{k\ell}\dr_k \pi_{\ell i} - \half  \pi^{k\ell}  \dr_i h_{k\ell}  -  \pi_{i}^k   \pth{ \div h_k - \half \dr_k \tr h}.\label{D2M}
\end{align}
\end{lemma}

\begin{proof}
See \cite{Fischer1980}.
\end{proof}

When solving the constraint equations around a given solution $(\bar{g},\bar{\pi})$, the linear obstructions come from elements of the cokernel of the linearized operators $D\HH[\bar{g},\bar{\pi}]$ and $D\MM[\bar{g},\bar{\pi}]$.

\begin{definition}
Let $f$ be a scalar function and $X$ a vector field. We say that $(f,X)$ is a KIDS for $(\bar{g},\bar{\pi})$ if $D\HH[\bar{g},\bar{\pi}]^*(f,X)=0$ and $D\MM[\bar{g},\bar{\pi}]^*(f,X)=0$.
\end{definition}
In the case of the Euclidean metric, we have
\begin{align*}
D\HH[e,0]^*(f) & = -(\De f)e + \Hess f,
\\ D\MM[e,0]^*(X) & = -\half \nab\otimes X,
\end{align*}
for $f$ a scalar function and $X$ a vector field. The following proposition, whose proof is postponed to Appendix \ref{appendix proof KIDS}, identifies in particular the KIDS of $(e,0)$.

\begin{proposition}\label{prop KIDS}
The operator $D\HH[e,0]^*$ satisfies the following properties.
\begin{itemize}
\item[(i)] The kernel of $D\HH[e,0]^*$ is generated by $w_{j,\ell}$ for $j=1,2$ and $-(j-1)\leq \ell \leq j-1$.
\item[(ii)] For $q\geq 3$, there exists a family of symmetric 2-tensors $(h_{j,\ell})_{3\leq j \leq q, -(j-1)\leq \ell \leq j-1}$ such that:
\begin{itemize}
\item the tensors $h_{j,\ell}$ are traceless, smooth and compactly supported in $\{1\leq r \leq 10\}$,
\item the matrix
\begin{align*}
\pth{ \left\langle h_{j,\ell}, D\HH[e,0]^*(w_{j',\ell'}) \right\rangle }_{3\leq j,j'\leq q, -(j-1)\leq \ell \leq j-1,-(j'-1)\leq \ell' \leq j'-1}
\end{align*}
is invertible,
\item the family of 1-form $\pth{\div h_{j,\ell}}_{3\leq j \leq q, -(j-1)\leq \ell \leq j-1}$ is linearly independent.
\end{itemize} 
\end{itemize}
The operator $D\MM[e,0]^*$ satisfies the following properties.
\begin{itemize}
\item[(i)] The kernel of $D\MM[e,0]^*$ is generated by $\left\{ W_{1,0,1},W_{1,0,2},W_{1,0,3}, \Om^-_{12}, \Om^-_{13}, \Om^-_{23}  \right\}$.
\item[(ii)] For $q\geq 2$, there exists a family of symmetric 2-tensors $(\varpi_Z)_{Z\in \mathcal{Z}_q}$ where 
\begin{align*}
\mathcal{Z}_2 \vcentcolon = \left\{ \Om^+_{11} , \Om^+_{22} , \Om^+_{33} , \Om^+_{12} , \Om^+_{13} , \Om^+_{23} \right\},
\end{align*}
and if $q\geq 3$
\begin{align*}
\mathcal{Z}_q \vcentcolon = \mathcal{Z}_2 \cup \left\{ W_{j,\ell,k}  \,\middle|\, 3\leq j\leq q, -(j-1)\leq \ell \leq j-1, k=1,2,3 \right\},
\end{align*}
such that:
\begin{itemize}
\item the tensors $\varpi_Z$ are smooth and compactly supported in $\{1\leq r \leq 10\}$,
\item the matrix
\begin{align*}
\pth{ \left\langle \varpi_Z, D\MM[e,0]^*(Z') \right\rangle }_{(Z,Z')\in\mathcal{Z}_q^2}
\end{align*}
is invertible,
\item the family of functions $\pth{(\De\tr + \div\div)\varpi_Z}_{Z\in \mathcal{Z}_q}$ is linearly independent.
\end{itemize}
\end{itemize}
\end{proposition}

\begin{definition}\label{def Aq Bq}
For $q\geq 3$, we define the space
\begin{align*}
\AA_q & \vcentcolon= \mathrm{Span}\Big( h_{j,\ell} , 3\leq j \leq q, -(j-1)\leq \ell \leq j-1\Big).
\end{align*}
For $q\geq 2$, we define the space
\begin{align*}
 \BB_q & \vcentcolon= \mathrm{Span}\pth{ \varpi_Z, Z\in\mathcal{Z}_q}.
\end{align*}
Note that $\dim \AA_q=q^2-4$ and $\dim \BB_q=3q^2-6$.
\end{definition}

Since elements of $\AA_q$ and $\BB_q$ are smooth and supported in $\{1\leq r \leq 10\}$, we will simply use the $W^{2,\infty}$ and $W^{1,\infty}$ topology on these spaces. This corresponds to the maximal degree of differentiation in the constraint equations for the metric and the reduced second fundamental form.

\begin{remark}
If one defines $\CC_q=\AA_q\times \BB_q$, we note that $\CC_q$ does indeed satisfy \eqref{complementary}. This will be crucially used in the proof of Proposition \ref{prop uniqueness}.
\end{remark}

\subsection{Kerr initial data sets}\label{section kerr family}

In this section, we construct and study the family of Kerr initial data sets, following Appendix F of \cite{Chrusciel2003} (see also \cite{Aretakis2023}). 

\subsubsection{Definition}\label{section Kerr}

Let $\vec{p}=(m,y,a,v)\in \RRR_+\times \RRR^3\times \RRR^3\times B_{\RRR^3}(0,1)$ and, following \cite{Chrusciel2011} and \cite{Mao2023}, consider the Kerr metric $\g_{m,|a|}$ of mass $m$ and angular momentum $|a|$ and rotating around the $x^3$-axis in Kerr-Schild coordinates $(t,x^1,x^2,x^3)$ given by
\begin{align*}
\g_{m,|a|} & = \m + \frac{2m\tilde{r}^3}{\tilde{r}^4+|a|^2(x^3)^2}\k\otimes \k,
\end{align*}
where $\m$ is the Minkowski metric, the 1-form $\k$ is given in coordinates by
\begin{align*}
\pth{ 1 , \frac{\tilde{r}x^1+|a|x^2}{\tilde{r}^2+|a|^2},\frac{\tilde{r}x^2-|a|x^1}{\tilde{r}^2+|a|^2}, \frac{x^3}{\tilde{r}} },
\end{align*}
and the function $\tilde{r}$ is implicitly defined by 
\begin{align*}
\frac{(x^1)^2 + (x^2)^2}{\tilde{r}^2+|a|^2} + \frac{(x^3)^2}{\tilde{r}^2} = 1.
\end{align*}

\begin{remark}
The Kerr-Schild coordinates are better to read off the distance between Kerr and Minkowski than the more usual Boyer-Lindquist coordinates. Indeed, in Boyer-Lindquist coordinates, the limit of $\g_{m,|a|}$ when $m$ tends to 0 is
\begin{align*}
- \d t^2 + \frac{r^2+a^2\cos^2\th}{r^2+a^2}\d r^2 + (r^2 +a^2\cos^2\th)\d\th^2 + (r^2+a^2)\sin^2\th\d\phi^2.
\end{align*}
This is in fact the Minkowski metric in the so-called "oblate spheroidal" coordinates, see \cite{Visser2009} for more details.
\end{remark}

We define a new coordinates system by setting 
\begin{align}\label{change of coord}
(x')^\mu & = \,^{(v)}\La^\mu_{\;\;\nu} \pth{ R^\nu_{\;\;\si} x^\si + y^\si }.
\end{align}
where $R$ is the rotation matrix such that $R\dr_3=\frac{a}{|a|}$ (if $a=0$ then $R$ is the identity matrix), and $\,^{(v)}\La$ is a Lorentz boost of relative velocity $|v|$ in the direction $\frac{v}{|v|}$ (see (8.2) in \cite{Aretakis2023}). 

\begin{definition}
We define $(g_{\vec{p}}, \pi_{\vec{p}})$ as the induced metric and reduced second fundamental form of the hypersurface $\{(x')^0=0\}$.
\end{definition}

For later use, we define the Lorentz factor $\ga\vcentcolon = \frac{1}{\sqrt{1-|v|^2}}$.

\subsubsection{Estimates for truncated black holes}

\begin{definition}\label{def chi p}
For $\vec{p}=(m,y,a,v)\in\RRR_+\times \RRR^3\times \RRR^3 \times B_{\RRR^3}(0,1)$, we define 
\begin{align*}
\chi_{\vec{p}}(r)=\chi\pth{\frac{r}{\la}},
\end{align*}
with
\begin{align*}
\la\vcentcolon=\max(2|y|,2|a|,3),
\end{align*}
where $\chi$ is defined in Definition \ref{def tool laplace}.
\end{definition}

The following proposition gathers important properties of the truncated initial data sets 
\begin{align*}
\pth{\chi_{\vec{p}}(g_{\vec{p}} - e),\chi_{\vec{p}} \pi_{\vec{p}}}.
\end{align*}
Its proof is postponed to Appendix \ref{appendix prop kerr}.

\begin{proposition}\label{prop kerr}
Let $\vec{p},\vec{p}\,'\in \RRR_+\times \RRR^3\times \RRR^3\times B_{\RRR^3}(0,1)$ with $m,m'\leq 1$. Then:
\begin{itemize}
\item[(i)] The following estimates hold
\begin{align}\label{estim kerr}
r^n \left| \nab^n\pth{ \chi_{\vec{p}}(g_{\vec{p}} - e) } \right| + r^{n+1}\left| \nab^n \pth{ \chi_{\vec{p}} \pi_{\vec{p}}} \right| & \lesssim_{\ga,n} \frac{m}{r},
\end{align}
and
\begin{align}
&r^n \left| \nab^n\pth{ \chi_{\vec{p}}(g_{\vec{p}} - e) - \chi_{\vec{p}\,'}(g_{\vec{p}\,'} - e) } \right| + r^{n+1}\left| \nab^n \pth{ \chi_{\vec{p}} \pi_{\vec{p}} - \chi_{\vec{p}\,'} \pi_{\vec{p}\,'}} \right| \label{diff kerr}
\\&\hspace{2cm} \lesssim_{\max(\ga,\ga'),n} \frac{|m-m'|}{r} + \max(m,m')\pth{\frac{|y-y'|+|a-a'|}{r^2}  + \frac{ | v-v'|}{r}}\nonumber.
\end{align}
\item[(ii)] There exists functions $\RR_{\HH,1}(\vec{p})$, $\RR_{\MM,W_{1,0,k}}(\vec{p})$, $\RR_{\HH,x^k}(\vec{p})$ and $\RR_{\MM,\Om^-_{ab}}(\vec{p})$ satisfying
\begin{align}
\left| \RR_{\HH,1}(\vec{p}) \right| + \left| \RR_{\MM,W_{1,0,k}}(\vec{p}) \right| + \left| \RR_{\HH,x^k}(\vec{p})\right| + \left| \RR_{\MM,\Om^-_{ab}}(\vec{p}) \right| & \lesssim_\ga m^2 \label{estim R bh bis},
\end{align}
and
\begin{align}
&\left| \RR_{\HH,1}(\vec{p}) - \RR_{\HH,1}(\vec{p}\,') \right| + \left| \RR_{\MM,W_{1,0,k}}(\vec{p}) - \RR_{\MM,W_{1,0,k}}(\vec{p}\,') \right| \non
\\& + \left| \RR_{\HH,x^k}(\vec{p}) - \RR_{\HH,x^k}(\vec{p}\,') \right| + \left| \RR_{\MM,\Om^-_{ab}}(\vec{p}) - \RR_{\MM,\Om^-_{ab}}(\vec{p}\,') \right|  \label{diff R bh bis}
\\&\hspace{1cm}\lesssim_{\max(\ga,\ga')} \max(m,m')\pth{ |m-m'| + \max(m,m')\pth{ |y-y'| + |a-a'| + |v-v'| }},\non
\end{align}
and such that the following integral identities hold
\begin{align}
\left\langle \HH\pth{ e + \chi_{\vec{p}}(g_{\vec{p}} - e),\chi_{\vec{p}} \pi_{\vec{p}}  }, 1 \right\rangle & = 16\pi \ga m + \RR_{\HH,1}(\vec{p}) \label{int kerr H 1}
\\  \left\langle \MM\pth{ e + \chi_{\vec{p}}(g_{\vec{p}} - e),\chi_{\vec{p}} \pi_{\vec{p}}  }, W_{1,0,k} \right\rangle & = - 8\pi \ga m v^k + \RR_{\MM,W_{1,0,k}}(\vec{p}) \label{int kerr M 1}
\\ \left\langle \HH\pth{ e + \chi_{\vec{p}}(g_{\vec{p}} - e),\chi_{\vec{p}} \pi_{\vec{p}}  }, x^k \right\rangle & = my^k + \RR_{\HH,x^k}(\vec{p}) \label{int kerr H x}
\\ \left\langle \MM\pth{ e + \chi_{\vec{p}}(g_{\vec{p}} - e),\chi_{\vec{p}} \pi_{\vec{p}}  }, \Om^-_{ij} \right\rangle & = 8\pi m a^k + \RR_{\MM,\Om^-_{ij}}(\vec{p}) \label{int kerr M x}
\end{align}
where in \eqref{int kerr M x} $k$ is such that $\in_{kij}=1$. 
\item[(iii)] We have
\begin{align}\label{constraint kerr}
\left| \Phi\pth{e + \chi_{\vec{p}}(g_{\vec{p}} - e), \chi_{\vec{p}}\pi_{\vec{p}} } \right| & \lesssim_\ga \frac{m}{r^3} \mathbbm{1}_{\{\la\leq r \leq 2\la\} },  
\end{align}
and
\begin{align}\label{diff constraint kerr}
&\left| \Phi\pth{e + \chi_{\vec{p}}(g_{\vec{p}} - e), \chi_{\vec{p}}\pi_{\vec{p}} } -  \Phi\pth{e + \chi_{\vec{p}\,'}(g_{\vec{p}\,'} - e), \chi_{\vec{p}\,'}\pi_{\vec{p}\,'} } \right| 
\\&\quad \lesssim_{\max(\ga,\ga')} \pth{  \frac{|m-m'|}{r^3} + \max(m,m')\pth{ \frac{|y-y'|+|a-a'|}{r^4} + \frac{|v-v'|}{r^3}}  } \nonumber
\\&\hspace{3cm} \times \mathbbm{1}_{\{\min(\la,\la')\leq r \leq 2 \max(\la,\la')\}}.\nonumber
\end{align}
\end{itemize}
\end{proposition}

\begin{remark}
Thanks to the cutoff functions $\chi_{\vec{p}}$ vanishing if $r\leq 3$, the $r^{-1}$ weights in the estimates of Proposition \ref{prop kerr} are equivalent to $(1+r)^{-1}$ weights.
\end{remark}

\section{Main results}\label{section main results}

In this section, we present the main result of this article, i.e. Theorem \ref{maintheorem}. Before that, we present first our conformal ansatz and then discuss the class of seed tensors we will consider in Theorem \ref{maintheorem}.

\subsection{A simplified conformal ansatz}\label{section ansatz}

Near Minkowski spacetime, we will consider the following conformal ansatz for $(g,\pi)$:
\begin{align}
(g,\pi) & = \pth{u^4 \bar{g}, \bar{\pi} + L_e X},\label{SCA}
\end{align}
for $(\bar{g},\bar{\pi})$ the parameter of the conformal formulation, $u$ a scalar function and $X$ a vector field. For the simplified ansatz \eqref{SCA}, the constraint equations rewrite
\begin{align*}
8\De_{\bar{g}} u  & =  u R(\bar{g})    + \half u^{-3}(\tr_{\bar{g}}\pth{\bar{\pi} + L_e X})^2 - u^{-3} | \bar{\pi} + L_e X |^2_{\bar{g}},
\\ \div_{\bar{g}}(\bar{\pi} + L_e X)_i & =  2u^{-1} \pth{  \dr_{i} u \tr_{\bar{g}}(\bar{\pi} + L_e X) - \bar{g}^{jk} \dr_k u  (\bar{\pi} + L_e X)_{ij}}.
\end{align*}
If $\bar{g}$ is assumed to be close to the Euclidean metric and $\bar{\pi}$ close to 0, we hope to show that $\wc{u}=u-1$ and $\wc{X}=X$ are small and we rewrite these equations as
\begin{align*}
8\De \wc{u}  & = \HH(\bar{g},\bar{\pi}) + \RR_\HH\pth{\bar{g},\bar{\pi},\wc{u},\wc{X}} ,
\\ \De \wc{X}_i & =  -\MM(\bar{g},\bar{\pi})_i  + \RR_\MM\pth{\bar{g},\bar{\pi},\wc{u},\wc{X}}_i ,
\end{align*}
where we defined the following remainders 
\begin{align*}
\RR_\HH\pth{\bar{g},\bar{\pi},\wc{u},\wc{X}} & = (1+\wc{u})^{-3} \pth{  \half \pth{\tr_{\bar{g}} L_e \wc{X}}^2 + \tr_{\bar{g}}\bar{\pi}\tr_{\bar{g}}L_eX - \left|  L_e \wc{X} \right|^2_{\bar{g}} - 2\pth{\bar{\pi}\cdot L_e\wc{X}}_{\bar{g}}  }
\\&\quad  + \pth{(1+\wc{u})^{-3} - 1}\pth{ \half \pth{\tr_{\bar{g}}\bar{\pi}}^2  - \left| \bar{\pi}  \right|^2_{\bar{g}} } + \wc{u} R(\bar{g}) - 8\pth{ \De_{\bar{g}} -  \De}\wc{u},
\\ \RR_\MM\pth{\bar{g},\bar{\pi},\wc{u},\wc{X}}_i & = 2(1+\wc{u})^{-1} \pth{  \dr_{i} \wc{u} \tr_{\bar{g}}\pth{\bar{\pi} + L_e \wc{X}} - \bar{g}^{jk} \dr_k \wc{u}  \pth{\bar{\pi} + L_e \wc{X}}_{ij}}
\\&\quad -\pth{ \div_{\bar{g}}\pth{L_e\wc{X}}_i - \div_e\pth{L_e\wc{X}}_i}.
\end{align*}

\subsection{Parametrization by TT tensors}\label{section param TT}

Here we show how, without loss of generality, we can reduce the space where the parameters $(\bar{g},\bar{\pi})$ of the simplified ansatz \eqref{SCA} live. In the next lemma (whose proof is postponed to Appendix \ref{section proof lem reduction}), we first decompose every symmetric 2-tensor with enough decay into the sum of a TT tensor, a multiple of $e$ and a Lie derivative, and give precise decay bounds.

\begin{lemma}\label{lem decomposition}
Let $k\geq 2$, $q=1,2$ and $0<\de<1$, and let a symmetric 2-tensor $h$ in $H^k_{-q-\de}$. Then there exists a unique scalar function $ \u[h]$ and a unique vector field $\X[h]$ such that
\begin{align}\label{decomposition}
h_{TT} \vcentcolon = h - \u[h]e - \nab\otimes \X[h] 
\end{align}
is a TT tensor. Moreover we have the bounds 
\begin{align}\label{estim decomp}
\l \u[h]\r_{H^k_{-q-\de}}+\l \X[h]\r_{H^{k+1}_{-q-\de+1}}+\l h_{TT} \r_{H^k_{-q-\de}}\lesssim \l h\r_{H^k_{-q-\de}}.
\end{align}
\end{lemma}

\begin{remark}
The decomposition \eqref{decomposition} is not orthogonal. It can be deduced from York's classical orthogonal decomposition in \cite{York1973}, this one being orthogonal.
\end{remark}

Building on Lemma \ref{lem decomposition}, the next proposition shows how one can benefit from the diffeomorphism invariance to simplify the parameters of the conformal ansatz. Its proof is postponed to Appendix \ref{section proof reduction}.

\begin{proposition}\label{prop reduction}
Let $0<\de<1$ and $\g$ a Lorentzian metric on $\RR=\{(t,x)\in \RRR\times \RRR^3, \; |t| < 1\}$ such that
\begin{align}\label{spacetime assumption}
\sup_{|t|<1}\pth{ \l \g - \m \r_{H^3_{-1-\de}(\Si_t)} + \l \dr_t \g  \r_{H^2_{-2-\de}(\Si_t)} } \leq \e,
\end{align}
where $\Si_t=\{t\}\times \RRR^3$. There exists $\e_0=\e_0(\de)>0$ such that if $0<\e\leq \e_0$, then there exists a diffeomorphism $\psi$ defined on $\left\{(t,x)\in \RRR\times \RRR^3, \; |t| < \half \right\} $, a scalar function $v$ and a vector field $Y$ defined on $\Si_0$ satisfying
\begin{align*}
\l v-1 \r_{H^3_{-1-\de}(\Si_0)} + \sup_{|t|<\half}\l \psi - \mathrm{Id}\r_{H^4_{- \de}(\Si_t)}  + \l Y \r_{H^3_{-1-\de}(\Si_0)} \lesssim \e ,
\end{align*} 
and such that the first and reduced second fundamental form of $\Si_0$ in the spacetime 
\begin{align*}
\pth{\left\{(t,x)\in \RRR\times \RRR^3, \; |t| < \half\right\},\psi^*\g}
\end{align*}
are given by
\begin{align*}
(\psi^*\g)_{| \Si_0} & = v^4 (e + \wc{g}), 
\\ \pi^{(\psi^*\g,\Si_0)} & =  L_{e}Y + \wc{\pi},
\end{align*}
with $\wc{g}$, $\wc{\pi}$ TT tensors.
\end{proposition}

This proposition has the following consequence: when solving the constraint equations near Minkowski spacetime, we can assume without loss of generality that the perturbations $(\wc{g},\wc{\pi})$ of the induced metric and of the reduced second fundamental form are TT tensors. Indeed, consider a perturbation of Minkowski with enough decay and regularity (as in \eqref{spacetime assumption}). From the point of view of stability problems (for instance), the choice of the initial slice is irrelevant and Proposition \ref{prop reduction} tells us that the initial data on $\RRR^3$ can be assumed to be of the form
\begin{align*}
\bar{g} & = v^4 (e + \wc{g}),
\\ \bar{\pi} & = L_e Y + \wc{\pi},
\end{align*}
with $(\wc{g},\wc{\pi})$ TT tensors. If we now want to solve the constraint equations on $\RRR^3$, we perturb $(\bar{g},\bar{\pi})$ and look for a solution $(g,\pi)$ under the form \eqref{SCA}, i.e. look for $u$ and $X$ such that
\begin{align*}
g & = (uv)^4 (e +\wc{g}),
\\ \pi & = L_e(X+Y) + \wc{\pi},
\end{align*}
solve the constraint equations. Now instead of solving for $(u,X)$, we can solve for $(uv,X+Y)$. This shows that if one's goal is to construct every decaying solution of the constraint equations in a neighborhood of Minkowski spacetime, then there is no loss of generality in assuming that $\bar{g}-e$ and $\bar{\pi}$ are TT tensors.

\begin{remark}
Note that this parametrization by TT tensors is consistent with the number of degrees of freedom of the constraint equations. Indeed, 4 degrees of freedom come from the four equations, 4 come from the two TT tensors $\wc{g}$ and $\wc{\pi}$, and 4 come from the diffeomorphism invariance, i.e. the choice of the initial slice. This adds up to 12 degrees of freedom, matching the 12 degrees of freedom of two symmetric 2-tensors.
\end{remark}

\begin{remark}
In order to obtain decaying solutions to the constraint equations, we will parametrize our set of solutions by decaying TT tensors, i.e. TT tensors in $H^2_{-q-\de}$ for $q\in \mathbb{N}^*$ and $0<\de<1$. Note they can be produced by the operator $P:H^5_{-q-\de+3}\longrightarrow H^2_{-q-\de}$ acting on symmetric 2-tensors defined by
\begin{align*}
P = \pth{ \De - \nab\otimes \div } \curl,
\end{align*}
where the $\curl$ of a symmetric 2-tensor is defined in Section 4.4 of \cite{Christodoulou1993}. Indeed one can show that if $h\in H^5_{-q-\de+3}$ then $P(h)$ is TT.
\end{remark}

\subsection{Restriction to a cone}\label{section restriction to a cone}

As explained in the introduction, the assumption that $(\wc{g},\wc{\pi})$ are TT will imply a coercive estimate on the mass $m$ of the black hole, i.e. for non-trivial TT tensors $(\wc{g},\wc{\pi})$ we will have $m>0$. However, the coupling between the equation for the mass $m$ and the equation for the boost velocity $v$ will imply that
\begin{align}\label{eq m modele}
16\pi m & = \sqrt{1-J(\wc{g},\wc{\pi})^2}\pth{ \frac{1}{4} \l \nabla \wc{g} \r_{L^2}^2 + \l \wc{\pi} \r_{L^2}^2 } + \text{lower order terms},
\end{align}
where the functional $J$ is given in the following definition.

\begin{definition}\label{def J}
For $(h,\varpi)$ two symmetric 2-tensors such that 
\begin{align*}
0< \frac{1}{4} \l \nabla h \r_{L^2}^2 + \l \varpi \r_{L^2}^2 < \infty,
\end{align*}
we define the following functional
\begin{align*}
J(h,\varpi) & = \frac{\sqrt{ \sum_{k=1,2,3}\pth{\int_{\RRR^3}\varpi^{ij}\dr_k h_{ij} }^2}}{\frac{1}{4} \l \nabla h \r_{L^2}^2 + \l \varpi \r_{L^2}^2 }.
\end{align*}
\end{definition}

The following proposition gathers important properties of the functional $J$. Its proof is postponed to Appendix \ref{section proof prop F}.

\begin{proposition}\label{prop F}
The functional $J$ satisfies the following properties.
\begin{itemize}
\item[(i)] If $(h,\varpi)\in \dot{H}^1\times L^2\setminus \{(0,0)\}$, then $J(h,\varpi)<1$, where the $\dot{H}^1$ norm is defined by $\l h\r_{\dot{H}^1}\vcentcolon =\l \nab h \r_{L^2}$.
\item[(ii)] There exists a sequence $(\wc{g}_n,\wc{\pi}_n)_{n\in\mathbb{N}}$ of TT tensors in $\dot{H}^1\times L^2\setminus \{(0,0)\}$ such that 
\begin{align*}
\lim_{n\to +\infty}J(\wc{g}_n,\wc{\pi}_n)=1.
\end{align*}
\end{itemize}
\end{proposition}

The first part of Proposition \ref{prop F} is very much reminiscent of the fact that the ADM momentum vector cannot be null unless it vanishes, see \cite{Ashtekar1982} and \cite{Beig1996}. As in these references, the equality case here is discarded for decay related reasons, see Appendix \ref{section proof prop F}. 

The second part of Proposition \ref{prop F} shows that, even when restricted to the space of TT tensors, the functional $J$ is not uniformly far from 1. This implies that the lower bound on the mass obtained from \eqref{eq m modele} cannot be uniform over the space of TT tensors and might degenerate if $J(\wc{g},\wc{\pi})$ gets close to 1. 

Therefore, we have no choice but to perform our construction over a set of TT tensors where $J(\wc{g},\wc{\pi})$ is assumed uniformly far from 1. More precisely, we fix $0<\a \leq 1$ and define the following set
\begin{align*}
V_\a & \vcentcolon= \enstq{(\wc{g},\wc{\pi})\in \dot{H}^1\times L^2\setminus \{(0,0)\}}{ J(\wc{g},\wc{\pi})\leq 1-\a}.
\end{align*}
Our construction will require the smallness constant $\e$ measuring the size of $\wc{g}$ and $\wc{\pi}$ to be small compared to $\a$. This is always possible since $V_\a$ has the structure of a cone and therefore by rescaling any elements of $V_\a$, one can reach any smallness with $\a$ still being fixed.

\subsection{Statement of the main results}\label{section main theo}

We are ready to state the precise version of our main result. We recall that the cutoff $\chi_{\vec{p}}$ has been introduced in Definition \ref{def chi p}.

\begin{theorem}[Main result, version 2]\label{maintheorem}
Let $q\in \mathbb{N}^*$, $0<\de<1$ and $0<\e\ll \a \leq 1$. Let $(\wc{g},\wc{\pi})$ be two TT tensors such that
\begin{align*}
(\wc{g},\wc{\pi}) \in V_\a \cap B_{H^4_{-q-\de}\times H^3_{-q-\de-1}}(0,\e),
\end{align*} 
and
\begin{align*}
\eta^2 \vcentcolon = \frac{1}{4} \l \nabla \wc{g} \r_{L^2}^2 + \l \wc{\pi} \r_{L^2}^2 >0.
\end{align*}
Then, there exists $\e_0=\e_0(q,\de,\a)>0$ such that for every $0<\e\leq \e_0$, the following holds.
\begin{itemize}
\item[(i)] If $q=1$, then there exists a solution $(g,\pi)$ of the constraint equations on $\RRR^3$ of the form
\begin{align*}
g & = u^4 \pth{ e + \chi_{\vec{p}} (g_{\vec{p}} - e )  + \wc{g}   },
\\ \pi & =  \chi_{\vec{p}} \pi_{\vec{p}}  + \wc{\pi} + L_{e }X ,
\end{align*}
where
where the black hole parameter $\vec{p}=(m,0,0,v)\in\RRR_+\times \RRR^3\times \RRR^3 \times B_{\RRR^3}(0,1)$ satisfies
\begin{align*}
\eta^2 \lesssim m \lesssim \eta^2, \qquad  |v| \leq 1-\frac{\a}{2} ,
\end{align*}
and $u$ and $X$ satisfy
\begin{align*}
\l u-1 \r_{H^2_{-1-\de}} + \l X \r_{H^2_{-1-\de}} \lesssim \eta\e .
\end{align*}
\item[(ii)] If $q=2$, then there exists a solution $(g,\pi)$ of the constraint equations on $\RRR^3$ of the form
\begin{align*}
g & = u^4 \pth{ e + \chi_{\vec{p}} (g_{\vec{p}} - e )  + \wc{g}   },
\\ \pi & =  \chi_{\vec{p}} \pi_{\vec{p}}  + \wc{\pi} + \breve{\pi} + L_{e }X ,
\end{align*}
where the black hole parameter $\vec{p}=(m,y,a,v)\in\RRR_+\times \RRR^3\times \RRR^3 \times B_{\RRR^3}(0,1)$ satisfies
\begin{align*}
\eta^2 \lesssim m \lesssim \eta^2, \qquad |a| + |y| \lesssim \pth{ \frac{\e}{\eta}}^{\frac{1}{\frac{3}{2}+\de}}, \qquad |v| \leq 1-\frac{\a}{2} ,
\end{align*}
and $u$, $X$ and the corrector $\breve{\pi}\in \BB_2$ satisfy
\begin{align*}
\l u-1 \r_{H^2_{-2-\de}} + \l X \r_{H^2_{-2-\de}} + \l \breve{\pi} \r_{W^{1,\infty}} \lesssim \eta\e .
\end{align*}
\item[(iii)] If $q\geq 3$, then there exists a solution $(g,\pi)$ of the constraint equations on $\RRR^3$ of the form
\begin{align*}
g & = u^4 \pth{ e + \chi_{\vec{p}} (g_{\vec{p}} - e )  + \wc{g} + \breve{g}  },
\\ \pi & =  \chi_{\vec{p}} \pi_{\vec{p}}  + \wc{\pi} + \breve{\pi} + L_{e }X ,
\end{align*}
where the black hole parameter $\vec{p}=(m,y,a,v)\in\RRR_+\times \RRR^3\times \RRR^3 \times B_{\RRR^3}(0,1)$ satisfies
\begin{align*}
\eta^2 \lesssim m \lesssim \eta^2, \qquad |a| + |y| \lesssim \pth{ \frac{\e}{\eta}}^{\frac{1}{q+\de-\half}}, \qquad |v| \leq 1-\frac{\a}{2} ,
\end{align*}
and $u$, $X$ and the correctors $(\breve{g},\breve{\pi})\in \AA_q\times \BB_q$ satisfy 
\begin{align*}
\l u-1 \r_{H^2_{-q-\de}} + \l X \r_{H^2_{-q-\de}} +  \l \breve{g} \r_{W^{2,\infty}} + \l \breve{\pi} \r_{W^{1,\infty}} \lesssim \eta\e .
\end{align*}
\end{itemize}
\end{theorem}

We can adapt our construction to the time-symmetric case, which corresponds to solutions of the constraint equations with $\pi=0$. In this case, the constraint equations reduce to $R(g)=0$ and there is no linear obstructions coming from the momentum constraint operator. This allows us to drop the cone assumption and to consider only $g_{\vec{p}}$ with $\vec{p}=(m,y,0,0)$ (which implies $\pi_{\vec{p}}=0$), i.e. all the translated Schwarzschild data sets.

\begin{corollary}[Time-symmetric case, version 2]\label{coro TS}
Let $q\in \mathbb{N}^*$, $0<\de<1$. Let $\wc{g}$ be a TT tensor such that $\l \wc{g} \r_{H^4_{-q-\de}} \leq \e$ and $\eta \vcentcolon =  \l \nabla \wc{g} \r_{L^2} >0$. There exists $\e_0=\e_0(q,\de)>0$ such that for every $0<\e\leq \e_0$ the following holds.
\begin{itemize}
\item[(i)] If $q=1$, then there exists a solution $g$ of $R(g)=0$ on $\RRR^3$ of the form 
\begin{align*}
g & = u^4 \pth{ e + \chi_{\vec{p}} (g_{\vec{p}} - e )  + \wc{g}  },
\end{align*}
where the black hole parameter $\vec{p}=(m,0,0,0)\in\RRR_+\times \RRR^3\times \RRR^3 \times B_{\RRR^3}(0,1)$ satisfies
\begin{align*}
\eta^2 \lesssim m \lesssim \eta^2,  
\end{align*}
and $u$ satisfies $\l u-1 \r_{H^2_{-1-\de}}  \lesssim \eta\e $.
\item[(ii)] If $q=2$, then there exists a solution $g$ of $R(g)=0$ on $\RRR^3$ of the form 
\begin{align*}
g & = u^4 \pth{ e + \chi_{\vec{p}} (g_{\vec{p}} - e )  + \wc{g}  },
\end{align*}
where the black hole parameter $\vec{p}=(m,y,0,0)\in\RRR_+\times \RRR^3\times \RRR^3 \times B_{\RRR^3}(0,1)$ satisfies
\begin{align*}
\eta^2 \lesssim m \lesssim \eta^2,   \qquad |y| \lesssim \pth{ \frac{\e}{\eta}}^{\frac{1}{\frac{3}{2}+\de}},
\end{align*}
and $u$ satisfies $\l u-1 \r_{H^2_{-2-\de}}  \lesssim \eta\e $.
\item[(iii)] If $q\geq 3$, then there exists a solution $g$ of $R(g)=0$ on $\RRR^3$ of the form 
\begin{align*}
g & = u^4 \pth{ e + \chi_{\vec{p}} (g_{\vec{p}} - e )  + \wc{g} + \breve{g}  },
\end{align*}
where the black hole parameter $\vec{p}=(m,y,0,0)\in\RRR_+\times \RRR^3\times \RRR^3 \times B_{\RRR^3}(0,1)$ satisfies
\begin{align*}
\eta^2 \lesssim m \lesssim \eta^2, \qquad |y| \lesssim \pth{ \frac{\e}{\eta}}^{\frac{1}{q+\de-\half}}, 
\end{align*}
and $u$ and the corrector $\breve{g} \in \AA_q $ satisfy $\l u-1 \r_{H^2_{-q-\de}} + \l \breve{g} \r_{W^{2,\infty}} \lesssim \eta\e $.
\end{itemize}
\end{corollary}

\begin{remark}
Note that in the cases $q=1$ (both in Theorem \ref{maintheorem} and Corollary \ref{coro TS}), where there is no need for a nontrivial center of mass, the cutoff appearing in the ansatz for $g$ or $\pi$ is $\chi$ defined in Definition \ref{def tool laplace}, i.e. we do not need $\chi_{\vec{p}}$ for $q=1$.
\end{remark}

We also give a uniqueness statement in the case of the full constraint equations with $q\geq 3$ (simpler statements hold in the cases $q=1,2$). Its proof is postponed to Appendix \ref{appendix uniqueness}.

\begin{proposition}[Uniqueness]\label{prop uniqueness}
Let $q\geq 3$, $0<\de<1$ and let $(g,\pi)$ be a solution of the constraint equations on $\RRR^3$. Assume that there exists two sets
\begin{align}\label{two sets}
\pth{\vec{p}_i, \wc{g}_i,\wc{\pi}_i,\breve{g}_i,\breve{\pi}_i,u_i, X_i}
\end{align}
for $i=1,2$ with $\vec{p}_i\in \RRR_+\times \RRR^3\times \RRR^3 \times B_{\RRR^3}(0,1)$, $\wc{g}_i$ and $\wc{\pi}_i$ TT tensors in $H^2_{-q-\de}$, $(\breve{g}_i,\breve{\pi}_i)\in\mathcal{A}_q\times \mathcal{B}_q$ and $\pth{u_i-1,X_i}\in (H^2_{-q-\de})^2$ such that $(g,\pi)$ are given by
\begin{align}
g & = u_i^4 \pth{ e + \chi_{\vec{p}_i} (g_{\vec{p}_i} - e )  + \wc{g}_i + \breve{g}_i  },\label{uniq1}
\\ \pi & =  \chi_{\vec{p}_i} \pi_{\vec{p}_i}  + \wc{\pi}_i + \breve{\pi}_i + L_{e }X_i ,\label{uniq2}
\end{align}
Then the two sets \eqref{two sets} are equal.
\end{proposition}

The discussions of Sections \ref{section param TT} and \ref{section restriction to a cone}, together with the existence statement of Theorem \ref{maintheorem} and the uniqueness statement of Proposition \ref{prop uniqueness} show that we have constructed decaying solutions of the constraint equations in a neighborhood of Minkowski spacetime on $\RRR^3$ as a graph over the set of decaying TT tensors.

\section{Solving the constraint equations}\label{section main proof}

This section is devoted to the proof of Theorem \ref{maintheorem} in the case $q\geq 3$. The proofs of the cases $q=1,2$ are simpler and we will only point out the simplifications in Remarks \ref{remark q12 1}, \ref{remark q12 2} and \ref{remark q12 3}. A similar remark holds for the proof of Corollary \ref{coro TS}, i.e. for the time-symmetric constraint equations, but we won't mention the simplifications further.

\subsection{Conformal formulation}

We define the base metric and base reduced second fundamental form for the conformal formulation of the constraint equations.

\begin{definition}\label{def para} We introduce the following notations:
\begin{itemize}
\item[(i)] The space of parameters is defined as 
\begin{align*}
P=\RRR_+\times \RRR^3\times \RRR^3 \times B_{\RRR^3}(0,1)\times \mathcal{A}_q\times \mathcal{B}_q,
\end{align*}
where $\AA_q$ and $\BB_q$ are defined in Definition \ref{def Aq Bq}.
\item[(ii)] For a parameter $\p=(\vec{p},\breve{g},\breve{\pi})\in P$ we define the base metric $\bar{g}(\p)$ by
\begin{align*}
\bar{g}(\p) & = e + \chi_{\vec{p}}(g_{\vec{p}}-e) + \wc{g} + \breve{g},
\end{align*}
and the base reduced second fundamental form $\bar{\pi}(\p)$ by
\begin{align*}
\bar{\pi}(\p) & = \chi_{\vec{p}}\pi_{\vec{p}} + \wc{\pi} + \breve{\pi},
\end{align*}
where $g_{\vec{p}}$, $\pi_{\vec{p}}$ and $\chi_{\vec{p}}$ are defined in Section \ref{section Kerr}.
\item[(iii)] We define the following distance between two parameters $\p$ and $\p'$:
\begin{align}\label{def distance}
d(\p,\p') & \vcentcolon = |m-m'| + \eta^2\pth{ |y-y'| + |a-a'|  + |v-v'| }  
\\&\quad + \l \breve{g} - \breve{g}'\r_{W^{2,\infty}}  + \l \breve{\pi} - \breve{\pi}'\r_{W^{1,\infty}}. \non
\end{align}
It is easy to see that $(P,d)$ defines a complete metric space.
\end{itemize}
\end{definition}

\begin{remark}
Note that \eqref{diff kerr} implies the rough estimate
\begin{align}\label{diff kerr 2}
&r^n \left| \nab^n\pth{ \chi_{\vec{p}}(g_{\vec{p}} - e) - \chi_{\vec{p}\,'}(g_{\vec{p}\,'} - e) } \right| + r^{n+1}\left| \nab^n \pth{ \chi_{\vec{p}} \pi_{\vec{p}} - \chi_{\vec{p}\,'} \pi_{\vec{p}\,'}} \right| 
\\&\hspace{5cm} \lesssim_{\max(\ga,\ga')} \pth{ 1  + \eta^{-2}\max(m,m')}\frac{d(\p,\p')}{r}\nonumber.
\end{align}
The estimate \eqref{diff kerr 2} does not benefit from the fact that in \eqref{diff kerr}, $|y-y'|+|a-a'|$ is divided by $r^2$ and not $r$. This better weight will be crucial in the proofs of Lemmas \ref{lem param contraction} and \ref{lem u X contraction}. 
\end{remark}

The base metric and reduced second fundamental form $(\bar{g}(\p),\bar{\pi}(\p))$ serve as parameters of the conformal method, in the sense that we look for solutions of the constraint equations under the form
\begin{equation}\label{conformal ansatz}
\begin{aligned}
g & = (1+\wc{u})^4 \bar{g}(\p),
\\ \pi & = \bar{\pi}(\p) + L_{e}\wc{X},
\end{aligned}
\end{equation}
for $\wc{u}$ a scalar function and $\wc{X}$ a vector field. As explained in Section \ref{section ansatz}, the constraint equations for $(g,\pi)$ are equivalent to the following system of equations for $\wc{u}$ and $\wc{X}$:
\begin{equation}\label{eq u X}
\left\{
\begin{aligned}
8\De \wc{u}  & = \HH(\bar{g}(\p),\bar{\pi}(\p)) + \RR_\HH\pth{\p,\wc{u},\wc{X}} ,
\\ \De \wc{X}_i & =  -\MM(\bar{g}(\p),\bar{\pi}(\p))_i  + \RR_\MM\pth{\p,\wc{u},\wc{X}}_i ,
\end{aligned}
\right.
\end{equation}
with remainders 
\begin{align}
&\RR_\HH\pth{\p,\wc{u},\wc{X}} \label{remainders 1}
\\& = (1+\wc{u})^{-3} \pth{  \half \pth{\tr_{\bar{g}(\p)} L_e \wc{X}}^2 + \tr_{\bar{g}(\p)}\bar{\pi}(\p)\tr_{\bar{g}(\p)}L_eX - \left|  L_e \wc{X} \right|^2_{\bar{g}(\p)} - 2\pth{\bar{\pi}(\p)\cdot L_e\wc{X}}_{\bar{g}(\p)}  } \nonumber
\\&\quad  + \pth{(1+\wc{u})^{-3} - 1}\pth{ \half \pth{\tr_{\bar{g}(\p)}\bar{\pi}(\p)}^2  - \left| \bar{\pi}(\p)  \right|^2_{\bar{g}(\p)} } + \wc{u} R(\bar{g}(\p)) - 8\pth{ \De_{\bar{g}(\p)} -  \De}\wc{u}\nonumber,
\end{align}
and
\begin{align}
\RR_\MM\pth{\p,\wc{u},\wc{X}}_i & = 2(1+\wc{u})^{-1} \pth{  \dr_{i} \wc{u} \tr_{\bar{g}(\p)}\pth{\bar{\pi}(\p) + L_e \wc{X}} - \bar{g}(\p)^{jk} \dr_k \wc{u}  \pth{\bar{\pi}(\p) + L_e \wc{X}}_{ij}} \label{remainders 2}
\\&\quad -\pth{ \div_{\bar{g}(\p)}\pth{L_e\wc{X}}_i - \div_e\pth{L_e\wc{X}}_i}.\nonumber
\end{align}
The strategy to prove Theorem \ref{maintheorem} is then to solve \eqref{eq u X} with a fixed point argument while also choosing the parameter $\p$ so that the RHS of the equations are orthogonal to the first $q^2$ and $3q^2$ spherical harmonics respectively appearing in the first and second equation when inverting the Laplacian operator (see Proposition \ref{prop laplace}). For the sake of clarity, we define
\begin{align*}
A\pth{ \p , \wc{u},\wc{X} } & \vcentcolon = \HH(\bar{g}(\p),\bar{\pi}(\p)) + \RR_\HH\pth{\p,\wc{u},\wc{X}},
\\ B\pth{ \p , \wc{u},\wc{X} }_i & \vcentcolon =  -\MM(\bar{g}(\p),\bar{\pi}(\p))_i  + \RR_\MM\pth{\p,\wc{u},\wc{X}}_i .
\end{align*}
The choice of the parameter $\p$ is part of the fixed point argument since the RHS of \eqref{eq u X} depend on $\wc{u}$ and $\wc{X}$ themselves. The parameter will thus depend on $\pth{\wc{u},\wc{X}}$ and we denote it $\p\pth{\wc{u},\wc{X}}$. To ensure $r^{-q-\de}$ asymptotic $r$-decay, we will choose $\p$ so that
\begin{itemize}
\item the Hamiltonian orthogonality conditions hold
\begin{align}\label{ortho 1}
\left\langle A\pth{ \p\pth{\wc{u},\wc{X}} , \wc{u},\wc{X} }, w_{j,\ell} \right\rangle = 0,
\end{align}
for $1\leq j \leq q$ and $-(j-1)\leq \ell \leq j-1$,
\item the momentum orthogonality conditions hold
\begin{align}
\left\langle B\pth{ \p\pth{\wc{u},\wc{X}} , \wc{u},\wc{X} }, W_{1,0,k} \right\rangle & = 0,\label{ortho 2}
\\ \left\langle B\pth{ \p\pth{\wc{u},\wc{X}} , \wc{u},\wc{X} }, \Om^{\pm}_{ab} \right\rangle & = 0,\label{ortho 3}
\\ \left\langle B\pth{ \p\pth{\wc{u},\wc{X}} , \wc{u},\wc{X} }, W_{j,\ell,k} \right\rangle & = 0,\label{ortho 4}
\end{align}
for $3\leq j \leq q$, $-(j-1)\leq \ell \leq j-1$, $k=1,2,3$ and $a,b=1,2,3$.
\end{itemize}

\begin{remark}\label{remark q12 1}
If $q=1$, we only need to ensure the orthogonality conditions \eqref{ortho 1} for $j=1$ and \eqref{ortho 2}. If $q=2$, we only need to ensure the orthogonality conditions \eqref{ortho 1} for $j=1,2$, \eqref{ortho 2} and \eqref{ortho 3}.
\end{remark}

We give a brief outline of the proof of Theorem \ref{maintheorem}.
\begin{itemize}
\item In Section \ref{section expansion}, we expand $A\pth{ \p , \wc{u},\wc{X} }$ and $B\pth{ \p , \wc{u},\wc{X} }$ and identify main terms and error terms.
\item In Section \ref{section projection}, we use these expansions to compute the projections appearing in the orthogonality conditions above.
\item In Section \ref{section def param}, we construct a map $\pth{\wc{u},\wc{X}}\mapsto \p\pth{\wc{u},\wc{X}}\in P$ such that the orthogonality conditions hold.
\item In Section \ref{section conclusion}, we solve the constraint equations with the parameter $\p\pth{\wc{u},\wc{X}}$ previously constructed.
\end{itemize}

\subsection{Expansions of $A\pth{ \p , \wc{u},\wc{X} }$ and $B\pth{ \p , \wc{u},\wc{X} }$}\label{section expansion}

In this section, we give two propositions expanding first the main terms in $A\pth{ \p , \wc{u},\wc{X} }$ and $B\pth{ \p , \wc{u},\wc{X} }$, i.e. $\HH(\bar{g}(\p),\bar{\pi}(\p))$ and $\MM(\bar{g}(\p),\bar{\pi}(\p))$, and then the remainders $\RR_\HH\pth{\p,\wc{u},\wc{X}}$ and $\RR_\MM\pth{\p,\wc{u},\wc{X}}$.

\begin{proposition}\label{prop main terms}
If $\p\in P$ is such that 
\begin{align}
m + \l \breve{g} \r_{W^{2,\infty}} + \l \breve{\pi} \r_{W^{1,\infty}} \leq \half,\label{assumption preli}
\end{align}
then we have
\begin{align}
\HH(\bar{g}(\p),\bar{\pi}(\p)) & = \HH\pth{e + \chi_{\vec{p}}(g_{\vec{p}}-e) , \chi_{\vec{p}}\pi_{\vec{p}} }  + D\HH[ e ,0]\pth{\breve{g}}  \label{exp H}
\\&\quad + \half D^2\HH[ e ,0]\pth{\pth{ \wc{g}  , \wc{\pi}  } , \pth{ \wc{g} , \wc{\pi}   }}  +   \breve{g}^{ij} \De \wc{g}_{ij}  + \mathrm{Err}_1(\p), \nonumber 
\\ \MM(\bar{g}(\p),\bar{\pi}(\p))_i & = \MM\pth{e + \chi_{\vec{p}}(g_{\vec{p}}-e),\chi_{\vec{p}}\pi_{\vec{p}} }_i  +D\MM[e ,0]( \breve{\pi})_i \label{exp M}
\\&\quad +  \half D^2\MM[e , 0]((\wc{g},\wc{\pi}),(\wc{g},\wc{\pi}))_i - \breve{g}^{k\ell}\dr_k \wc{\pi}_{\ell i} + \mathrm{Err}_2(\p)_i \nonumber,
\end{align}
where the error terms satisfy
\begin{align}\label{estim error}
|\mathrm{Err}_i(\p)| & \lesssim_\ga    \frac{\e m}{(1+r)^{q+\de+3}} 
\\&\quad + \frac{\e}{(1+r)^{q+\de}}\pth{ |\nab\wc{g}|^2 + |\wc{\pi}|^2}  + \pth{\l \breve{g}\r_{W^{2,\infty}}^2 + \l \breve{\pi} \r_{W^{1,\infty}}^2}\mathbbm{1}_{\{1\leq r \leq 10\}} \nonumber
\\&\quad + \pth{ \l \breve{g} \r_{W^{2,\infty}} + \l \breve{\pi} \r_{W^{1,\infty}} } \pth{ |\wc{g}| + |\nab\wc{g}| + |\wc{\pi}| + m }  \mathbbm{1}_{\{1\leq r \leq 10\}}  .\nonumber
\end{align}
Moreover, there exists numbers $T_1^{abcd}$ and $S_2^{abc}$ for $a,b,c,d=1,2,3$ satisfying 
\begin{align}
\left| T_1^{abcd} \right| + \left| S_2^{abc} \right| \lesssim_{\max(\ga,\ga')}\pth{ 1  + \eta^{-2}\max(m,m')}d(\p,\p')\label{estim T S}
\end{align}
and such that
\begin{align}\label{diff error}
& \left| \mathrm{Err}_i(\p) - \mathrm{Err}_i(\p') - T_1^{abcd} (1+r)^{-1}\dr_a\dr_b\wc{g}_{cd} - S_2^{abc} (1+r)^{-1} \dr_a \wc{\pi}_{bc} \right|
\\&\quad \lesssim_{\max(\ga,\ga')} \pth{  (1+r)^{-2}\pth{|\nab\wc{g}| + |\wc{\pi}| }+ (1+r)^{-3}|\wc{g}| } \pth{ 1  + \eta^{-2}\max(m,m')}d(\p,\p')\nonumber
\\&\quad\quad + \max\pth{ \l \breve{g}\r_{W^{2,\infty}},\l \breve{g}'\r_{W^{2,\infty}},\l \breve{\pi}\r_{W^{1,\infty}},\l \breve{\pi}'\r_{W^{1,\infty}}} \pth{ 1  + \eta^{-2}\max(m,m')}d(\p,\p')  \mathbbm{1}_{\{1\leq r \leq 10\}}\nonumber
\\&\quad\quad + \pth{ \max(m,m')+  |\wc{g}| + |\nab\wc{g}| + |\wc{\pi}| }d(\p,\p')\mathbbm{1}_{\{1\leq r \leq 10\}}\nonumber,
\end{align}
for $\p,\p'\in P$ both satisfying \eqref{assumption preli} and $i=1,2$.
\end{proposition}

\begin{proof}
Note that \eqref{estim kerr} and the assumption $m\leq 1$ imply that
\begin{align}\label{estim kerr 2}
\left| e + \chi_{\vec{p}}(g_{\vec{p}}-e) \right| & \lesssim_\ga 1.
\end{align}
We start with the Hamiltonian constraint operator. We recall Definition \ref{def para} and expand:
\begin{align*}
\HH(\bar{g}(\p),\bar{\pi}(\p)) & = \HH(e+\chi_{\vec{p}}(g_{\vec{p}}-e) + \wc{g} + \breve{g},\chi_{\vec{p}}\pi_{\vec{p}} + \wc{\pi} + \breve{\pi})
\\&= \HH(e+\chi_{\vec{p}}(g_{\vec{p}}-e) ,\chi_{\vec{p}}\pi_{\vec{p}} ) + D\HH[e+\chi_{\vec{p}}(g_{\vec{p}}-e),\chi_{\vec{p}}\pi_{\vec{p}} ](\wc{g} + \breve{g},\wc{\pi}+\breve{\pi})
\\&\quad + \half D^2\HH[e+\chi_{\vec{p}}(g_{\vec{p}}-e),\chi_{\vec{p}}\pi_{\vec{p}} ]((\wc{g} + \breve{g}, \wc{\pi} + \breve{\pi}),(\wc{g} + \breve{g}, \wc{\pi} + \breve{\pi}))
\\&\quad + \mathrm{Err}\HH[e+\chi_{\vec{p}}(g_{\vec{p}}-e),\chi_{\vec{p}}\pi_{\vec{p}} ](\wc{g} + \breve{g}, \wc{\pi} + \breve{\pi})
\\& = \HH(e+\chi_{\vec{p}}(g_{\vec{p}}-e) ,\chi_{\vec{p}}\pi_{\vec{p}} ) + A_1(\p) + A_2(\p) + A_3(\p),
\end{align*}
where $D\HH$, $D^2\HH$ and $\mathrm{Err}\HH$ are as defined in Lemma \ref{lem Phi}, and
\begin{align*}
A_1(\p) & \vcentcolon= D\HH[e+\chi_{\vec{p}}(g_{\vec{p}}-e),\chi_{\vec{p}}\pi_{\vec{p}} ](\wc{g} + \breve{g},\wc{\pi}+\breve{\pi}),
\\ A_2(\p) & \vcentcolon= \half D^2\HH[e+\chi_{\vec{p}}(g_{\vec{p}}-e),\chi_{\vec{p}}\pi_{\vec{p}} ]((\wc{g} + \breve{g}, \wc{\pi} + \breve{\pi}),(\wc{g} + \breve{g}, \wc{\pi} + \breve{\pi})),
\\ A_3(\p) & \vcentcolon= \mathrm{Err}\HH[e+\chi_{\vec{p}}(g_{\vec{p}}-e),\chi_{\vec{p}}\pi_{\vec{p}} ](\wc{g} + \breve{g}, \wc{\pi} + \breve{\pi}).
\end{align*}
Thanks to \eqref{DH bis} we have schematically
\begin{align}
A_1(\p) - D\HH[ e ,0]\pth{\breve{g}}  & =  \chi_{\vec{p}}(g_{\vec{p}}-e)\nabla^2 (\wc{g} + \breve{g}) +  (\wc{g} + \breve{g}) \nabla^2(\chi_{\vec{p}}(g_{\vec{p}}-e)) \label{error A1}
\\&\quad + (e+\chi_{\vec{p}}(g_{\vec{p}}-e))^2 \pth{ \nabla (\chi_{\vec{p}}(g_{\vec{p}}-e)) \nabla (\wc{g} + \breve{g}) + \chi_{\vec{p}}\pi_{\vec{p}}(\wc{\pi}+\breve{\pi}) }\nonumber
\\&\quad + (e+\chi_{\vec{p}}(g_{\vec{p}}-e)) (\wc{g} + \breve{g}) \pth{ (\nab (\chi_{\vec{p}}(g_{\vec{p}}-e)))^2 + (\chi_{\vec{p}}\pi_{\vec{p}})^2 }\nonumber
\end{align}
where we have used that $D\HH[ e ,0]\pth{ \wc{g}}=0$ since $\wc{g}$ is TT (recall \eqref{DH}). Thanks to \eqref{D2H bis} we have schematically
\begin{align*}
& A_2(\p) - \half D^2\HH[ e ,0]\pth{\pth{ \wc{g} + \breve{g} , \wc{\pi} + \breve{\pi}   } , \pth{ \wc{g} + \breve{g} , \wc{\pi} + \breve{\pi}   }  }
\\&\hspace{2cm} =  (e+ \chi_{\vec{p}}(g_{\vec{p}}-e))(\wc{g}+\breve{g}) \pth{ \nab (\chi_{\vec{p}}(g_{\vec{p}}-e))  \nab(\wc{g}+\breve{g})   + \chi_{\vec{p}}\pi_{\vec{p}}   ( \wc{\pi} + \breve{\pi} ) }
\\&\hspace{2cm}\quad +  \chi_{\vec{p}}(g_{\vec{p}}-e) \pth{ (\nab(\wc{g}+\breve{g}))^2 +  (\wc{\pi} + \breve{\pi})^2 } 
\\&\hspace{2cm}\quad + \pth{ (\chi_{\vec{p}}\pi_{\vec{p}})^2  +  (\nab(\chi_{\vec{p}}(g_{\vec{p}}-e)))^2 } (\wc{g}+\breve{g})^2.
\end{align*}
Moreover, using \eqref{D2H} and the fact that $\wc{g}$ is TT we obtain
\begin{align*}
\half D^2\HH[ e ,0]\pth{\pth{ \wc{g} + \breve{g} , \wc{\pi} + \breve{\pi}   } , \pth{ \wc{g} + \breve{g} , \wc{\pi} + \breve{\pi}   }  } & -  \half D^2\HH[ e ,0]\pth{\pth{ \wc{g}  , \wc{\pi}  } , \pth{ \wc{g} , \wc{\pi}   }}  -   \breve{g}^{ij} \De \wc{g}_{ij}
\\& =  \wc{g} \nab^2\breve{g}  +  \breve{g} \nab^2\breve{g}   +  \nab\wc{g} \nab\breve{g}   + ( \nab\breve{g})^2+   \wc{\pi}  \breve{\pi} + (\breve{\pi})^2 ,
\end{align*}
so that we obtain
\begin{align}
A_2(\p) -  &\half D^2\HH[ e ,0]\pth{\pth{ \wc{g}  , \wc{\pi}  } , \pth{ \wc{g} , \wc{\pi}   }}  -   \breve{g}^{ij} \De \wc{g}_{ij} \nonumber
\\& = (e+ \chi_{\vec{p}}(g_{\vec{p}}-e))(\wc{g}+\breve{g}) \pth{ \nab (\chi_{\vec{p}}(g_{\vec{p}}-e))  \nab(\wc{g}+\breve{g})   + \chi_{\vec{p}}\pi_{\vec{p}}   ( \wc{\pi} + \breve{\pi} ) }\label{error A2}
\\&\quad +  \chi_{\vec{p}}(g_{\vec{p}}-e) \pth{ (\nab(\wc{g}+\breve{g}))^2 +  (\wc{\pi} + \breve{\pi})^2 } + \pth{ (\chi_{\vec{p}}\pi_{\vec{p}})^2  +  (\nab(\chi_{\vec{p}}(g_{\vec{p}}-e)))^2 } (\wc{g}+\breve{g})^2\nonumber
\\&\quad +  \wc{g} \nab^2\breve{g}  +  \breve{g} \nab^2\breve{g}   +  \nab\wc{g} \nab\breve{g}   + ( \nab\breve{g})^2+   \wc{\pi}  \breve{\pi} + (\breve{\pi})^2\nonumber.
\end{align}
Finally, \eqref{ErrH} gives
\begin{align}
A_3(\p) & = (e+ \chi_{\vec{p}}(g_{\vec{p}}-e)) (\wc{g}+\breve{g})(\nab (\wc{g}+\breve{g}))^2   \label{error A3}
\\&\quad  +  (\wc{g}+\breve{g})^2 \pth{ \chi_{\vec{p}}\pi_{\vec{p}}(\wc{\pi} + \breve{\pi})  + \nab(\chi_{\vec{p}}(g_{\vec{p}}-e))\nab (\wc{g}+\breve{g}) } \nonumber
\\&\quad + (\wc{g}+\breve{g})^2 \pth{ (\nab (\wc{g}+\breve{g}))^2 + (\wc{\pi} + \breve{\pi})^2 }.\nonumber
\end{align}
Defining
\begin{align*}
\mathrm{Err}_1 = \sum_{i=1}^3 A_i(\p),
\end{align*}
and putting \eqref{error A1}, \eqref{error A2} and \eqref{error A3} together we obtain \eqref{exp H} with the following schematic expression for the error term:
\begin{align*}
\mathrm{Err}_1(\p) & =   (\wc{g} + \breve{g}) \nabla^2(\chi_{\vec{p}}(g_{\vec{p}}-e)) 
\\&\quad  + (e+\chi_{\vec{p}}(g_{\vec{p}}-e))^2 \pth{ \nabla (\chi_{\vec{p}}(g_{\vec{p}}-e)) \nabla (\wc{g} + \breve{g}) + \chi_{\vec{p}}\pi_{\vec{p}}(\wc{\pi}+\breve{\pi}) }
\\&\quad  + (e+\chi_{\vec{p}}(g_{\vec{p}}-e)) (\wc{g} + \breve{g}) \pth{\pth{ \nab(\chi_{\vec{p}}(g_{\vec{p}}-e)+\wc{g}+\breve{g}) }^2 + (\chi_{\vec{p}}\pi_{\vec{p}})^2   + \chi_{\vec{p}}\pi_{\vec{p}}   ( \wc{\pi} + \breve{\pi} ) }
\\&\quad +  (\wc{g}+\breve{g})^2 \pth{( \chi_{\vec{p}}\pi_{\vec{p}}+\wc{\pi} + \breve{\pi})^2  + \pth{ \nab(\chi_{\vec{p}}(g_{\vec{p}}-e)+\wc{g}+\breve{g}) }^2 }
\\&\quad +  \chi_{\vec{p}}(g_{\vec{p}}-e) \pth{\nabla^2 (\wc{g} + \breve{g})+ (\nab(\wc{g}+\breve{g}))^2 +  (\wc{\pi} + \breve{\pi})^2 }
\\&\quad +  \wc{g} \nab^2\breve{g}  +  \breve{g} \nab^2\breve{g}   +  \nab\wc{g} \nab\breve{g}   + ( \nab\breve{g})^2+   \wc{\pi}  \breve{\pi} + (\breve{\pi})^2.
\end{align*}
Using \eqref{estim kerr}, \eqref{assumption preli}, \eqref{estim kerr 2}, $\l\wc{g}\r_{C^2_{-q-\de}}+\l\wc{\pi}\r_{C^1_{-q-\de-1}}\lesssim \e$ (follows from Lemma \ref{lem plongement}) and the fact that $\breve{g}$ and $\breve{\pi}$ are supported in $\{1\leq r \leq 10\}$, we first obtain the bound
\begin{align}\label{estim Err 1}
\left| \mathrm{Err}_1(\p) \right| & \lesssim \frac{\e m}{(1+r)^{q+\de+3}} + \frac{\e}{(1+r)^{q+\de}}\pth{ |\nab\wc{g}|^2 + |\wc{\pi}|^2}  + \pth{\l \breve{g}\r_{W^{2,\infty}}^2 + \l \breve{\pi} \r_{W^{1,\infty}}^2}\mathbbm{1}_{\{1\leq r \leq 10\}} 
\\&\quad + \pth{ \l \breve{g} \r_{W^{2,\infty}} + \l \breve{\pi} \r_{W^{1,\infty}} } \pth{ |\wc{g}| + |\nab\wc{g}| + |\wc{\pi}| + m }  \mathbbm{1}_{\{1\leq r \leq 10\}} \nonumber.
\end{align}
Using in addition the distance $d(\p,\p')$ defined in \eqref{def distance} and \eqref{diff kerr 2} we also obtain
\begin{align}
&\left| \mathrm{Err}_1(\p) - \mathrm{Err}_1(\p')  -  T_1^{abcd} (1+r)^{-1}\dr_a\dr_b\wc{g}_{cd}  \right| \label{diff Err 1}
\\&\hspace{2cm} \lesssim_{\max(\ga,\ga')} \pth{   (1+r)^{-2}\pth{|\nab\wc{g}| + |\wc{\pi}|} + (1+r)^{-3}|\wc{g}| } \pth{ 1  + \eta^{-2}\max(m,m')}d(\p,\p')\nonumber
\\&\hspace{2cm}\quad + \max\pth{ \l \breve{g}\r_{W^{2,\infty}},\l \breve{g}'\r_{W^{2,\infty}}} \pth{ 1  + \eta^{-2}\max(m,m')}d(\p,\p')  \mathbbm{1}_{\{1\leq r \leq 10\}}\nonumber
\\&\hspace{2cm}\quad + \pth{ \max(m,m')+  |\wc{g}| + |\nab\wc{g}| + |\wc{\pi}| }d(\p,\p')\mathbbm{1}_{\{1\leq r \leq 10\}},\nonumber
\end{align}
where the numbers $T_1^{abcd}$ satisfy $\left| T_1^{abcd} \right| \lesssim \pth{ 1  + \eta^{-2}\max(m,m')}d(\p,\p')$. The terms $T_1^{abcd} (1+r)^{-1}\dr_a\dr_b\wc{g}_{cd}$ come from the terms $\chi_{\vec{p}}(g_{\vec{p}}-e)\nab^2\wc{g}$ in $\mathrm{Err}_1(\p)$, see Remark \ref{remark IPP} for an explanation on why we need to isolate them.

We now look at the momentum constraint. We recall Definition \ref{def para} and expand with the notations of Lemma \ref{lem Phi}:
\begin{align*}
\MM(\bar{g}(\p),\bar{\pi}(\p)) & = \MM(e+\chi_{\vec{p}}(g_{\vec{p}}-e) + \wc{g} + \breve{g},\chi_{\vec{p}}\pi_{\vec{p}} + \wc{\pi} + \breve{\pi})
\\&= \MM(e+\chi_{\vec{p}}(g_{\vec{p}}-e) ,\chi_{\vec{p}}\pi_{\vec{p}} ) + D\MM[e+\chi_{\vec{p}}(g_{\vec{p}}-e),\chi_{\vec{p}}\pi_{\vec{p}} ](\wc{g} + \breve{g},\wc{\pi}+\breve{\pi})
\\&\quad + \half D^2\MM[e+\chi_{\vec{p}}(g_{\vec{p}}-e),\chi_{\vec{p}}\pi_{\vec{p}} ]((\wc{g} + \breve{g}, \wc{\pi} + \breve{\pi}),(\wc{g} + \breve{g}, \wc{\pi} + \breve{\pi}))
\\&\quad + \mathrm{Err}\MM[e+\chi_{\vec{p}}(g_{\vec{p}}-e),\chi_{\vec{p}}\pi_{\vec{p}} ](\wc{g} + \breve{g}, \wc{\pi} + \breve{\pi})
\\& = \vcentcolon \MM(e+\chi_{\vec{p}}(g_{\vec{p}}-e) ,\chi_{\vec{p}}\pi_{\vec{p}} ) + B_1(\p) + B_2(\p) + B_3(\p).
\end{align*}
Thanks to \eqref{DM bis} we have schematically
\begin{align}
B_1(\p) - D\MM[e,0](\breve{\pi}) & =  \chi_{\vec{p}}(g_{\vec{p}}-e)  \nab(\wc{\pi}+\breve{\pi}) + (\wc{g}+\breve{g})  \nab (\chi_{\vec{p}}\pi_{\vec{p}}) \label{error B1}
\\&\quad + (e+\chi_{\vec{p}}(g_{\vec{p}}-e))^2 \pth{  \nab(\chi_{\vec{p}}(g_{\vec{p}}-e)) (\wc{\pi}+\breve{\pi}) +  \nab(\wc{g}+\breve{g})\chi_{\vec{p}}\pi_{\vec{p}} }\nonumber
\\&\quad + (e+\chi_{\vec{p}}(g_{\vec{p}}-e))\nab(\chi_{\vec{p}}(g_{\vec{p}}-e))  \chi_{\vec{p}}\pi_{\vec{p}}  (\wc{g}+\breve{g})\nonumber
\end{align}
where we have used \eqref{DM} and the fact that $\wc{\pi}$ is TT to get $D\MM[e,0](\wc{\pi})=0$. Thanks to \eqref{D2M bis} we have schematically
\begin{align*}
&B_2 (\p) - \half D^2\MM[e,0 ]((\wc{g} + \breve{g}, \wc{\pi} + \breve{\pi}),(\wc{g} + \breve{g}, \wc{\pi} + \breve{\pi})) 
\\&\hspace{2cm}= \chi_{\vec{p}}(g_{\vec{p}}-e) (\wc{g} + \breve{g}) \nab(\wc{\pi} + \breve{\pi})
\\&\hspace{2cm}\quad +  (e+\chi_{\vec{p}}(g_{\vec{p}}-e))(\wc{g}+\breve{g}) \pth{ (\wc{\pi}+\breve{\pi}) \nab(\chi_{\vec{p}}(g_{\vec{p}}-e)) +  \chi_{\vec{p}}\pi_{\vec{p}} \nab(\wc{g}+\breve{g}) }
\\&\hspace{2cm}\quad + \nab(\chi_{\vec{p}}(g_{\vec{p}}-e))  \chi_{\vec{p}}\pi_{\vec{p}}   (\wc{g}+\breve{g})^2.
\end{align*}
Moreover, using \eqref{D2M} we have
\begin{align*}
&\half D^2\MM[e , 0]((\wc{g} + \breve{g},\wc{\pi} + \breve{\pi}),(\wc{g} + \breve{g},\wc{\pi} + \breve{\pi})_i -  \half D^2\MM[e , 0]((\wc{g},\wc{\pi}),(\wc{g},\wc{\pi}))_i + \breve{g}^{k\ell}\dr_k \wc{\pi}_{\ell i} 
\\&\hspace{8cm} = \wc{g}  \nab\breve{\pi} + \wc{\pi} \nab\breve{g} + \breve{\pi} \nab\wc{g} + \breve{g} \nab\breve{\pi} + \breve{\pi} \nab\breve{g},
\end{align*}
so that we obtain
\begin{align}
& B_2 (\p)_i -   \half D^2\MM[e , 0]((\wc{g},\wc{\pi}),(\wc{g},\wc{\pi}))_i + \breve{g}^{k\ell}\dr_k \wc{\pi}_{\ell i} \label{error B2}
\\&\hspace{2cm} = \chi_{\vec{p}}(g_{\vec{p}}-e) (\wc{g} + \breve{g}) \nab(\wc{\pi} + \breve{\pi})\nonumber
\\&\hspace{2cm}\quad +  (e+\chi_{\vec{p}}(g_{\vec{p}}-e))(\wc{g}+\breve{g}) \pth{ (\wc{\pi}+\breve{\pi}) \nab(\chi_{\vec{p}}(g_{\vec{p}}-e)) +  \chi_{\vec{p}}\pi_{\vec{p}} \nab(\wc{g}+\breve{g}) }\nonumber
\\&\hspace{2cm}\quad + \nab(\chi_{\vec{p}}(g_{\vec{p}}-e))  \chi_{\vec{p}}\pi_{\vec{p}}   (\wc{g}+\breve{g})^2 + \wc{g}  \nab\breve{\pi} + \wc{\pi} \nab\breve{g} + \breve{\pi} \nab\wc{g} + \breve{g} \nab\breve{\pi} + \breve{\pi} \nab\breve{g}.\nonumber
\end{align}
Finally, \eqref{ErrM} implies
\begin{align}
B_3(\p) & =  (e+\chi_{\vec{p}}(g_{\vec{p}}-e)) (\wc{g}+\breve{g}) (\wc{\pi}+\breve{\pi})\nab(\wc{g}+\breve{g})  \label{error B3}
\\&\quad + (\wc{g}+\breve{g})^2(\wc{\pi}+\breve{\pi}) \nab(\chi_{\vec{p}}(g_{\vec{p}}-e)+\wc{g}+\breve{g}) \nonumber
\\&\quad + (\wc{g}+\breve{g})^2 \chi_{\vec{p}}\pi_{\vec{p}} \nab(\wc{g}+\breve{g}).\nonumber
\end{align}
Putting \eqref{error B1}, \eqref{error B2} and \eqref{error B3} together we obtain \eqref{exp M} with the following expression for the error term
\begin{align*}
\mathrm{Err}_2(\p) & = \chi_{\vec{p}}(g_{\vec{p}}-e)  \nab \wc{\pi}  + (\chi_{\vec{p}}(g_{\vec{p}}-e) + \breve{g})  \nab \breve{\pi} + (\wc{g}+\breve{g} +\nab\breve{g})  \nab (\chi_{\vec{p}}\pi_{\vec{p}} + \breve{\pi}) 
\\&\quad + (e+\chi_{\vec{p}}(g_{\vec{p}}-e))^2 \pth{  \nab(\chi_{\vec{p}}(g_{\vec{p}}-e)) (\wc{\pi}+\breve{\pi}) +  \nab(\wc{g}+\breve{g})\chi_{\vec{p}}\pi_{\vec{p}} }
\\&\quad + \chi_{\vec{p}}(g_{\vec{p}}-e) (\wc{g} + \breve{g}) \nab(\wc{\pi} + \breve{\pi}) + \wc{\pi} \nab\breve{g} + \breve{\pi} \nab\wc{g}  
\\&\quad +  (e+\chi_{\vec{p}}(g_{\vec{p}}-e))(\wc{g}+\breve{g}) \nab(\chi_{\vec{p}}(g_{\vec{p}}-e)+\wc{g}+\breve{g}) ( \chi_{\vec{p}}\pi_{\vec{p}}+\wc{\pi}+\breve{\pi}) 
\\&\quad + (\wc{g}+\breve{g})^2( \chi_{\vec{p}}\pi_{\vec{p}}   +\wc{\pi}+\breve{\pi}) \nab(\chi_{\vec{p}}(g_{\vec{p}}-e)+\wc{g}+\breve{g}) .
\end{align*}
As above this implies 
\begin{align*}
\left| \mathrm{Err}_2(\p) \right| & \lesssim  \frac{\e m}{(1+r)^{q+\de+3}} + \frac{\e}{(1+r)^{q+\de}}\pth{ |\nab\wc{g}|^2 + |\wc{\pi}|^2}  + \pth{\l \breve{g}\r_{W^{2,\infty}}^2 + \l \breve{\pi} \r_{W^{1,\infty}}^2}\mathbbm{1}_{\{1\leq r \leq 10\}} 
\\&\quad + \pth{ \l \breve{g} \r_{W^{2,\infty}} + \l \breve{\pi} \r_{W^{1,\infty}} } \pth{ |\wc{g}| + |\nab\wc{g}| + |\wc{\pi}| + m }  \mathbbm{1}_{\{1\leq r \leq 10\}} \nonumber,
\end{align*}
which implies \eqref{estim error} together with \eqref{estim Err 1}. We also obtain
\begin{align}
&\left| \mathrm{Err}_2(\p) - \mathrm{Err}_2(\p') - S_2^{abc} (1+r)^{-1} \dr_a \wc{\pi}_{bc} \right|\label{diff Err 2}
\\&\quad \lesssim_{\max(\ga,\ga')}  \pth{    (1+r)^{-2}\pth{|\nab\wc{g}| + |\wc{\pi}|} + (1+r)^{-3}|\wc{g}| } \pth{ 1  + \eta^{-2}\max(m,m')}d(\p,\p')\nonumber
\\&\quad\quad + \max\pth{ \l \breve{g}\r_{W^{2,\infty}},\l \breve{g}'\r_{W^{2,\infty}},\l \breve{\pi}\r_{W^{1,\infty}},\l \breve{\pi}'\r_{W^{1,\infty}}} \pth{ 1  + \eta^{-2}\max(m,m')}d(\p,\p')  \mathbbm{1}_{\{1\leq r \leq 10\}}\nonumber
\\&\quad\quad + \pth{ \max(m,m')+  |\wc{g}| + |\nab\wc{g}| + |\wc{\pi}| }d(\p,\p')\mathbbm{1}_{\{1\leq r \leq 10\}},\nonumber
\end{align}
where the numbers $S_2^{abc}$ satisfy $\left| S_2^{abc} \right| \lesssim \pth{ 1  + \eta^{-2}\max(m,m')}d(\p,\p')$. The terms $S_2^{abc} (1+r)^{-1} \dr_a \wc{\pi}_{bc}$ comes from the terms $\chi_{\vec{p}}(g_{\vec{p}}-e)\nab\wc{\pi}$ in $\mathrm{Err}_2(\p)$, see Remark \ref{remark IPP} for an explanation on why we need to isolate them. Together with \eqref{diff Err 1}, \eqref{diff Err 2} implies \eqref{diff error} and concludes the proof of the proposition.
\end{proof}

\begin{remark}\label{remark IPP}
Some of the terms involving $\nab^2\wc{g}$ and $\nab\wc{\pi}$ have been isolated in \eqref{exp H}, \eqref{exp M} and \eqref{diff error} so that, after integration by parts, they can be estimated with the smallness constant $\eta$ instead of $\e$ (a similar remark applies for the terms of the form $\nab^2\wc{g}$ in \eqref{estim remainders} and \eqref{diff remainders} in the next proposition). This will be used in the proofs of Lemmas \ref{lem estim diff F} and \ref{lem diff param}.
\end{remark}

\begin{proposition}\label{prop remainders}
Let $\pth{\wc{u},\wc{X}}\in H^2_{-q-\de}\times H^2_{-q-\de}$ and $\p\in P$ satisfying
\begin{align}\label{assumption preli 2}
m +\l \breve{g} \r_{W^{2,\infty}} + \l \breve{\pi} \r_{W^{1,\infty}} + \l \wc{u}\r_{L^\infty} \leq \half.
\end{align}
For $\OO=\HH,\MM$, there exist numbers $U_\OO^{abcd}$ and $V_\OO^{abcd}$ for $a,b,c,d=1,2,3$ satisfying 
\begin{align*}
\left| U_\OO^{abcd} \right|  &  \lesssim_\ga 1,
\\  \left| V_\OO^{abcd} \right| &  \lesssim_{\max(\ga,\ga')} 1, 
\end{align*}
and such that
\begin{align}
&\left| \RR_\OO\pth{\p,\wc{u},\wc{X}} - U_\OO^{abcd}\wc{u}\dr_a\dr_b\wc{g}_{cd} \right|  \nonumber
\\&\quad \lesssim_\ga  \left|\nab \wc{X}\right|^2 + \left|\nab\wc{u}\nab\wc{X}\right|  +  \pth{ \frac{m}{1+r} + |\wc{g}| }\left|\nab^2\pth{\wc{u} + \wc{X}}\right| \label{estim remainders}
\\&\quad\quad + \pth{ \frac{m}{(1+r)^2} + |\nab\wc{g}|+ |\wc{\pi}| } \left|\nab\pth{ \wc{X} + \wc{u} }\right| \nonumber
\\&\quad\quad +  \pth{ \frac{m}{(1+r)^3}  + \pth{\frac{m}{(1+r)^2} + |\nab\wc{g}|+|\wc{\pi}|}^2}|\wc{u}| \nonumber
\\&\quad\quad + \pth{ \l \breve{g} \r_{W^{2,\infty}} + \l \breve{\pi}\r_{W^{1,\infty}} } \pth{  \left|\nab^2\wc{X}\right| + \left| \nab \wc{X}\right|^2 +  \left|\nab \wc{X}\right| + |\nab^{\leq 2}\wc{u}| }  \mathbbm{1}_{\{1\leq r \leq 10\}},\nonumber
\end{align}
and 
\begin{align}
&\left|\RR_\OO\pth{\p,\wc{u},\wc{X}} -  \RR_\OO\pth{\p',\wc{u}',\wc{X}'} - V_\OO^{abcd} (\wc{u}-\wc{u}')\dr_a\dr_b\wc{g}_{cd}\right| \label{diff remainders}
\\&\quad \lesssim_{\max(\ga,\ga')}  \left| \pth{\nab\wc{X}}^2 - \pth{\nab\wc{X}'}^2 \right| + \left| \nab\wc{u}\nab\wc{X} - \nab\wc{u}'\nab\wc{X}'  \right| \nonumber
\\&\quad\quad + \pth{ 1  + \eta^{-2}\max(m,m')}  \left|\nab^2\pth{\wc{u} + \wc{X}+\wc{u}' + \wc{X}'}\right| \frac{d(\p,\p')}{1+r}      \nonumber
\\&\quad\quad + \pth{ 1  + \eta^{-2}\max(m,m')}   \left|\nab\pth{\wc{u} + \wc{X}+\wc{u}' + \wc{X}'}\right| \frac{d(\p,\p')}{(1+r)^2}      \nonumber
\\&\quad\quad + \pth{ 1  + \eta^{-2}\max(m,m')}|\wc{u}+\wc{u}'| \frac{d(\p,\p')}{(1+r)^3}\nonumber
\\&\quad\quad +  \pth{ \frac{m+m'}{1+r} + |\wc{g}| } \left| \nab^2\pth{\wc{u}-\wc{u}' + \wc{X}-\wc{X}'}\right| \nonumber
\\&\quad\quad +  \pth{ \frac{m+m'}{(1+r)^2} + |\nab\wc{g}| + |\wc{\pi}| }\left|\nab\pth{\wc{u}-\wc{u}' + \wc{X}-\wc{X}'}\right|\nonumber
\\&\quad\quad + \pth{ \frac{m+m'}{(1+r)^3}  + \pth{ \frac{m+m'}{(1+r)^2} + |\nab\wc{g}| + |\wc{\pi}| }^2}| \wc{u}-\wc{u}'|\nonumber
\\&\quad\quad + \max\pth{ \l \breve{g} \r_{W^{2,\infty}}, \l \breve{\pi}\r_{W^{1,\infty}},\l \breve{g}' \r_{W^{2,\infty}}, \l \breve{\pi}'\r_{W^{1,\infty}}}\nonumber
\\&\hspace{2cm}\quad \times \left| \nab^{\leq 2}\pth{ \wc{u}-\wc{u}' + \wc{X}- \wc{X'}} +  \pth{\nab\wc{X}}^2 - \pth{\nab\wc{X}'}^2 \right| \mathbbm{1}_{\{1\leq r \leq 10\}}\nonumber
\\&\quad\quad +   \left| \nab^{\leq 2}\pth{ \wc{u}+\wc{u}' + \wc{X}+ \wc{X'}} +  \pth{\nab\wc{X}}^2 + \pth{\nab\wc{X}'}^2 \right|d(\p,\p')\mathbbm{1}_{\{1\leq r \leq 10\}}\nonumber,
\end{align}
where $\pth{\wc{u},\wc{X}},\pth{\wc{u}',\wc{X}'}\in \pth{ H^2_{-q-\de}}^4$ and $\p,\p'\in P$ satisfy \eqref{assumption preli 2}.
\end{proposition}

\begin{proof}
We start with $\RR_\HH\pth{\p,\wc{u},\wc{X}}$, whose expression is given in \eqref{remainders 1}. Schematically we obtain
\begin{align*}
\RR_\HH\pth{\p,\wc{u},\wc{X}} & = (1+\wc{u})^{-3} \bar{g}(\p)^2 \pth{\pth{  \nab\wc{X} }^2 + \bar{\pi}(\p)\nab\wc{X}} +  \pth{(1+\wc{u})^{-3} - 1} \pth{\bar{g}(\p)\bar{\pi}(\p)}^2   
\\&\quad + \wc{u}R(\bar{g}(\p)) + (\bar{g}(\p)-e)\nab^2\wc{u} + \bar{g}(\p)^2\nab\bar{g}(\p)\nab\wc{u}.
\end{align*}
From \eqref{estim kerr}, \eqref{assumption preli 2}, \eqref{estim kerr 2} and $\left| \pth{(1+\wc{u})^{-3} - 1} \right|  \lesssim |\wc{u}|$ we obtain
\begin{align*}
&\left| \RR_\HH\pth{\p,\wc{u},\wc{X}} - U_\HH^{abcd}\wc{u}\dr_a\dr_b \wc{g}_{cd} \right| 
\\& \lesssim_\ga  \left|\nab \wc{X}\right|^2 + \pth{ \frac{m}{(1+r)^2} + |\wc{\pi}| }\left| \nab\pth{ \wc{X} + \wc{u} }\right| + |\wc{u}|\pth{ \frac{m}{(1+r)^2} + |\nab\wc{g}| + |\wc{\pi}|}^2  
\\&\quad + |\wc{u}| \frac{m}{(1+r)^3}  + \pth{ \frac{m}{1+r} + |\wc{g}| }|\nab^2\wc{u}| + |\nab\wc{g}||\nab\wc{u}|
\\&\quad + \pth{ \l \breve{g} \r_{W^{2,\infty}} + \l \breve{\pi}\r_{W^{1,\infty}} } \pth{ \left| \nab \wc{X}\right|^2 +  |\nab \wc{X}| + |\nab^{\leq 2}\wc{u}| }  \mathbbm{1}_{\{1\leq r \leq 10\}},
\end{align*}
where the terms of the form $\wc{u} \nab^2 \wc{g}$ that we isolated come from $\wc{u}R(\bar{g}(\p)) $ in $\RR_\HH\pth{\p,\wc{u},\wc{X}}$.
Moreover, using \eqref{diff kerr 2} we also obtain
\begin{align*}
&\left|\RR_\HH\pth{\p,\wc{u},\wc{X}} -  \RR_\HH\pth{\p',\wc{u}',\wc{X}'} - V_\HH^{abcd} (\wc{u}-\wc{u}')\dr_a\dr_b\wc{g}_{cd}\right| 
\\&\quad \lesssim_{\max(\ga,\ga')}  \left| \pth{\nab\wc{X}}^2 - \pth{\nab\wc{X}'}^2 \right| 
\\&\quad\quad + \pth{ 1  + \eta^{-2}\max(m,m')}\left|\nab^2\pth{\wc{u} + \wc{X}+\wc{u}' + \wc{X}'}\right| \frac{d(\p,\p')}{1+r}
\\&\quad\quad +  \pth{ \frac{m+m'}{1+r} + |\wc{g}| } \left| \nab^2(\wc{u}-\wc{u}')\right| 
\\&\quad\quad +  \pth{ 1  + \eta^{-2}\max(m,m')}\left|\nab\pth{\wc{u} + \wc{X}+\wc{u}' + \wc{X}'}\right| \frac{d(\p,\p')}{(1+r)^2}
\\&\quad\quad +  \pth{ \frac{m+m'}{(1+r)^2} + |\nab\wc{g}| + |\wc{\pi}| }\left|\nab\pth{\wc{u}-\wc{u}' + \wc{X}-\wc{X}'}\right|  
\\&\quad\quad + \pth{ 1  + \eta^{-2}\max(m,m')}|\wc{u}+\wc{u}'| \frac{d(\p,\p')}{(1+r)^3}
\\&\quad\quad + \pth{ \frac{m+m'}{(1+r)^3}  + \pth{ \frac{m+m'}{(1+r)^2} + |\nab\wc{g}| + |\wc{\pi}| }^2}| \wc{u}-\wc{u}'|
\\&\quad\quad + \max\pth{ \l \breve{g} \r_{W^{2,\infty}}, \l \breve{\pi}\r_{W^{1,\infty}},\l \breve{g}' \r_{W^{2,\infty}}, \l \breve{\pi}'\r_{W^{1,\infty}}}
\\&\hspace{2cm}\quad \times \left| \nab^{\leq 2}( \wc{u}-\wc{u}') +\nab\pth{ \wc{X}- \wc{X'}} +  \pth{\nab\wc{X}}^2 - \pth{\nab\wc{X}'}^2 \right| \mathbbm{1}_{\{1\leq r \leq 10\}}
\\&\quad\quad +   \left| \nab^{\leq 2}( \wc{u}+\wc{u}') +\nab\pth{ \wc{X}+ \wc{X'}}  +  \pth{\nab\wc{X}}^2 + \pth{\nab\wc{X}'}^2 \right|d(\p,\p')\mathbbm{1}_{\{1\leq r \leq 10\}}.
\end{align*}
We deal with $\RR_\MM\pth{\p,\wc{u},\wc{X}}$ in a similar way, relying on \eqref{remainders 2} which implies the following schematic expression
\begin{align*}
\RR_\MM\pth{\p,\wc{u},\wc{X}} & = (1+\wc{u})^{-1} \bar{g}(\p) \nab\wc{u} \pth{ \bar{\pi}(\p) + \nab\wc{X} } + (\bar{g}(\p)-e)\nab^2\wc{X} + \bar{g}(\p)^2\nab\bar{g}(\p)\nab\wc{X}.
\end{align*}
This concludes the proof of the proposition.
\end{proof}

\subsection{Projections of $A\pth{ \p , \wc{u},\wc{X} }$ and $B\pth{ \p , \wc{u},\wc{X} }$}\label{section projection}

In this section, we give two lemmas computing the projections appearing on the RHS of the orthogonality conditions \eqref{ortho 1}-\eqref{ortho 4}. Note that we still assume that the parameter $\p\in P$ satisfies \eqref{assumption preli 2} so that the results of Propositions \ref{prop main terms} and \ref{prop remainders} hold.

\begin{lemma}\label{lem ortho calcul A}
For $-1\leq \ell \leq 1$ we have
\begin{align}
\left\langle A\pth{ \p , \wc{u},\wc{X} }, w_{1,0} \right\rangle & = 16\pi \ga m  - \eta^2 + \RR_{\HH,1}(\vec{p}) \label{eq A w10}  
\\&\quad + \left\langle  \breve{g}^{ij} \De \wc{g}_{ij}  + \mathrm{Err}_1(\p) +\RR_\HH\pth{\p,\wc{u},\wc{X}}  , w_{1,0}\right\rangle ,\nonumber
\\ \left\langle A\pth{ \p , \wc{u},\wc{X} }, w_{2,\ell} \right\rangle & = \sqrt{3}my^{\ell'} + Q^A_{2,\ell}(\wc{g},\wc{\pi}) + \sqrt{3}\RR_{\HH,x^{\ell'}}(\vec{p})  \label{eq A w2}
\\&\quad + \left\langle \breve{g}^{ij} \De \wc{g}_{ij}  + \mathrm{Err}_1(\p) +\RR_\HH\pth{\p,\wc{u},\wc{X}}  , w_{2,\ell}\right\rangle \nonumber,
\end{align}
where $\ell'$ is such that\footnote{In fact, we have $w_{2,-1}=\sqrt{3}x^1$, $w_{2,0}=\sqrt{3}x^3$ and $w_{2,1}=\sqrt{3}x^2$.} $w_{2,\ell}= \sqrt{3}x^{\ell'}$, and 
\begin{align*}
|Q^A_{2,\ell}(\wc{g},\wc{\pi})|\lesssim \eta^{2(1-\th)}\e^{2\th},
\end{align*}
with
\begin{align*}
\th\vcentcolon=\frac{1}{2q+2\de-1}.
\end{align*}
For $3\leq j\leq q$ and $-(j-1)\leq \ell \leq j-1$ we have
\begin{align}
\left\langle A\pth{ \p , \wc{u},\wc{X} }, w_{j,\ell} \right\rangle & = \left\langle \breve{g}  , D\HH[e,0]^*( w_{j,\ell}) \right\rangle  +  \left\langle \HH\pth{e + \chi_{\vec{p}}(g_{\vec{p}}-e) , \chi_{\vec{p}}\pi_{\vec{p}} } , w_{j,\ell} \right\rangle  \label{eq A wj}
\\&\quad + Q^A_{j,\ell}(\wc{g},\wc{\pi}) + \left\langle  \breve{g}^{ij} \De \wc{g}_{ij}  + \mathrm{Err}_1(\p) +\RR_\HH\pth{\p,\wc{u},\wc{X}}  , w_{j,\ell}\right\rangle \nonumber,
\end{align}
where 
\begin{align*}
|Q^A_{j,\ell}(\wc{g},\wc{\pi})|\lesssim \eta\e.
\end{align*}
\end{lemma}

\begin{proof}
Thanks to \eqref{exp H} we have
\begin{align}
\left\langle A\pth{ \p , \wc{u},\wc{X} }, w_{1,0} \right\rangle & = \left\langle \HH\pth{e + \chi_{\vec{p}}(g_{\vec{p}}-e) , \chi_{\vec{p}}\pi_{\vec{p}} } , w_{1,0} \right\rangle + \left\langle D\HH[ e ,0]\pth{\breve{g}}, w_{1,0} \right\rangle \label{A w10}
\\&\quad + \half \left\langle D^2\HH[ e ,0]\pth{\pth{ \wc{g}  , \wc{\pi}  } , \pth{ \wc{g} , \wc{\pi}   }}, w_{1,0} \right\rangle \nonumber
\\&\quad + \left\langle \breve{g}^{ij} \De \wc{g}_{ij} + \mathrm{Err}_1(\p) +\RR_\HH\pth{\p,\wc{u},\wc{X}} , w_{1,0} \right\rangle\nonumber.
\end{align}
For the first term in \eqref{A w10}, we use \eqref{int kerr H 1}. The second term vanishes since $w_{1,0}$ belongs to the kernel of $D\HH[e,0]^*$ (see Proposition \ref{prop KIDS}). For the third term we use \eqref{D2H}, the fact that $\wc{g}$ and $\wc{\pi}$ are TT and integrate by parts to obtain
\begin{align*}
\half \left\langle D^2\HH[ e ,0]\pth{\pth{ \wc{g}  , \wc{\pi}  } , \pth{ \wc{g} , \wc{\pi}   }}, w_{1,0} \right\rangle & =  \int_{\RRR^3}\pth{   \wc{g}^{ij} \De \wc{g}_{ij}     + \frac{3}{4}  | \nabla \wc{g}|^2  - \frac{1}{2}    \dr^b \wc{g}^{jk}  \dr_{j} \wc{g}_{kb} - |\wc{\pi}|^2}
\\& = - \frac{1}{4}\l \nab\wc{g}\r_{L^2}^2 - \l \wc{\pi}\r_{L^2}^2
\\& = -\eta^2,
\end{align*}
where we recall that $w_{1,0}=1$. This concludes the proof of \eqref{eq A w10}.

We now turn to the proof of \eqref{eq A w2}. We again use \eqref{exp H} to obtain
\begin{align}
\left\langle A\pth{ \p , \wc{u},\wc{X} }, w_{2,\ell} \right\rangle & = \left\langle \HH\pth{e + \chi_{\vec{p}}(g_{\vec{p}}-e) , \chi_{\vec{p}}\pi_{\vec{p}} } , w_{2,\ell} \right\rangle + \left\langle D\HH[ e ,0]\pth{\breve{g}}, w_{2,\ell} \right\rangle \label{A w2}
\\&\quad + \half \left\langle D^2\HH[ e ,0]\pth{\pth{ \wc{g}  , \wc{\pi}  } , \pth{ \wc{g} , \wc{\pi}   }}, w_{2,\ell}\right\rangle \nonumber
\\&\quad + \left\langle \breve{g}^{ij} \De \wc{g}_{ij} + \mathrm{Err}_1(\p) +\RR_\HH\pth{\p,\wc{u},\wc{X}} , w_{2,\ell} \right\rangle\nonumber.
\end{align}
For the first term in \eqref{A w2}, we use \eqref{int kerr H x}. The second term vanishes since $w_{2,\ell}$ belongs to the kernel of $D\HH[e,0]^*$ (see Proposition \ref{prop KIDS}). For the third term, we rename it $Q^A_{2,\ell}(\wc{g},\wc{\pi})$ and use \eqref{D2H} and the fact that $\wc{g}$ is TT:
\begin{align}\label{3rd}
\left| Q^A_{2,\ell}(\wc{g},\wc{\pi}) \right| & \lesssim  \int_{\RRR^3}r   | \nab \wc{g} |^2 +  \left| \int_{\RRR^3} x^{\ell'}  \wc{g}^{ij} \De \wc{g}_{ij} \right| + \int_{\RRR^3}r  |\wc{\pi} |^2,
\end{align}
where we also used $|w_{2,\ell}|\lesssim r$. We estimate each integral on the RHS of \eqref{3rd}.
\begin{itemize}
\item For the first integral, we apply Hölder's inequality with well-chosen exponents
\begin{align*}
\int_{\RRR^3}r   | \nab \wc{g} |^2 & \lesssim \l (\nab\wc{g})^{2-2\th} \r_{L^\frac{1}{1-\th}} \l (1+r) (\nab\wc{g})^{2\th} \r_{L^{\frac{1}{\th}}}
\\& \lesssim \l \nab\wc{g} \r_{L^2}^{2(1-\th)} \l (1+r)^{\frac{1}{2\th}}\nab\wc{g} \r_{L^2}^{2\th}
\\& \lesssim \eta^{2(1-\th)} \e^{2\th},
\end{align*}
where we recall that $\th=\frac{1}{2q+2\de-1}$.
\item For the second integral, we integrate by parts
\begin{align*}
\int_{\RRR^3} x^{\ell'}  \wc{g}^{ij} \De \wc{g}_{ij} & = - \int_{\RRR^3} \dr^k (x^{\ell'} \wc{g}^{ij} ) \dr_k \wc{g}_{ij}
\\& = - \int_{\RRR^3} x^{\ell'} |\nab\wc{g}|^2 - \int_{\RRR^3} \wc{g}^{ij}  \dr_{\ell'} \wc{g}_{ij}.
\end{align*}
By integration by parts again, we see that $ \int_{\RRR^3} \wc{g}^{ij}  \dr_{\ell'} \wc{g}_{ij}= \half \int_{\RRR^3} \dr_{\ell'} |\wc{g}|^2=0$. Therefore we have
\begin{align*}
\left| \int_{\RRR^3} x^{\ell'}  \wc{g}^{ij} \De \wc{g}_{ij} \right| & \lesssim \int_{\RRR^3}r   | \nab \wc{g} |^2
\\&\lesssim \eta^{2(1-\th)} \e^{2\th},
\end{align*}
where we used the estimate of the first integral.
\item The third integral on the RHS of \eqref{3rd} is estimated in the same manner as the first using the fact that $\wc{\pi}$ satisfies the same $L^2$ bounds as $\nab\wc{g}$ up to a constant factor.
\end{itemize}
We have thus proved that
\begin{align*}
\left|Q^A_{2,\ell}(\wc{g},\wc{\pi})\right| & \lesssim \eta^{2(1-\th)} \e^{2\th}.
\end{align*}
This concludes the proof of \eqref{eq A w2}.

Finally, we prove \eqref{eq A wj}. This proceeds in a similar manner to before. From \eqref{exp H} we obtain
\begin{align}
\left\langle A\pth{ \p , \wc{u},\wc{X} }, w_{j,\ell} \right\rangle & = \left\langle \HH\pth{e + \chi_{\vec{p}}(g_{\vec{p}}-e) , \chi_{\vec{p}}\pi_{\vec{p}} }, w_{j,\ell} \right\rangle + \left\langle D\HH[ e ,0]\pth{\breve{g}}  , w_{j,\ell} \right\rangle   \label{A wj}
\\&\quad + \half \left\langle D^2\HH[ e ,0]\pth{\pth{ \wc{g}  , \wc{\pi}  } , \pth{ \wc{g} , \wc{\pi}   }}, w_{j,\ell} \right\rangle  \nonumber
\\&\quad + \left\langle \breve{g}^{ij} \De \wc{g}_{ij} + \mathrm{Err}_1(\p) +\RR_\HH\pth{\p,\wc{u},\wc{X}} , w_{j,\ell} \right\rangle\nonumber.
\end{align}
We leave the first term as it is. The second term can be rewritten
\begin{align*}
\left\langle D\HH[ e ,0]\pth{\breve{g}}  , w_{j,\ell} \right\rangle = \left\langle \breve{g}  , D\HH[e,0]^*( w_{j,\ell}) \right\rangle.
\end{align*}
For the third term in \eqref{A wj}, we rename it $Q^A_{j,\ell}(\wc{g},\wc{\pi})$ and use \eqref{D2H}, $|w_{j,\ell}|\lesssim r^{j-1}$ and integrate by parts the terms of the form $\wc{g}\nab^2\wc{g}$:
\begin{align}\label{third term}
\left| Q^A_{j,\ell}(\wc{g},\wc{\pi}) \right| & \lesssim \int_{\RRR^3}r^{j-1}\pth{ |\nab\wc{g}|^2 + |\wc{\pi}|^2} + \int_{\RRR^3} r^{j-2} | \wc{g}| |\nab\wc{g}|.
\end{align}
We estimate each integral on the RHS of \eqref{third term}.
\begin{itemize}
\item For the first integral, we apply Hölder's inequality with well-chosen exponents
\begin{align}
\int_{\RRR^3}r^{j-1}   | \nab \wc{g} |^2 & \lesssim \l (\nab\wc{g})^{2-2(j-1)\th} \r_{L^\frac{1}{1-(j-1)\th}} \l (1+r)^{j-1} (\nab\wc{g})^{2(j-1)\th} \r_{L^{\frac{1}{(j-1)\th}}} \label{interpolation utile}
\\& \lesssim \l \nab\wc{g} \r_{L^2}^{2(1-(j-1)\th)} \l (1+r)^{\frac{1}{2\th}}\nab\wc{g} \r_{L^2}^{2(j-1)\th}\nonumber
\\& \lesssim \eta^{2(1-(j-1)\th)} \e^{2(j-1)\th}\nonumber
\\& \lesssim \eta \e,\nonumber
\end{align}
where we used $\eta\lesssim\e$ and $2(j-1)\th<1$ since $j\leq q$ and $\de>0$. Since $\wc{\pi}$ and $\nab\wc{g}$ satisfy the same estimates we also obtain
\begin{align*}
\int_{\RRR^3}r^{j-1} |\wc{\pi}|^2 \lesssim \eta\e.
\end{align*}
\item For the second integral in \eqref{third term} we first use Cauchy-Schwarz inequality and Hardy's inequality $\l r^{-1} \wc{g} \r_{L^2}\lesssim \l \nab\wc{g}\r_{L^2}\lesssim \eta $ to get 
\begin{align*}
\int_{\RRR^3} r^{j-2} | \wc{g}| |\nab\wc{g}| & \lesssim \eta \pth{ \int_{\RRR^3} (1+r)^{2(j-1)} |\nab \wc{g}|^2 }^\half.
\end{align*}
We estimate the integral $\int_{\RRR^3} (1+r)^{2(j-1)} |\nab \wc{g}|^2$ as above, with a $(1+r)^{2(j-1)}$ weight instead of a $(1+r)^{j-1}$ weight. Using $2(j-1)\th<1$, we obtain
\begin{align*}
\int_{\RRR^3} (1+r)^{2(j-1)} |\nab \wc{g}|^2 & \lesssim \eta^{2(1-2(j-1)\th)} \e^{4(j-1)\th}
\\& \lesssim \e^2
\end{align*}
and thus
\begin{align*}
\int_{\RRR^3} r^{j-2} | \wc{g}| |\nab\wc{g}| & \lesssim \eta\e.
\end{align*}
\end{itemize}
We have proved that 
\begin{align*}
\left|  Q^A_{j,\ell}(\wc{g},\wc{\pi}) \right| & \lesssim \eta \e.
\end{align*}
This concludes the proof of \eqref{eq A wj}.
\end{proof}

\begin{lemma}\label{lem ortho calcul B}
For $k,a,b=1,2,3$ we have
\begin{align}
\left\langle B\pth{ \p , \wc{u},\wc{X} }, W_{1,0,k} \right\rangle & = 8\pi\ga m v^k + \half \int_{\RRR^3} \wc{\pi}^{ij}  \dr_k \wc{g}_{ij}  - \RR_{\MM,W_{1,0,k}}(\vec{p})  \label{eq B W1}
\\&\quad + \left\langle  \breve{g}^{k\ell}\dr_k \wc{\pi}_{\ell i} - \mathrm{Err}_2(\p)_i + \RR_\MM\pth{\p,\wc{u},\wc{X}}_i, W_{1,0,k}^i \right\rangle, \nonumber
\\ \left\langle B\pth{ \p , \wc{u},\wc{X} }, \Om^{-}_{ab} \right\rangle & = - 8\pi m a^k   + Q^B_{ab}(\wc{g},\wc{\pi}) - \RR_{\MM,\Om^{-}_{ab}}(\vec{p})  \label{eq B W2}
\\&\quad\quad + \left\langle  \breve{g}^{k\ell}\dr_k \wc{\pi}_{\ell i} - \mathrm{Err}_2(\p)_i + \RR_\MM\pth{\p,\wc{u},\wc{X}}_i, (\Om^{-}_{ab})^i \right\rangle,\nonumber
\end{align}
where $k$ is such that $\in_{kab}=1$, and 
\begin{align*}
|Q^B_{ab}(\wc{g},\wc{\pi})|\lesssim \eta^{2(1-\th)}\e^{2\th}.
\end{align*}
For $Z\in\mathcal{Z}_q$ we have
\begin{align}
\left\langle B\pth{ \p , \wc{u},\wc{X} }, Z \right\rangle & = - \left\langle  \breve{\pi}, D\MM[e ,0]^*(Z) \right\rangle - \left\langle \MM\pth{e + \chi_{\vec{p}}(g_{\vec{p}}-e),\chi_{\vec{p}}\pi_{\vec{p}} }, Z \right\rangle  \label{eq B Z}
\\&\quad + Q^B_Z(\wc{g},\wc{\pi}) + \left\langle  \breve{g}^{k\ell}\dr_k \wc{\pi}_{\ell i} - \mathrm{Err}_2(\p)_i + \RR_\MM\pth{\p,\wc{u},\wc{X}}_i, Z^i \right\rangle  \nonumber,
\end{align}
where 
\begin{align*}
|Q^B_{Z}(\wc{g},\wc{\pi})|\lesssim \eta\e.
\end{align*}
\end{lemma}

\begin{proof}
We start with the proof of \eqref{eq B W1}. Thanks to \eqref{exp M} we have
\begin{align}
\left\langle B\pth{ \p , \wc{u},\wc{X} }, W_{1,0,k} \right\rangle & = - \left\langle \MM\pth{e + \chi_{\vec{p}}(g_{\vec{p}}-e),\chi_{\vec{p}}\pi_{\vec{p}} }, W_{1,0,k} \right\rangle \label{B W1}
\\&\quad  - \left\langle D\MM[e ,0]( \breve{\pi}) , W_{1,0,k} \right\rangle -\half  \left\langle D^2\MM[e , 0]((\wc{g},\wc{\pi}),(\wc{g},\wc{\pi})), W_{1,0,k} \right\rangle\nonumber
\\&\quad + \left\langle \breve{g}^{k\ell}\dr_k \wc{\pi}_{\ell i} - \mathrm{Err}_2(\p)_i + \RR_\MM\pth{\p,\wc{u},\wc{X}}_i, W_{1,0,k}^i \right\rangle. \nonumber
\end{align}
For the first term on the RHS of \eqref{B W1} we use \eqref{int kerr M 1}. The second term vanishes since $W_{1,0,k}$ belongs to the kernel of $D\MM[e,0]^*$ (see Proposition \ref{prop KIDS}). For the third term, we use \eqref{D2M}, integrate by parts and use the fact that $\wc{g}$ is TT:
\begin{align*}
-\half  \left\langle D^2\MM[e , 0]((\wc{g},\wc{\pi}),(\wc{g},\wc{\pi})), W_{1,0,k} \right\rangle & = \half \int_{\RRR^3} \wc{\pi}^{ij}  \dr_k \wc{g}_{ij}.  
\end{align*}
This concludes the proof of \eqref{eq B W1}.

We now turn to the proof of \eqref{eq B W2}, we again use \eqref{exp M} to obtain
\begin{align}
\left\langle B\pth{ \p , \wc{u},\wc{X} }, \Om^{-}_{ab} \right\rangle & = - \left\langle \MM\pth{e + \chi_{\vec{p}}(g_{\vec{p}}-e),\chi_{\vec{p}}\pi_{\vec{p}} }, \Om^{-}_{ab} \right\rangle - \left\langle D\MM[e ,0]( \breve{\pi}), \Om^{-}_{ab} \right\rangle \label{B W2}
\\&\quad -\half \left\langle D^2\MM[e , 0]((\wc{g},\wc{\pi}),(\wc{g},\wc{\pi})), \Om^{-}_{ab} \right\rangle \nonumber
\\&\quad + \left\langle \breve{g}^{k\ell}\dr_k \wc{\pi}_{\ell i} - \mathrm{Err}_2(\p)_i + \RR_\MM\pth{\p,\wc{u},\wc{X}}_i, (\Om^{-}_{ab})^i \right\rangle. \nonumber
\end{align}
For the first term in \eqref{B W2}, we use \eqref{int kerr M x}. The second term vanishes since $\Om^{-}_{ab}$ belongs to the kernel of $D\MM[e,0]^*$ (see Proposition \ref{prop KIDS}). For the third term, we rename it $Q^B_{ab}(\wc{g},\wc{\pi})$, use \eqref{D2M}, integrate by parts and use the fact that $\wc{g}$ is TT:
\begin{align}\label{3rd b}
\left| Q^B_{ab}(\wc{g},\wc{\pi}) \right| & \lesssim \int_{\RRR^3} r|\wc{\pi}|^2 + \int_{\RRR^3} r |\nab\wc{g}|^2 + \int_{\RRR^3} |\wc{g}| |\wc{\pi}|.
\end{align}
The first two integrals on the RHS of \eqref{3rd b} are estimated as in the proof of \eqref{eq A w2}:
\begin{align*}
 \int_{\RRR^3} r|\wc{\pi}|^2 + \int_{\RRR^3} r |\nab\wc{g}|^2 \lesssim \eta^{2(1-\th)} \e^{2\th}.
\end{align*}
For the third integral on the RHS of \eqref{3rd b}, we first use Cauchy-Schwarz inequality and Hardy's inequality $\l r^{-1} \wc{g} \r_{L^2}\lesssim \l \nab\wc{g}\r_{L^2} \lesssim \eta$ to get
\begin{align*}
\int_{\RRR^3} |\wc{g}| |\wc{\pi}| & \lesssim \eta \pth{ \int_{\RRR^3} (1+r)^2 |\wc{\pi}|^2 }^\half.
\end{align*} 
We estimate the integral $\int_{\RRR^3} (1+r)^2 |\wc{\pi}|^2$ as in the proof of \eqref{eq A w2}, with now a $(1+r)^2$ weight instead of a $1+r$ weight. We obtain
\begin{align*}
\int_{\RRR^3} (1+r)^2 |\wc{\pi}|^2 \lesssim \eta^{2\pth{1-2\th}}\e^{4\th},
\end{align*}
and thus
\begin{align*}
\int_{\RRR^3} |\wc{g}| |\wc{\pi}| & \lesssim \eta^{2(1-\th)} \e^{2\th}.
\end{align*} 
We have proved that
\begin{align*}
\left|Q^B_{ab}(\wc{g},\wc{\pi}) \right| & \lesssim \eta^{2(1-\th)} \e^{2\th}.
\end{align*}
This concludes the proof of \eqref{eq B W2}.

Finally we prove \eqref{eq B Z}. Thanks to \eqref{exp H} we have
\begin{align}
\left\langle B\pth{ \p , \wc{u},\wc{X} }, Z \right\rangle & = - \left\langle \MM\pth{e + \chi_{\vec{p}}(g_{\vec{p}}-e),\chi_{\vec{p}}\pi_{\vec{p}} } , Z \right\rangle - \left\langle D\MM[e ,0]( \breve{\pi}), Z \right\rangle\label{B Z}
\\&\quad - \half  \left\langle D^2\MM[e , 0]((\wc{g},\wc{\pi}),(\wc{g},\wc{\pi})), Z \right\rangle \nonumber
\\&\quad + \left\langle \breve{g}^{k\ell}\dr_k \wc{\pi}_{\ell i} - \mathrm{Err}_2(\p)_i + \RR_\MM\pth{\p,\wc{u},\wc{X}}_i, Z^i \right\rangle \nonumber.
\end{align}
We leave the first term as it is. The second term can be rewritten as
\begin{align*}
- \left\langle D\MM[e ,0]( \breve{\pi}), Z \right\rangle & = - \left\langle  \breve{\pi}, D\MM[e ,0]^*(Z) \right\rangle.
\end{align*}
For the third term in \eqref{B Z}, we rename it $Q^B_Z(\wc{g},\wc{\pi})$ and proceed as we did for the third term in \eqref{A wj}, since for each $Z\in\mathcal{Z}_q$ there exists $j\leq q$ such that $|Z|\lesssim r^{j-1}$. We obtain
\begin{align*}
\left| Q^B_Z(\wc{g},\wc{\pi}) \right| & \lesssim \eta \e.
\end{align*}
This concludes the proof of \eqref{eq B Z}.
\end{proof}

\begin{remark}\label{remark q12 2}
As Lemmas \ref{lem ortho calcul A} and \ref{lem ortho calcul B} show, if $q=1$ we only need the mass $m$ and the boost velocity $v$. Similarly, if $q=2$, we need the full black hole parameter $\vec{p}=(m,y,a,v)$ and the corrector $\breve{\pi}$ but we don't need the corrector $\breve{g}$.
\end{remark}

\subsection{Solving for $\p\pth{\wc{u},\wc{X}}$}\label{section def param}

We now solve for $\p\pth{\wc{u},\wc{X}}$ satisfying the orthogonality conditions \eqref{ortho 1}-\eqref{ortho 4}, assuming that the scalar function $\wc{u}$ and vector field $\wc{X}$ satisfy
\begin{align}\label{assumption u X}
\l \wc{u} \r_{H^2_{-q-\de}} + \l \wc{X} \r_{H^2_{-q-\de}} \leq C_0 \eta \e,
\end{align}
for some $C_0>0$ large enough and to be chosen later. The constant $\e$ will then be chosen small compared to $C_0^{-1}$, so that the control of $\wc{u}$ in \eqref{assumption u X} is consistent with \eqref{assumption preli 2}.

\subsubsection{The parameter map}

According to Lemmas \ref{lem ortho calcul A} and \ref{lem ortho calcul B}, the orthogonality conditions \eqref{ortho 1}-\eqref{ortho 4} are equivalent to the following non-linear system for the parameter $\p$:
\begin{equation}\label{system param}
\left\{
\begin{aligned}
16\pi \ga m & =   \eta^2  + F_1\pth{\p,\wc{u},\wc{X}},
\\ 16\pi\ga m v^k & =  - \int_{\RRR^3} \wc{\pi}^{ij}  \dr_k \wc{g}_{ij}  + F_{2,k}\pth{\p,\wc{u},\wc{X}},
\\ my^{\ell'} & =  -\frac{1}{\sqrt{3}}Q^A_{2,\ell}(\wc{g},\wc{\pi})  + \frac{1}{\sqrt{3}} F_{3,\ell'}\pth{\p,\wc{u},\wc{X}},
\\ m a^k & =  \frac{1}{8\pi}Q^B_{ab}(\wc{g},\wc{\pi}) + \frac{1}{8\pi}F_{4,k}\pth{\p,\wc{u},\wc{X}},
\\ \left\langle \breve{g}  , D\HH[e,0]^*( w_{j,\ell}) \right\rangle & = -\left\langle \HH\pth{e + \chi_{\vec{p}}(g_{\vec{p}}-e) , \chi_{\vec{p}}\pi_{\vec{p}} } , w_{j,\ell} \right\rangle 
\\&\quad - Q^A_{j,\ell}(\wc{g},\wc{\pi}) + F_{5,j,\ell}\pth{\p,\wc{u},\wc{X}},
\\ \left\langle  \breve{\pi}, D\MM[e ,0]^*(Z) \right\rangle & = -  \left\langle \MM\pth{e + \chi_{\vec{p}}(g_{\vec{p}}-e),\chi_{\vec{p}}\pi_{\vec{p}} }, Z \right\rangle 
\\&\quad +  Q^B_Z(\wc{g},\wc{\pi})   + F_{6,Z}\pth{\p,\wc{u},\wc{X}},
\end{aligned}
\right.
\end{equation}
where the remainders $F_i$ are given by
\begin{align*}
F_1\pth{\p,\wc{u},\wc{X}} & = - \RR_{\HH,1}(\vec{p})  - \left\langle  \breve{g}^{ij} \De \wc{g}_{ij}  + \mathrm{Err}_1(\p) +\RR_\HH\pth{\p,\wc{u},\wc{X}}  , w_{1,0}\right\rangle,
\\ F_{2,k}\pth{\p,\wc{u},\wc{X}} & = \RR_{\MM,W_{1,0,k}}(\vec{p})  - \left\langle  \breve{g}^{k\ell}\dr_k \wc{\pi}_{\ell i} - \mathrm{Err}_2(\p)_i + \RR_\MM\pth{\p,\wc{u},\wc{X}}_i, W_{1,0,k}^i \right\rangle,
\\ F_{3,\ell'}\pth{\p,\wc{u},\wc{X}} & = - \RR_{\HH,x^{\ell'}}(\vec{p})  -\left\langle  \breve{g}^{ij} \De \wc{g}_{ij}  + \mathrm{Err}_1(\p) +\RR_\HH\pth{\p,\wc{u},\wc{X}}  , w_{2,\ell}\right\rangle,
\\ F_{4,k}\pth{\p,\wc{u},\wc{X}} & = - \RR_{\MM,\Om^{-}_{ab}}(\vec{p})  + \left\langle  \breve{g}^{k\ell}\dr_k \wc{\pi}_{\ell i} - \mathrm{Err}_2(\p)_i + \RR_\MM\pth{\p,\wc{u},\wc{X}}_i, (\Om^{-}_{ab})^i \right\rangle,
\\ F_{5,j,\ell}\pth{\p,\wc{u},\wc{X}} & = - \left\langle \breve{g}^{ij} \De \wc{g}_{ij}  + \mathrm{Err}_1(\p) +\RR_\HH\pth{\p,\wc{u},\wc{X}}  , w_{j,\ell}\right\rangle,
\\ F_{6,Z}\pth{\p,\wc{u},\wc{X}} & =  \left\langle  \breve{g}^{k\ell}\dr_k \wc{\pi}_{\ell i} - \mathrm{Err}_2(\p)_i + \RR_\MM\pth{\p,\wc{u},\wc{X}}_i, Z^i \right\rangle,
\end{align*}
and where it is implied that the last two equations of \eqref{system param} must hold respectively for all $3\leq j \leq q$, $-(j-1)\leq \ell \leq j-1$ and for all $Z\in\mathcal{Z}_q$. The following lemma allows us to invert these equations. We leave the proof to the reader since it directly follows from the construction of the families $(h_{j,\ell})_{3\leq j \leq q, -(j-1)\leq \ell \leq j-1}$ and $(\varpi_Z)_{Z\in\mathcal{Z}_q}$ in Proposition \ref{prop KIDS} and in particular the invertibility of the matrices
\begin{align*}
\pth{ \left\langle h_{j,\ell}, D\HH[e,0]^*(w_{j',\ell'}) \right\rangle }_{3\leq j,j'\leq q, -(j-1)\leq \ell \leq j-1,-(j'-1)\leq \ell' \leq j'-1}
\end{align*}
and
\begin{align*}
\pth{ \left\langle \varpi_Z, D\MM[e,0]^*(Z') \right\rangle }_{(Z,Z')\in\mathcal{Z}_q^2}.
\end{align*}

\begin{lemma}\label{lem invert breve}
Let $(G_{j,\ell})_{3\leq j \leq q, -(j-1)\leq \ell \leq j-1}$ and $(P_Z)_{Z\in\mathcal{Z}_q}$ be two collections of real numbers. There exist two unique tensors $(\breve{g},\breve{\pi})\in\mathcal{A}_q\times \mathcal{B}_q$ such that
\begin{align*}
\left\langle \breve{g}  , D\HH[e,0]^*( w_{j,\ell}) \right\rangle & = G_{j,\ell},
\\ \left\langle  \breve{\pi}, D\MM[e ,0]^*(Z) \right\rangle & = P_Z,
\end{align*}
respectively for all $3\leq j \leq q$, $-(j-1)\leq \ell \leq j-1$ and for all $Z\in\mathcal{Z}_q$. Moreover they satisfy
\begin{align*}
\l \breve{g} \r_{W^{2,\infty}} & \lesssim \max_{3\leq j \leq q, -(j-1)\leq \ell \leq j-1}|G_{j,\ell}|,
\\  \l \breve{\pi} \r_{W^{1,\infty}} & \lesssim \max_{Z\in\mathcal{Z}_q}|P_Z|.
\end{align*}
\end{lemma}

We now introduce the map which will allow us to solve system \eqref{system param} via a fixed point argument.

\begin{definition}\label{def param map}
For $D_0>0$, we define the following subset of the parameter space $P$:
\begin{align*}
B_{D_0} & = B_{\RRR}\pth{0,D_0 \eta^2}\times B_{\RRR^3}\pth{0, D_0 \pth{ \frac{\e}{\eta}}^{2\th}} \times B_{\RRR^3}\pth{0, D_0 \pth{ \frac{\e}{\eta}}^{2\th}} \times B_{\RRR^3}\pth{0,1-\frac{\a}{2}}
\\&\qquad \times B_{\mathcal{A}_q}\pth{0,D_0^{q+1}\eta\e} \times B_{\mathcal{B}_q}\pth{0,D_0^{q+1} \eta\e}.
\end{align*} 
The parameter map $\Psi^{(p)}$ is defined by
\begin{align*}
&\Psi^{(p)} : B_{D_0} \longrightarrow B_{D_0}
\\&\hspace{1.45cm} \p \longmapsto \p'
\end{align*}
such that $\p'=\pth{m',y',a',v',\breve{g}',\breve{\pi}'}$ solves the system
\begin{align}
16\pi \ga' m' & =   \eta^2  + F_1\pth{\p,\wc{u},\wc{X}}, \label{eq m}
\\ 16\pi\ga' m' (v')^k & =  - \int_{\RRR^3} \wc{\pi}^{ij}  \dr_k \wc{g}_{ij}  + F_{2,k}\pth{\p,\wc{u},\wc{X}}, \label{eq v}
\\ m'(y')^{\ell'} & =  -\frac{1}{\sqrt{3}}Q^A_{2,\ell}(\wc{g},\wc{\pi})  + \frac{1}{\sqrt{3}} F_{3,\ell'}\pth{\p,\wc{u},\wc{X}}, \label{eq y}
\\ m' (a')^k & =  \frac{1}{8\pi}Q^B_{ab}(\wc{g},\wc{\pi}) + \frac{1}{8\pi}F_{4,k}\pth{\p,\wc{u},\wc{X}},\label{eq a}
\\ \left\langle \breve{g}'  , D\HH[e,0]^*( w_{j,\ell}) \right\rangle & = -\left\langle \HH\pth{e + \chi_{\vec{p}\,'}(g_{\vec{p}\,'}-e) , \chi_{\vec{p}\,'}\pi_{\vec{p}\,'} } , w_{j,\ell} \right\rangle \label{eq g breve}
\\&\quad   - Q^A_{j,\ell}(\wc{g},\wc{\pi})  + F_{5,j,\ell}\pth{\p,\wc{u},\wc{X}}, \nonumber
\\ \left\langle  \breve{\pi}', D\MM[e ,0]^*(Z) \right\rangle & = -  \left\langle \MM\pth{e + \chi_{\vec{p}\,'}(g_{\vec{p}\,'}-e),\chi_{\vec{p}\,'}\pi_{\vec{p}\,'} }, Z \right\rangle   \label{eq pi breve}
\\&\quad +  Q^B_Z(\wc{g},\wc{\pi}) + F_{6,Z}\pth{\p,\wc{u},\wc{X}}.
\end{align}
\end{definition}

Note that the system \eqref{eq m}-\eqref{eq pi breve} is not linear since the LHS of the equations \eqref{eq m}-\eqref{eq a} involve products of different components of $\vec{p}\,'$ and also since $\vec{p}\,'$ appears in a nonlinear way on the RHS of the equations \eqref{eq g breve}-\eqref{eq pi breve} for $\breve{g}'$ and $\breve{\pi}'$. It has nevertheless a triangular structure since $\vec{p}\,'$ will be defined through \eqref{eq m}-\eqref{eq a} and then $\breve{g}'$ and $\breve{\pi}'$ will be defined through \eqref{eq g breve} and \eqref{eq pi breve}.

\begin{remark}\label{remark gamma}
If $\p \in B_{D_0}$, then in particular $|v|\leq 1 - \frac{\a}{2}$ which implies
\begin{align*}
\ga \leq \frac{1}{\sqrt{1- \pth{ 1 - \frac{\a}{2} }^2}}.
\end{align*}
Since $\e$ will be chosen small compared to $\a$, and hence to $\ga^{-1}$, in what follows, all the symbols $\lesssim_\ga$ (or $\lesssim_{\max(\ga,\ga')}$ if $\p,\p'\in B_{D_0}$) will simply become $\lesssim$. 
\end{remark}

The next lemma gives estimates for the remainders of the equations \eqref{eq m}-\eqref{eq pi breve}.

\begin{lemma} \label{lem estim diff F}
If $\p,\p'\in B_{D_0}$ then
\begin{align}
\sum_{i=1,\dots,6}\left| F_i\pth{\p,\wc{u},\wc{X}} \right| & \lesssim_{D_0,C_0} \eta^2\e,  \label{estim F}
\\ \sum_{i=1,2}\left| F_i\pth{\p,\wc{u},\wc{X}} - F_i\pth{\p',\wc{u},\wc{X}}\right|  &  \lesssim_{D_0,C_0} \eta d(\p,\p') , \label{diff F 1}
\\ \sum_{i=3,4}\left| F_i\pth{\p,\wc{u},\wc{X}} - F_i\pth{\p',\wc{u},\wc{X}}\right|  &  \lesssim_{D_0,C_0} \eta^{1-2\th}\e^{2\th} d(\p,\p') , \label{diff F 3}
\\ \sum_{i=5,6} \left| F_i\pth{\p,\wc{u},\wc{X}} - F_i\pth{\p',\wc{u},\wc{X}}\right| & \lesssim_{D_0,C_0} \e d(\p,\p'). \label{diff F 2}
\end{align}
\end{lemma}

\begin{proof}
We recall that $\wc{u}$ and $\wc{X}$ are assumed to satisfy \eqref{assumption u X}. In order to prove \eqref{estim F}, we note that all the remainders $F_b$ for $b=1,\dots,6$ are of the following schematic form\footnote{In this proof, the scalar product with $r^{j-1}$ is an abuse of notations standing for a scalar product with a scalar function $f_j$ satisfying $|f_j|+r|\nab f_j|\lesssim r^{j-1}$.}  
\begin{align}
F^{(i,\OO,Y,j)}\pth{\p,\wc{u},\wc{X}} = \RR_{\OO,Y}(\vec{p}) + \left\langle \breve{g}(\nab^2 \wc{g} + \nab \wc{\pi}) + \mathrm{Err}_i(\p) + \RR_\OO\pth{ \p,\wc{u},\wc{X}}, r^{j-1} \right\rangle,\label{def F iOj}
\end{align}
for $i=1,2$, $\OO=\HH,\MM$, $Y=1,x^k,W_{1,0,k},\Om^{-}_{ab}$ and $1\leq j \leq q$. For the first term in \eqref{def F iOj} we simply use \eqref{estim R bh bis}:
\begin{align*}
\left| \RR_{\OO,Y}(\vec{p}) \right| \lesssim_{D_0} \eta^4.
\end{align*}
We now estimate each term in the scalar product in \eqref{def F iOj}. For the first term, we use the support property of $\breve{g}$ and integrate by parts:
\begin{align*}
\left| \left\langle \breve{g}(\nab^2 \wc{g} + \nab \wc{\pi}) , r^{j-1} \right\rangle \right| & \lesssim \l \breve{g} \r_{W^{1,\infty}} \pth{ \l \nab\wc{g}\r_{L^2} + \l \wc{\pi}\r_{L^2}}
\\&\lesssim_{D_0} \eta^2\e.
\end{align*}
For the second term, we use \eqref{estim error} and $j\leq q$:
\begin{align*}
\left| \left\langle \mathrm{Err}_i(\p) , r^{j-1} \right\rangle \right| & \lesssim \e m \l (1+r)^{j-q-\de-4} \r_{L^1} + \e \pth{ \l \nab\wc{g}\r_{L^2}^2 + \l \wc{\pi}\r_{L^2}^2 } + \l \breve{g}\r_{W^{2,\infty}}^2 + \l \breve{\pi} \r_{W^{1,\infty}}^2
\\&\quad + \pth{ \l \breve{g} \r_{W^{2,\infty}} + \l \breve{\pi} \r_{W^{1,\infty}} } \pth{ m + \l r^{-1}\wc{g}\r_{L^2} + \l \nab \wc{g}\r_{L^2} + \l \wc{\pi} \r_{L^2} }
\\&\lesssim_{D_0}\eta^2\e.
\end{align*}
For the third term, we use \eqref{estim remainders}, $j\leq q$, $\de>0$:
\begin{align*}
&\left| \left\langle \RR_\OO\pth{ \p,\wc{u},\wc{X}} , r^{j-1} \right\rangle \right| 
\\&\qquad \lesssim  \l \nab\wc{X}\r_{L^2_{-q-\de-1}}^2 + \l \nab\wc{X}\r_{L^2_{-q-\de-1}}  \l \nab\wc{u}\r_{L^2_{-q-\de-1}}  
\\&\qquad\quad + \pth{ m \l (1+r)^{j-q-\de-\frac{5}{2}}\r_{L^2} + \l r^{-1}\wc{g}\r_{L^2} } \pth{ \l \nab^2\wc{u}\r_{L^2_{-q-\de-2}}  + \l \nab^2\wc{X}\r_{L^2_{-q-\de-2}}  }
\\&\qquad\quad + \pth{ m \l (1+r)^{j-q-\de-\frac{5}{2}} \r_{L^2} + \l \nab\wc{g}\r_{L^2} + \l\wc{\pi}\r_{L^2} } \pth{ \l \nab\wc{u}\r_{L^2_{-q-\de-1}}  + \l \nab\wc{X}\r_{L^2_{-q-\de-1}}  }
\\&\qquad\quad + \left| \int_{\RRR^3} r^{j-1}\wc{u} \nab^2\wc{g} \right| +  m \l (1+r)^{j-q-\de-\frac{5}{2}} \r_{L^2} \l \wc{u} \r_{L^2_{-q-\de}} + \pth{ \l \nab\wc{g}\r_{L^2}^2 + \l \wc{\pi} \r_{L^2}^2 } \l \wc{u} \r_{H^2_{-q-\de}}
\\&\qquad\quad + \pth{ \l \breve{g} \r_{W^{2,\infty}} + \l \breve{\pi}\r_{W^{1,\infty}} } \pth{ \l \wc{u} \r_{H^2_{-q-\de}} + \l \wc{X} \r_{H^2_{-q-\de}} }
\\&\qquad\lesssim_{D_0,C_0}\eta^2\e + \left| \int_{\RRR^3} r^{j-1}\wc{u} \nab^2\wc{g} \right|,
\end{align*}
where for the terms multiplied by $\mathbbm{1}_{\{1\leq r \leq 10\}}$ we used Cauchy-Schwarz inequality on a fixed bounded domain and added artificial weights as in $\l f \mathbbm{1}_{\{1\leq r \leq 10\}} \r_{L^1} \lesssim \l f \r_{L^2_\mu}$ (which holds for any $\mu\in\RRR$). We integrate by parts the term $r^{j-1}\wc{u} \nab^2\wc{g}$ in order to get $\eta^2\e$ instead of $\eta\e^2$:
\begin{align*}
\left| \int_{\RRR^3} r^{j-1}\wc{u} \nab^2\wc{g} \right| & \lesssim \left| \int_{\RRR^3} r^{j-1}\nab\wc{u} \nab\wc{g} \right| + \left| \int_{\RRR^3} r^{j-2}\wc{u} \nab\wc{g} \right|
\\&\lesssim \l \nab\wc{g} \r_{L^2} \pth{ \l \nab\wc{u}\r_{L^2_{-q-\de-1}} + \l \wc{u} \r_{L^2_{-q-\de}} }
\\&\lesssim_{C_0}\eta^2\e.
\end{align*}
We have proved that $\left| F^{(i,\OO,j)}\pth{\p,\wc{u},\wc{X}} \right| \lesssim_{D_0,C_0}\eta^2\e$, which proves \eqref{estim F}.

We now turn to the proof of \eqref{diff F 1}-\eqref{diff F 2}. We have schematically
\begin{align}
F^{(i,\OO,j)}\pth{\p,\wc{u},\wc{X}} - F^{(i,\OO,j)}\pth{\p',\wc{u},\wc{X}} & = \RR_{\OO,Y}(\vec{p}) - \RR_{\OO,Y}(\vec{p}\,')   \label{diff F inter}
\\&\quad  + \left\langle (\nab^2\wc{g} + \nab\wc{\pi})(\breve{g}-\breve{g}') , r^{j-1} \right\rangle \nonumber
\\&\quad + \left\langle \mathrm{Err}_i(\p) - \mathrm{Err}_i(\p') , r^{j-1} \right\rangle\nonumber
\\&\quad + \left\langle \RR_\OO\pth{\p,\wc{u},\wc{X}} - \RR_\OO\pth{\p',\wc{u},\wc{X}} , r^{j-1} \right\rangle\nonumber.
\end{align}
We estimate each terms in \eqref{diff F inter}. For the first term, we simply use \eqref{diff R bh bis}:
\begin{align*}
\left| \RR_{\OO,Y}(\vec{p}) - \RR_{\OO,Y}(\vec{p}\,') \right| & \lesssim_{D_0} \eta d(\p,\p').
\end{align*}
For the second term in \eqref{diff F inter}, we integrate by parts and use the support property of $\breve{g}-\breve{g}'$:
\begin{align*}
\left| \left\langle (\nab^2\wc{g} + \nab\wc{\pi})(\breve{g}-\breve{g}') , r^{j-1} \right\rangle \right| & \lesssim \left| \left\langle (\nab\wc{g} + \wc{\pi})\nab(\breve{g}-\breve{g}') , r^{j-1} \right\rangle \right| + \left| \left\langle (\nab\wc{g} + \wc{\pi})(\breve{g}-\breve{g}') , r^{j-2} \right\rangle \right|
\\&\lesssim \pth{ \l \nab\wc{g} \r_{L^2} + \l\wc{\pi}\r_{L^2} } \l \breve{g} - \breve{g}' \r_{W^{2,\infty}}
\\&\lesssim \eta d(\p,\p').
\end{align*}
For the fourth term in \eqref{diff F inter}, we use \eqref{diff remainders}, $m+m'\lesssim D_0 \eta^2$ and $j\leq q$:
\begin{align*}
&\left|  \left\langle \RR_\OO\pth{\p,\wc{u},\wc{X}} - \RR_\OO\pth{\p',\wc{u},\wc{X}}  , r^{j-1} \right\rangle \right| 
\\&\hspace{2cm} \lesssim_{D_0} \pth{ \int_{\RRR^3}r^{j-2}\left|\nab^2\pth{\wc{u}+\wc{X}}\right| + \int_{\RRR^3} r^{j-3}\left|\nab\pth{\wc{u}+\wc{X}}\right|   + \int_{\RRR^3}r^{j-4}|\wc{u}|  }d(\p,\p')
\\&\hspace{2cm}\quad + \pth{ \l \wc{u} \r_{H^2_{-q-\de}} + \l \wc{X} \r_{H^2_{-q-\de}} } d(\p,\p')
\\&\hspace{2cm} \lesssim_{D_0} \pth{ 1 + \l (1+r)^{j-q-\de-\frac{5}{2}} \r_{L^2}} \pth{ \l \wc{u} \r_{H^2_{-q-\de}} + \l \wc{X} \r_{H^2_{-q-\de}} } d(\p,\p')
\\&\hspace{2cm} \lesssim_{D_0,C_0} \eta\e d(\p,\p').
\end{align*}
For the third term in \eqref{diff F inter}, we distinguish between $j=1$, $j=2$ and $3\leq j \leq q$, which explains the difference between \eqref{diff F 1}, \eqref{diff F 3} and \eqref{diff F 2}. In all cases, we use \eqref{diff error}, $m+m'\lesssim D_0 \eta^2$ and 
\begin{align*}
\l \breve{g}\r_{W^{2,\infty}}+\l \breve{g}'\r_{W^{2,\infty}}+\l \breve{\pi}\r_{W^{1,\infty}}+\l \breve{\pi}'\r_{W^{1,\infty}}\lesssim D_0^{q+1}\eta\e.
\end{align*}
This gives
\begin{align*}
&\left|  \left\langle \mathrm{Err}_i(\p) - \mathrm{Err}_i(\p') , r^{j-1} \right\rangle \right|  
\\& \lesssim \left|T_1^{abcd}\right| \left| \int_{\RRR^3} \dr_a\dr_b\wc{g}_{cd} (1+r)^{j-2}  \right| + \left| S_2^{abc}\right| \left| \int_{\RRR^3}  \dr_a \wc{\pi}_{bc} (1+r)^{j-2}  \right| 
\\&\quad + \pth{ \int_{\RRR^3}(1+r)^{j-3}\pth{|\nab\wc{g}| + |\wc{\pi}| } + \int_{\RRR^3}(1+r)^{j-4}|\wc{g}|  }d(\p,\p') + C(D_0)\eta d(\p,\p')
\\&\lesssim \pth{ \int_{\RRR^3}(1+r)^{j-3}\pth{|\nab\wc{g}| + |\wc{\pi}| } + \int_{\RRR^3}(1+r)^{j-4}|\wc{g}|  }d(\p,\p') + C(D_0)\eta d(\p,\p'),
\end{align*}
where we integrated by parts in the first two integrals and used \eqref{estim T S}. 
\begin{itemize}
\item If $j=1$ (corresponds to \eqref{diff F 1}), we can use the fact that $(1+r)^{-2}\in L^2$ and obtain
\begin{align*}
\left|  \left\langle \mathrm{Err}_i(\p) - \mathrm{Err}_i(\p') , 1 \right\rangle \right| & \lesssim \l (1+r)^{-2} \r_{L^2} \pth{ \l \nab\wc{g}\r_{L^2} + \l \wc{\pi}\r_{L^2}  + \l r^{-1} \wc{g}\r_{L^2}  }d(\p,\p') 
\\&\quad + C(D_0)\eta d(\p,\p')
\\& \lesssim_{D_0}\eta d(\p,\p').
\end{align*}
\item If $j=2$ (corresponds to \eqref{diff F 3}), we interpolate between the $L^2$ norms and the weighted norms, thanks to Cauchy-Schwarz inequality and then a well-chosen Hölder's inequality:
\begin{align*}
\int_{\RRR^3}(1+r)^{-2}|\wc{g}| & \lesssim \l (1+r)^{-3-\de} \r_{L^1}^\half  \pth{ \int_{\RRR^3} (1+r)^{-1+\de} |\wc{g}|^2 }^\half 
\\& \lesssim \l \pth{\frac{|\wc{g}|^2}{r^2}}^{1-2\th} \r_{L^{\frac{1}{1-2\th}}}^{\half} \l \pth{ r^{2q+2\de-3}|\wc{g}|^2 }^{2\th} \r_{L^{\frac{1}{2\th}}}^\half 
\\& \lesssim \eta^{1-2\th}\e^{2\th},
\end{align*}
and
\begin{align*}
&\int_{\RRR^3}(1+r)^{-1}\pth{|\nab\wc{g}| + |\wc{\pi}| } 
\\& \lesssim \l (1+r)^{-3-\de} \r_{L^1}^\half \pth{ \pth{ \int_{\RRR^3}|\nab\wc{g}|^2 (1+r)^2 }^\half + \pth{\int_{\RRR^3}|\wc{\pi}|^2 (1+r)^2 }^\half  }
\\& \lesssim \eta^{1-2\th}\e^{2\th} ,
\end{align*}
where we have used the computation \eqref{interpolation utile} with $j=3$. This gives
\begin{align*}
\left|  \left\langle \mathrm{Err}_i(\p) - \mathrm{Err}_i(\p') , r \right\rangle \right|  & \lesssim_{D_0}\eta^{1-2\th}\e^{2\th}d(\p,\p').
\end{align*}
\item If $3\leq j\leq q$ (corresponds to \eqref{diff F 2}), we use the weighted norms of $\wc{g}$ and $\wc{\pi}$: 
\begin{align*}
\left|  \left\langle \mathrm{Err}_i(\p) - \mathrm{Err}_i(\p') , r^{j-1} \right\rangle \right|  & \lesssim \l (1+r)^{j-q-\de-\frac{5}{2}} \r_{L^2} \pth{ \l \wc{g} \r_{H^1_{-q-\de}} + \l \wc{\pi} \r_{L^2_{-q-\de-1}}  }d(\p,\p') 
\\&\quad + C(D_0)\eta d(\p,\p')
\\& \lesssim_{D_0}\e d(\p,\p').
\end{align*}
\end{itemize}
This concludes the proof of Lemma \ref{lem estim diff F}. 
\end{proof}

\subsubsection{A fixed point for the parameter map}

\begin{lemma}\label{lem param bound}
Consider $\pth{\wc{u},\wc{X}}\in H^2_{-q-\de}\times H^2_{-q-\de}$ satisfying \eqref{assumption u X}. If $D_0$ is a large enough universal constant, and if $\e$ is small enough compared to $C_0^{-1}$ and $D_0^{-1}$, then the parameter map $\Psi^{(p)}$ (see Definition \ref{def param map}) is well-defined.
\end{lemma}

\begin{proof}
Since the equations \eqref{eq m} and \eqref{eq v} are coupled through $\ga'$, the parameter map $\Psi^{(p)}$ is a priori not well-defined. However the coupling is easy to untangle, since the quantity $16\pi\ga'm'$ (i.e. the LHS of \eqref{eq m}) appears on the LHS of \eqref{eq v}. Hence, \eqref{eq v} rewrites
\begin{align*}
(v')^k & = \frac{ - \int_{\RRR^3} \wc{\pi}^{ij}  \dr_k \wc{g}_{ij}  + F_{2,k}\pth{\p,\wc{u},\wc{X}}}{ \frac{1}{4}\l \nab\wc{g}\r_{L^2}^2 + \l \wc{\pi}\r_{L^2}^2  + F_1\pth{\p,\wc{u},\wc{X}}}
\\& = \frac{ - \int_{\RRR^3} \wc{\pi}^{ij}  \dr_k \wc{g}_{ij}  }{ \frac{1}{4}\l \nab\wc{g}\r_{L^2}^2 + \l \wc{\pi}\r_{L^2}^2} \pth{ 1 + \frac{F_1\pth{\p,\wc{u},\wc{X}}}{\eta^2}}^{-1} + \frac{   F_{2,k}\pth{\p,\wc{u},\wc{X}}}{ \eta^2  + F_1\pth{\p,\wc{u},\wc{X}}}
\end{align*}
where we used the definition of $\eta$ (see Theorem \ref{maintheorem}). Thanks to \eqref{estim F} and of the bound
\begin{align*}
\left| \frac{ - \int_{\RRR^3} \wc{\pi}^{ij}  \dr_k \wc{g}_{ij}  }{ \frac{1}{4}\l \nab\wc{g}\r_{L^2}^2 + \l \wc{\pi}\r_{L^2}^2} \right|< 1,
\end{align*}
we obtain
\begin{align*}
(v')^k & = \frac{ - \int_{\RRR^3} \wc{\pi}^{ij}  \dr_k \wc{g}_{ij}  }{ \frac{1}{4}\l \nab\wc{g}\r_{L^2}^2 + \l \wc{\pi}\r_{L^2}^2} + C(D_0,C_0)\e.
\end{align*}
This defines the vector $v'$ and then implies $|v'|=J(\wc{g},\wc{\pi})  + C(D_0,C_0)\e$ (see Definition \ref{def J} for the definition of the functional $J(\wc{g},\wc{\pi})$). Thanks to the assumption $(\wc{g},\wc{\pi})\in V_\a$, we obtain 
\begin{align*}
|v'|  \leq 1 - \a + C(D_0,C_0)\e
\end{align*}
and then $|v'|  \leq 1 - \frac{\a}{2}$ if $\e$ is small compared to $C_0^{-1}$, $D_0^{-1}$ and $\a$. This implies $1\leq \ga'\lesssim 1$ according to Remark \ref{remark gamma}.

We can now define $m'$ through \eqref{eq m} and with \eqref{estim F} obtain the bound
\begin{align*}
 m' & =  \frac{\eta^2}{16\pi \ga'} \pth{ 1  + C(D_0,C_0)\e}.
\end{align*}
If $D_0$ is large enough and if $\e$ is small compared to $C_0^{-1}$ and $D_0^{-1}$ then we have the upper bound $m'\leq D_0 \eta^2$. Thanks to $\ga'\lesssim 1$, this also implies the lower bound $\eta^2\lesssim m'$.

We can now define $y'$ and $a'$ through \eqref{eq y} and \eqref{eq a}, after dividing by $m'$ using crucially the lower bound $\eta^2\lesssim m'$. With \eqref{estim F} and $\left| Q^A_{2,\ell}(\wc{g},\wc{\pi}) \right| + \left| Q^B_{ab}(\wc{g},\wc{\pi}) \right| \lesssim \eta^{2(1-\th)}\e^{2\th}$ (see Lemmas \ref{lem ortho calcul A} and \ref{lem ortho calcul B}) we obtain the bound
\begin{align*}
|y'| + |a'| \lesssim \pth{\frac{\e}{\eta}}^{2\th} + C(D_0,C_0)\e.
\end{align*}
If $D_0$ is large enough and $\e$ is small compared to $C_0^{-1}$ and $D_0^{-1}$ (recall that $\eta\lesssim\e$) this implies the bound $|y'|+|a'|\leq D_0\pth{\frac{\e}{\eta}}^{2\th}$.

With the help of Lemma \ref{lem invert breve} we can uniquely define $\breve{g}'\in\mathcal{A}_q$ and $\breve{\pi}'\in\mathcal{B}_q$ as solutions of \eqref{eq g breve} and \eqref{eq pi breve}. Moreover, using $ \left| Q^A_{j,\ell}(\wc{g},\wc{\pi}) \right| + \left| Q^B_Z(\wc{g},\wc{\pi})\right| \lesssim \eta\e$ from Lemmas \ref{lem ortho calcul A} and \ref{lem ortho calcul B} and \eqref{estim F}, we obtain the bound 
\begin{align*}
\l \breve{g}'\r_{W^{2,\infty}} &+ \l \breve{\pi}'\r_{W^{1,\infty}} 
\\& \lesssim  \int_{\RRR^3} \pth{ \left| \HH\pth{e + \chi_{\vec{p}\,'}(g_{\vec{p}\,'}-e) , \chi_{\vec{p}\,'}\pi_{\vec{p}\,'} }\right| + \left| \MM\pth{e + \chi_{\vec{p}\,'}(g_{\vec{p}\,'}-e) , \chi_{\vec{p}\,'}\pi_{\vec{p}\,'} }\right|  }r^{q-1} 
\\&\quad + \eta\e\pth{ 1 + C(D_0,C_0)\eta}
\\&\lesssim  m' \int_{\{ \la'\leq r \leq 2 \la'\}} r^{q-4}+ \eta\e\pth{ 1 + C(D_0,C_0)\eta}
\\&\lesssim m' (\la')^{q-1} + \eta\e\pth{ 1 + C(D_0,C_0)\eta}
\\& \lesssim D_0^q \eta^{2-2(q-1)\th}\e^{2(q-1)\th} + \eta\e\pth{ 1 + C(D_0,C_0)\eta},
\end{align*}
where we used \eqref{constraint kerr} and the estimates obtained on $m'$, $y'$ and $a'$. Using now $2(q-1)\th<1$ we obtain
\begin{align*}
\l \breve{g}'\r_{W^{2,\infty}} + \l \breve{\pi}'\r_{W^{1,\infty}} & \lesssim  \eta\e\pth{ 1 + D_0^q + C(D_0,C_0)\eta}.
\end{align*}
If $D_0$ is large enough, and if $\e$ is small compared to $C_0^{-1}$ and $D_0^{-1}$ we get the bound $\l \breve{g}' \r_{W^{2,\infty}} + \l \breve{\pi}' \r_{W^{1,\infty}} \leq D_0^{q+1}\eta\e$. We have proved that if $D_0$ is large compared to universal constants, and $\e$ is small compared to $D_0^{-1}$, $C_0^{-1}$ and universal constants, the parameter map $\Psi^{(p)}$ is well-defined. This concludes the proof of the lemma.
\end{proof}

\begin{lemma}\label{lem param contraction}
Consider $\pth{\wc{u},\wc{X}}\in H^2_{-q-\de}\times H^2_{-q-\de}$ satisfying \eqref{assumption u X}. If $D_0$ is large enough and if $\e$ is small enough compared to $D_0^{-1}$ and $C_0^{-1}$, then the parameter map $\Psi^{(p)}$ is a contraction for the distance $d$ defined in \eqref{def distance}.
\end{lemma}

\begin{proof}
Let $\p_1,\p_2\in B_{D_0}$ and let $\p'_i=\Psi^{(p)}(\p_i)$. By taking the difference of the equations defining $\p'_i$ we obtain
\begin{align}
16\pi \pth{\ga'_1 m'_1 - \ga'_2m'_2} & =    F_1\pth{\p_1,\wc{u},\wc{X}} - F_1\pth{\p_2,\wc{u},\wc{X}}, \label{eq diff m}
\\ 16\pi \pth{ \ga'_1 m'_1 (v'_1)^k - \ga'_2 m'_2 (v'_2)^k} & =    F_{2,k}\pth{\p_1,\wc{u},\wc{X}} - F_{2,k}\pth{\p_2,\wc{u},\wc{X}}, \label{eq diff v}
\\ \sqrt{3}\pth{ m'_1(y'_1)^{\ell'} - m'_2(y'_2)^{\ell'}} & =  F_{3,\ell'}\pth{\p_1,\wc{u},\wc{X}} - F_{3,\ell'}\pth{\p_2,\wc{u},\wc{X}}, \label{eq diff y}
\\ 8\pi\pth{ m'_1 (a'_1)^k - m'_2 (a'_2)^k} & =  F_{4,k}\pth{\p_1,\wc{u},\wc{X}} -  F_{4,k}\pth{\p_2,\wc{u},\wc{X}},\label{eq diff a}
\end{align}
and
\begin{align}
&\left\langle \breve{g}'_1 - \breve{g}'_2  , D\HH[e,0]^*( w_{j,\ell}) \right\rangle \label{eq diff g breve}
\\&\quad=  \left\langle \HH\pth{e + \chi_{\vec{p}\,'_2}(g_{\vec{p}\,'_2}-e) , \chi_{\vec{p}\,'_2}\pi_{\vec{p}\,'_2} } -  \HH\pth{e + \chi_{\vec{p}\,'_1}(g_{\vec{p}\,'_1}-e) , \chi_{\vec{p}\,'_1}\pi_{\vec{p}\,'_1} } , w_{j,\ell} \right\rangle \nonumber
\\&\quad\quad + F_{5,j,\ell}\pth{\p_1,\wc{u},\wc{X}} - F_{5,j,\ell}\pth{\p_2,\wc{u},\wc{X}},\nonumber
\\ & \left\langle  \breve{\pi}'_1 - \breve{\pi}'_2, D\MM[e ,0]^*(Z) \right\rangle \label{eq diff pi breve}
\\&\quad =  \left\langle \MM\pth{e + \chi_{\vec{p}\,'_2}(g_{\vec{p}\,'_2}-e) , \chi_{\vec{p}\,'_2}\pi_{\vec{p}\,'_2} } -  \MM\pth{e + \chi_{\vec{p}\,'_1}(g_{\vec{p}\,'_1}-e) , \chi_{\vec{p}\,'_1}\pi_{\vec{p}\,'_1} } , Z \right\rangle \nonumber
\\&\quad\quad    F_{6,Z}\pth{\p_1,\wc{u},\wc{X}} - F_{6,Z}\pth{\p_2,\wc{u},\wc{X}}.\nonumber
\end{align}

We start by deducing from \eqref{eq diff v} that
\begin{align*}
\eta^2 \left|(v'_1)^k - (v'_2)^k\right| & \lesssim \frac{\eta^2}{\ga'_1 m'_1} \left| \ga'_1 m'_1 - \ga'_2 m'_2\right| \left|(v'_2)^k\right|  +\frac{\eta^2}{\ga'_1 m'_1}\left| F_{2,k}\pth{\p_1,\wc{u},\wc{X}} - F_{2,k}\pth{\p_2,\wc{u},\wc{X}} \right|
\\&\lesssim_{D_0} \left| F_1\pth{\p_1,\wc{u},\wc{X}} - F_1\pth{\p_2,\wc{u},\wc{X}} \right| + \left| F_{2,k}\pth{\p_1,\wc{u},\wc{X}} - F_{2,k}\pth{\p_2,\wc{u},\wc{X}} \right|
\end{align*}
where we also used \eqref{eq diff m} and the estimates following from $\p\in B_{D_0}$ and the lower bound on the masses (recall also Remark \ref{remark gamma}). Using \eqref{diff F 1} this gives
\begin{align}
\eta^2 \left| v'_1-v'_2 \right| & \lesssim_{D_0,C_0}\eta d(\p_1,\p_2).\label{estim diff v}
\end{align}
We now deduce from \eqref{eq diff m} that
\begin{align}
\left| m'_1-m'_2 \right| & \lesssim \frac{m'_2}{\ga'_1}\left| \ga'_1- \ga'_2\right| + \left|  F_1\pth{\p_1,\wc{u},\wc{X}} - F_1\pth{\p_2,\wc{u},\wc{X}} \right| \nonumber
\\&\lesssim \eta^2 \left| v'_1-v'_2 \right| + C(D_0,C_0)\eta d(\p_1,\p_2)\nonumber
\\&\lesssim_{D_0,C_0}\eta d(\p_1,\p_2),\label{estim diff m}
\end{align}
where we used the fact that $x\mapsto \frac{1}{\sqrt{1-x^2}}$ is Lipschitz on $\left[0,1-\frac{\a}{2}\right]$, \eqref{estim diff v} and \eqref{diff F 1}. From \eqref{eq diff y} we get
\begin{align*}
\eta^2\left| (y'_1)^{\ell'} - (y'_2)^{\ell'}\right| & \lesssim \frac{\eta^2|y'_2|}{m'_1} \left| m'_1 - m'_2\right| + \frac{\eta^2}{m'_1} \left| F_{3,\ell'}\pth{\p_1,\wc{u},\wc{X}} - F_{3,\ell'}\pth{\p_2,\wc{u},\wc{X}}\right|
\\&\lesssim_{D_0} \pth{\frac{\e}{\eta}}^{2\th}\left| m'_1 - m'_2\right| + C(D_0,C_0)\eta^{1-2\th}\e^{2\th} d(\p_1,\p_2)
\end{align*}
where we have used \eqref{diff F 3}. Using \eqref{estim diff m} we obtain
\begin{align}
\eta^2\left| y'_1 - y'_2\right| & \lesssim_{D_0,C_0}\eta^{1-2\th}\e^{2\th}   d(\p_1,\p_2).\label{estim diff y}
\end{align}
Estimating $\left| a'_1-a'_2\right|$ from \eqref{eq diff a} is done identically, and we obtain
\begin{align}
\eta^2\left| a'_1 - a'_2\right| & \lesssim_{D_0,C_0}\eta^{1-2\th}\e^{2\th}  d(\p_1,\p_2).\label{estim diff a}
\end{align}
Finally, we note that the invertibility of the matrices 
\begin{align*}
\pth{ \left\langle h_{j,\ell}, D\HH[e,0]^*(w_{j',\ell'}) \right\rangle }_{3\leq j,j'\leq q, -(j-1)\leq \ell \leq j-1,-(j'-1)\leq \ell' \leq j'-1}
\end{align*}
and
\begin{align*}
\pth{ \left\langle \varpi_Z, D\MM[e,0]^*(Z') \right\rangle }_{(Z,Z')\in\mathcal{Z}_q^2}
\end{align*}
allows us to deduce from \eqref{eq diff g breve} and \eqref{eq diff pi breve} that
\begin{align*}
\l \breve{g}'_1 - \breve{g}'_2 \r_{W^{2,\infty}} &+ \l \breve{\pi}'_1 - \breve{\pi}'_2 \r_{W^{1,\infty}} 
\\& \lesssim \int_{\RRR^3} \left| \HH\pth{e + \chi_{\vec{p}\,'_2}(g_{\vec{p}\,'_2}-e) , \chi_{\vec{p}\,'_2}\pi_{\vec{p}\,'_2} } -  \HH\pth{e + \chi_{\vec{p}\,'_1}(g_{\vec{p}\,'_1}-e) , \chi_{\vec{p}\,'_1}\pi_{\vec{p}\,'_1} }  \right| r^{q-1} 
\\&\quad + \int_{\RRR^3} \left| \MM\pth{e + \chi_{\vec{p}\,'_2}(g_{\vec{p}\,'_2}-e) , \chi_{\vec{p}\,'_2}\pi_{\vec{p}\,'_2} } -  \MM\pth{e + \chi_{\vec{p}\,'_1}(g_{\vec{p}\,'_1}-e) , \chi_{\vec{p}\,'_1}\pi_{\vec{p}\,'_1} }  \right| r^{q-1} 
\\&\quad + C(D_0,C_0) \e d(\p_1,\p_2),
\end{align*}
where we have used \eqref{diff F 2}. Using now \eqref{diff constraint kerr} we obtain
\begin{align*}
\l \breve{g}'_1 - \breve{g}'_2 \r_{W^{2,\infty}} + \l \breve{\pi}'_1 - \breve{\pi}'_2 \r_{W^{1,\infty}} & \lesssim \pth{ |m'_1-m'_2|+\eta^2|v'_1-v'_2|} \int_{\{\min(\la_1,\la_2)\leq r \leq 2 \max(\la_1,\la_2)\}} r^{q-4}
\\&\quad + \eta^2 \pth{ |y'_1-y'_2| + |a'_1-a'_2| } \int_{\{\min(\la_1,\la_2)\leq r \leq 2 \max(\la_1,\la_2)\}} r^{q-5}
\\&\quad + C(D_0,C_0) \e d(\p_1,\p_2)
\\&\lesssim \pth{ |m'_1-m'_2|+\eta^2|v'_1-v'_2|} \max(\la_1,\la_2)^{q-1}
\\&\quad + \eta^2 \pth{ |y'_1-y'_2| + |a'_1-a'_2| } \max(\la_1,\la'_2)^{q-2}
\\&\quad + C(D_0,C_0) \e d(\p_1,\p_2).
\end{align*}
Using now \eqref{estim diff v}, \eqref{estim diff m}, \eqref{estim diff y}, \eqref{estim diff a} and $2(q-1)\th<1$ we obtain
\begin{align*}
\l \breve{g}'_1 - \breve{g}'_2 \r_{W^{2,\infty}} + \l \breve{\pi}'_1 - \breve{\pi}'_2 \r_{W^{1,\infty}} & \lesssim_{D_0,C_0}  \eta^{1-2(q-1)\th}\e^{2(q-1)\th}d(\p_1,\p_2) +  \e d(\p_1,\p_2)
\\& \lesssim_{D_0,C_0} \e d(\p_1,\p_2).
\end{align*}
Putting everything together, we obtain
\begin{align*}
d\pth{ \p'_1,\p'_2} & \lesssim_{D_0,C_0} \e d(\p_1,\p_2).
\end{align*}
Letting $\e$ be small compared to $D_0^{-1}$ and $C_0^{-1}$, we have proved that the parameter map $\Psi^{(p)}$ is a contraction.
\end{proof}

The Banach fixed point theorem and the fact that the system \eqref{system param} is equivalent to the orthogonality conditions \eqref{ortho 1}-\eqref{ortho 4} then imply the following statement on the parameter $\p\pth{\wc{u},\wc{X}}$.

\begin{corollary}\label{coro param}
Consider $\pth{\wc{u},\wc{X}}\in H^2_{-q-\de}\times H^2_{-q-\de}$ satisfying \eqref{assumption u X}. If $\e>0$ is small enough compared to $C_0^{-1}$, there exists a unique $\p\pth{\wc{u},\wc{X}}\in P$ such that the orthogonality conditions \eqref{ortho 1}-\eqref{ortho 4} hold and such that the following bounds hold
\begin{align*}
\eta^2 \lesssim m\pth{\wc{u},\wc{X}} \lesssim \eta^2, \qquad \left| y\pth{\wc{u},\wc{X}}\right| + \left|a\pth{\wc{u},\wc{X}}\right| \lesssim \pth{\frac{\e}{\eta}}^{2\th}, \qquad \left| v\pth{\wc{u},\wc{X}}\right|\leq 1 - \frac{\a}{2}, 
\end{align*}
and
\begin{align*}
\l \breve{g}\pth{\wc{u},\wc{X}} \r_{W^{2,\infty}} + \l \breve{\pi}\pth{\wc{u},\wc{X}} \r_{W^{1,\infty}} \lesssim \eta \e.
\end{align*}
\end{corollary}

\begin{remark}\label{remark q12 3}
Following Remarks \ref{remark q12 2}, we note that in the case $q=1$ the proof of Corollary \ref{coro param} considerably simplifies mainly because we don't need to solve equations \eqref{eq y} and \eqref{eq a}, and hence we don't need to divide by the mass.
\end{remark}

\subsubsection{A priori estimates on $\p\pth{\wc{u},\wc{X}}$}

In order to solve for $\wc{u}$ and $\wc{X}$, we will need to estimate how $\p\pth{\wc{u},\wc{X}}$ obtained in Corollary \ref{coro param} depends on $\pth{\wc{u},\wc{X}}$. This is provided by the next lemma.

\begin{lemma}\label{lem diff param}
Consider $\pth{\wc{u},\wc{X}},\pth{\wc{u'},\wc{X'}}\in H^2_{-q-\de}\times H^2_{-q-\de}$ both satisfying \eqref{assumption u X}. If $\e>0$ is small enough compared to $C_0^{-1}$, we have
\begin{align}
\left| m\pth{\wc{u},\wc{X}} - m\pth{\wc{u}',\wc{X}'} \right| &+ \eta^2 \left|v\pth{\wc{u},\wc{X}}-v\pth{\wc{u}',\wc{X}'}\right| \nonumber
\\& \lesssim_{C_0} \eta \pth{   \l \wc{u} - \wc{u}' \r_{H^2_{-q-\de}} + \l \wc{X} - \wc{X}' \r_{H^2_{-q-\de}}  },\label{estim diff m v lem}
\\ \eta^2 \Big( \left|y\pth{\wc{u},\wc{X}}-y\pth{\wc{u}',\wc{X}'}\right| &+ \left|a\pth{\wc{u},\wc{X}}-a\pth{\wc{u}',\wc{X}'}\right| \Big)\nonumber
\\& \lesssim_{C_0} \eta^{1-2\th}\e^{2\th} \pth{   \l \wc{u} - \wc{u}' \r_{H^2_{-q-\de}} + \l \wc{X} - \wc{X}' \r_{H^2_{-q-\de}}  } ,\label{estim diff y a lem}
\\\l \breve{g}\pth{\wc{u},\wc{X}} - \breve{g}\pth{\wc{u}',\wc{X}'} \r_{W^{2,\infty}} &+ \l \breve{\pi}\pth{\wc{u},\wc{X}} - \breve{\pi}\pth{\wc{u}',\wc{X}'} \r_{W^{1,\infty}} \nonumber
\\& \lesssim_{C_0}  \e \pth{   \l \wc{u} - \wc{u}' \r_{H^2_{-q-\de}} + \l \wc{X} - \wc{X}' \r_{H^2_{-q-\de}} } . \label{estim diff breve lem}
\end{align}
\end{lemma}

\begin{proof}
For simplicity, we use the notation $\p=\p\pth{\wc{u},\wc{X}}$ and $\p'=\p\pth{\wc{u}',\wc{X}'}$. We substract the two systems satisfied by $\p$ and $\p'$ and obtain
\begin{align}
16\pi \pth{\ga m - \ga'm'} & =  F_1\pth{\p,\wc{u},\wc{X}} - F_1\pth{\p',\wc{u}',\wc{X}'},\label{eq diff m 2}
\\ 16\pi\pth{\ga m v^k - \ga' m' (v')^k } & = F_{2,k}\pth{\p,\wc{u},\wc{X}} - F_{2,k}\pth{\p',\wc{u}',\wc{X}'},\label{eq diff v 2}
\\ \sqrt{3}\pth{my^{\ell'} - m'(y')^{\ell'}} & =   F_{3,\ell'}\pth{\p,\wc{u},\wc{X}} - F_{3,\ell'}\pth{\p',\wc{u}',\wc{X}'},\label{eq diff y 2}
\\ 8\pi\pth{ m a^k - m' (a')^k } & =  F_{4,k}\pth{\p,\wc{u},\wc{X}} - F_{4,k}\pth{\p',\wc{u}',\wc{X}'},\label{eq diff a 2}
\end{align}
and
\begin{align}
&\left\langle \breve{g} -\breve{g}' , D\HH[e,0]^*( w_{j,\ell}) \right\rangle \label{eq diff g breve 2}
\\&\quad  = \left\langle\HH\pth{e + \chi_{\vec{p}\,'}(g_{\vec{p}\,'}-e) , \chi_{\vec{p}\,'}\pi_{\vec{p}\,'} }- \HH\pth{e + \chi_{\vec{p}}(g_{\vec{p}}-e) , \chi_{\vec{p}}\pi_{\vec{p}} }  , w_{j,\ell} \right\rangle  \nonumber
\\& \quad\quad   + F_{5,j,\ell}\pth{\p,\wc{u},\wc{X}} - F_{5,j,\ell}\pth{\p',\wc{u}',\wc{X}'},\nonumber
\\ & \left\langle  \breve{\pi} - \breve{\pi}', D\MM[e ,0]^*(Z) \right\rangle  \label{eq diff pi breve 2}
\\&\quad =   \left\langle \MM\pth{e + \chi_{\vec{p}\,'}(g_{\vec{p}\,'}-e),\chi_{\vec{p}\,'}\pi_{\vec{p}\,'} } - \MM\pth{e + \chi_{\vec{p}}(g_{\vec{p}}-e),\chi_{\vec{p}}\pi_{\vec{p}} }, Z \right\rangle \nonumber
\\&\quad\quad  + F_{6,Z}\pth{\p,\wc{u},\wc{X}} - F_{6,Z}\pth{\p',\wc{u}',\wc{X}'}. \nonumber
\end{align}
As in the proof of Lemma \ref{lem param contraction}, we deduce from the system \eqref{eq diff m 2}-\eqref{eq diff pi breve 2} that
\begin{align}
| m - m'| + \eta^2 |v-v'| & \lesssim \sum_{i=1,2}\left| F_i\pth{\p,\wc{u},\wc{X}} - F_i\pth{\p',\wc{u}',\wc{X}'} \right|, \label{estim diff m v}
\\ \eta^2 \pth{ |y-y'| + |a-a'| }& \lesssim \pth{\frac{\e}{\eta}}^{2\th}\sum_{i=1,2}\left| F_i\pth{\p,\wc{u},\wc{X}} - F_i\pth{\p',\wc{u}',\wc{X}'} \right|  \label{estim diff y a}
\\&\quad + \sum_{i=3,4}\left| F_i\pth{\p,\wc{u},\wc{X}} - F_i\pth{\p',\wc{u}',\wc{X}'} \right|,\nonumber
\\ \l \breve{g} - \breve{g}' \r_{W^{2,\infty}} + \l \breve{\pi} - \breve{\pi}' \r_{W^{1,\infty}}  & \lesssim  \pth{\frac{\e}{\eta}}^{2(q-1)\th} \sum_{i=1,2}\left| F_i\pth{\p,\wc{u},\wc{X}} - F_i\pth{\p',\wc{u}',\wc{X}'} \right| \label{estim diff breve}
\\&\quad +  \pth{\frac{\e}{\eta}}^{2(q-2)\th} \sum_{i=3,4}\left| F_i\pth{\p,\wc{u},\wc{X}} - F_i\pth{\p',\wc{u}',\wc{X}'} \right|\nonumber
\\&\quad + \sum_{i=5,6} \left| F_i\pth{\p,\wc{u},\wc{X}} - F_i\pth{\p',\wc{u}',\wc{X}'}\right|.\nonumber
\end{align}
The triangle inequality implies for each $i=1,\dots,6$ that
\begin{align}\label{a priori inter 1}
&\left| F_i\pth{\p,\wc{u},\wc{X}} - F_i\pth{\p',\wc{u}',\wc{X}'}\right| 
\\&\quad \lesssim \left| F_i\pth{\p,\wc{u},\wc{X}} - F_i\pth{\p',\wc{u},\wc{X}}\right| + \left| F_i\pth{\p',\wc{u},\wc{X}} - F_i\pth{\p',\wc{u}',\wc{X}'}\right|.\nonumber
\end{align}
The first term in \eqref{a priori inter 1} can be estimated by \eqref{diff F 1}, \eqref{diff F 3} and \eqref{diff F 2}, depending on the value of $i$. We now estimate the second term. Thanks to the expression of the terms $F_i$, this amounts to estimating a term of the form\footnote{In this proof, the scalar product with $r^{j-1}$ is an abuse of notations standing for a scalar product with a scalar function $f_j$ satisfying $|f_j|+r|\nab f_j|\lesssim r^{j-1}$.}
\begin{align*}
\left\langle \RR_\OO\pth{ \p',\wc{u},\wc{X}} - \RR_\OO\pth{ \p',\wc{u}',\wc{X}'}, r^{j-1} \right\rangle
\end{align*}
for $1\leq j \leq q$ and $\OO$ denoting either $\HH$ or $\MM$. We use \eqref{diff remainders}
\begin{align*}
&\left| \left\langle \RR_\OO\pth{ \p',\wc{u},\wc{X}} - \RR_\OO\pth{ \p',\wc{u}',\wc{X}'}, r^{j-1} \right\rangle \right| 
\\&\quad \lesssim  \left| \int_{\RRR^3} r^{j-1} (\wc{u}-\wc{u}') \nab^2\wc{g} \right| + \eta^2 \l (1+r)^{j-q-\de-\frac{5}{2}} \r_{L^2} \pth{ \l \wc{u} - \wc{u}' \r_{H^2_{-q-\de}} + \l \wc{X} - \wc{X}' \r_{H^2_{-q-\de}} }
\\&\quad\quad + \pth{ \l \nab\wc{u} \r_{L^2_{-q-\de-1}} +  \l \nab\wc{X} \r_{L^2_{-q-\de-1}} +  \l \nab\wc{X}' \r_{L^2_{-q-\de-1}} }
\\&\hspace{3cm}\times\pth{ \l\nab( \wc{u} - \wc{u}') \r_{L^2_{-q-\de-1}} + \l \nab\pth{ \wc{X} - \wc{X}'} \r_{L^2_{-q-\de-1}} }
\\&\quad\quad + \l r^{-1} \wc{g} \r_{L^2} \pth{ \l\nab^2( \wc{u} - \wc{u}') \r_{L^2_{-q-\de-2}} + \l \nab^2\pth{ \wc{X} - \wc{X}'} \r_{L^2_{-q-\de-2}} }
\\&\quad\quad + \pth{ \l\nab\wc{g}\r_{L^2} + \l \wc{\pi} \r_{L^2} }\pth{ \l\nab( \wc{u} - \wc{u}') \r_{L^2_{-q-\de-1}} + \l \nab\pth{ \wc{X} - \wc{X}'} \r_{L^2_{-q-\de-1}} }
\\&\quad\quad + \pth{ \l \nab\wc{g}\r_{L^2}^2 + \l \wc{\pi} \r_{L^2}^2 } \l \wc{u} - \wc{u}' \r_{H^2_{-q-\de}} + \eta\e \pth{ \l \wc{u} - \wc{u}' \r_{H^2_{-q-\de}} + \l \wc{X} - \wc{X}' \r_{H^2_{-q-\de}} }.
\end{align*}
We integrate by parts in the integral involving $\nab^2\wc{g}$ and obtain
\begin{align*}
&\left| \left\langle \RR_\OO\pth{ \p',\wc{u},\wc{X}} - \RR_\OO\pth{ \p',\wc{u}',\wc{X}'}, r^{j-1} \right\rangle \right| 
\\&\quad \lesssim_{C_0} \eta \pth{ \l \wc{u} - \wc{u}' \r_{H^2_{-q-\de}} + \l \wc{X} - \wc{X}' \r_{H^2_{-q-\de}} } 
\\&\quad\quad +   \int_{\RRR^3} |\nab r^{j-1}| |\wc{u}-\wc{u}'| |\nab\wc{g}|  +   \int_{\RRR^3} r^{j-1} |\nab(\wc{u}-\wc{u}')| |\nab\wc{g}|
\\&\quad \lesssim_{C_0} \eta \pth{ \l \wc{u} - \wc{u}' \r_{H^2_{-q-\de}} + \l \wc{X} - \wc{X}' \r_{H^2_{-q-\de}} }.
\end{align*}
Together with \eqref{a priori inter 1}, \eqref{diff F 1}, \eqref{diff F 3} and \eqref{diff F 2}, this implies
\begin{align*}
&\sum_{i=1,2}\left| F_i\pth{\p,\wc{u},\wc{X}} - F_i\pth{\p',\wc{u}',\wc{X}'} \right|  \lesssim_{C_0} \eta d(\p,\p') + \eta \pth{ \l \wc{u} - \wc{u}' \r_{H^2_{-q-\de}} + \l \wc{X} - \wc{X}' \r_{H^2_{-q-\de}} },
\\ \sum_{i=3,4}&\left| F_i\pth{\p,\wc{u},\wc{X}} - F_i\pth{\p',\wc{u}',\wc{X}'} \right| 
\\&\quad \lesssim_{C_0} \eta^{1-2\th}\e^{2\th} d(\p,\p') + \eta \pth{ \l \wc{u} - \wc{u}' \r_{H^2_{-q-\de}} + \l \wc{X} - \wc{X}' \r_{H^2_{-q-\de}} },
\\ & \sum_{i=5,6}\left| F_i\pth{\p,\wc{u},\wc{X}} - F_i\pth{\p',\wc{u}',\wc{X}'} \right|  \lesssim_{C_0} \e d(\p,\p') + \eta \pth{ \l \wc{u} - \wc{u}' \r_{H^2_{-q-\de}} + \l \wc{X} - \wc{X}' \r_{H^2_{-q-\de}} }.
\end{align*}
Plugging this in \eqref{estim diff m v}, \eqref{estim diff y a} and \eqref{estim diff breve} we obtain
\begin{align}
| m - m'| + \eta^2 |v-v'| & \lesssim_{C_0} \eta \pth{ d(\p,\p') +   \l \wc{u} - \wc{u}' \r_{H^2_{-q-\de}} + \l \wc{X} - \wc{X}' \r_{H^2_{-q-\de}}  },\label{estim diff m v bis}
\\ \eta^2 \pth{ |y-y'| + |a-a'| }& \lesssim_{C_0} \eta^{1-2\th}\e^{2\th} \pth{  d(\p,\p') +  \l \wc{u} - \wc{u}' \r_{H^2_{-q-\de}} + \l \wc{X} - \wc{X}' \r_{H^2_{-q-\de}}  } ,\label{estim diff y a bis}
\\ \l \breve{g} - \breve{g}' \r_{W^{2,\infty}} + \l \breve{\pi} - \breve{\pi}' \r_{W^{1,\infty}}  & \lesssim_{C_0}  \e \pth{ d(\p,\p') +  \l \wc{u} - \wc{u}' \r_{H^2_{-q-\de}} + \l \wc{X} - \wc{X}' \r_{H^2_{-q-\de}} } . \label{estim diff breve bis}
\end{align}
We sum these inequalities and after taking $\e$ small enough compared to $C_0^{-1}$ we obtain
\begin{align*}
d(\p,\p') \lesssim_{C_0} \e \pth{ \l \wc{u} - \wc{u}' \r_{H^2_{-q-\de}} + \l \wc{X} - \wc{X}' \r_{H^2_{-q-\de}} }.
\end{align*}
We plug this into \eqref{estim diff m v bis}, \eqref{estim diff y a bis} and \eqref{estim diff breve bis} and finally obtain \eqref{estim diff m v lem}, \eqref{estim diff y a lem} and \eqref{estim diff breve lem}.
\end{proof}

\subsection{Solving for $\pth{\wc{u},\wc{X}}$}\label{section conclusion}

In this section, we solve \eqref{eq u X} with a fixed point argument. In order to get the appropriate decay for the solutions $u$ and $X$, we set the parameter $\p$ appearing in \eqref{eq u X} to be the one constructed in Section \ref{section def param}. Therefore, the system we wish to solve is
\begin{equation}
\left\{
\begin{aligned}
8\De \wc{u}  & =  A\pth{ \p\pth{\wc{u},\wc{X}} , \wc{u},\wc{X} },
\\  \De \wc{X}_i & =B\pth{ \p\pth{\wc{u},\wc{X}} , \wc{u},\wc{X} }_i.
\end{aligned}
\right.
\end{equation}

\begin{definition}
Let $C_0>0$. The solution map $\Psi^{(s)}$ is defined by
\begin{align*}
&\Psi^{(s)} : B_{\pth{ H^2_{-q-\de}}^4}(0,C_0\eta\e) \longrightarrow B_{\pth{ H^2_{-q-\de}}^4}(0,C_0\eta\e)
\\&\hspace{2.7cm} \pth{ \wc{u},\wc{X} } \longmapsto \pth{ \Psi^{(s)}_1\pth{ \wc{u},\wc{X} },\Psi^{(s)}_2\pth{ \wc{u},\wc{X} }}
\end{align*}
such that $\pth{ \Psi^{(s)}_1\pth{ \wc{u},\wc{X} },\Psi^{(s)}_2\pth{ \wc{u},\wc{X} }}$ solves the system
\begin{align}
8\De\Psi^{(s)}_1\pth{ \wc{u},\wc{X} } & =  A\pth{ \p\pth{\wc{u},\wc{X}} , \wc{u},\wc{X} }, \label{eq psi s 1}
\\  \De \Psi^{(s)}_2\pth{ \wc{u},\wc{X} }_i & =B\pth{ \p\pth{\wc{u},\wc{X}} , \wc{u},\wc{X} }_i,  \label{eq psi s 2}
\end{align}
where $\p\pth{\wc{u},\wc{X}}$ is defined in Corollary \ref{coro param}.
\end{definition}

\begin{lemma}\label{lem u X bound}
If $C_0$ is large enough and $\e$ is small compared to $C_0^{-1}$, the solution map $\Psi^{(s)}$ is well-defined.
\end{lemma}

\begin{proof}
In order to show that the solution map $\Psi^{(s)}$ is well-defined, we fix 
\begin{align*}
\pth{\wc{u},\wc{X}}\in B_{\pth{ H^2_{-q-\de}}^4}(0,C_0\eta\e)
\end{align*}
and solve \eqref{eq psi s 1}-\eqref{eq psi s 2}. First, note that if $C_0$ is large enough and $\e$ is small compared to $C_0^{-1}$, Corollary \ref{coro param} applies and the parameter $\p\pth{\wc{u},\wc{X}}$ exists and satisfies all the bounds of Corollary \ref{coro param}. For the sake of clarity, we simply denote $\p\pth{\wc{u},\wc{X}}$ by $\p$, $m\pth{\wc{u},\wc{X}}$ by $m$ and so on.

We want to apply Proposition \ref{prop laplace} to the system \eqref{eq psi s 1}-\eqref{eq psi s 2}. In order to do that, we only need to estimate the $L^2_{-q-\de-2}$ norm of $A\pth{ \p, \wc{u},\wc{X} }$ and $B\pth{ \p , \wc{u},\wc{X} }$. We have
\begin{align}
&\l A\pth{ \p , \wc{u},\wc{X} } \r_{L^2_{-q-\de-2}} + \l B\pth{ \p , \wc{u},\wc{X} } \r_{L^2_{-q-\de-2}} \label{L2 bound}
\\&\quad \lesssim \l \HH(\bar{g}(\p),\bar{\pi}(\p)) \r_{L^2_{-q-\de-2}} + \l \MM(\bar{g}(\p),\bar{\pi}(\p)) \r_{L^2_{-q-\de-2}} \nonumber
\\&\quad\quad + \l \RR_\HH\pth{\p,\wc{u},\wc{X}} \r_{L^2_{-q-\de-2}} + \l \RR_\MM\pth{\p,\wc{u},\wc{X}} \r_{L^2_{-q-\de-2}}.\nonumber
\end{align}
We start with the first two norms in \eqref{L2 bound}, using the expansions of Proposition \ref{prop main terms}. From \eqref{exp H} we obtain
\begin{align}
&\l \HH(\bar{g}(\p),\bar{\pi}(\p)) \r_{L^2_{-q-\de-2}} \label{L2 inter 1}
\\& \lesssim \l \HH\pth{e + \chi_{\vec{p}}(g_{\vec{p}}-e) , \chi_{\vec{p}}\pi_{\vec{p}} }  \r_{L^2_{-q-\de-2}} + \l D\HH[ e ,0]\pth{\breve{g}}+ \breve{g}^{ij} \De \wc{g}_{ij}  \r_{L^2_{-q-\de-2}} \nonumber
\\&\quad + \l D^2\HH[ e ,0]\pth{\pth{ \wc{g}  , \wc{\pi}  } , \pth{ \wc{g} , \wc{\pi}   }} \r_{L^2_{-q-\de-2}} + \l  \mathrm{Err}_1(\p) \r_{L^2_{-q-\de-2}}\nonumber.
\end{align}
We estimate each terms in \eqref{L2 inter 1}.
\begin{itemize}
\item For the first term, we use \eqref{constraint kerr}
\begin{align*}
\l \HH\pth{ e + \chi_{\vec{p}}(g_{\vec{p}} - e), \chi_{\vec{p}}\pi_{\vec{p}} } \r_{L^2_{-q-\de-2}} & \lesssim m \l r^{q+\de - \frac{5}{2}} \r_{L^2\pth{ \{ \la \leq r \leq 2 \la \} }}
\\& \lesssim m \la^{q+\de-1}
\\& \lesssim \eta^{2 - 2(q+\de-1)\th} \e^{2(q+\de-1)\th}
\\&\lesssim \eta \e,
\end{align*}
where we used the estimates of Corollary \ref{coro param} and $2(q+\de-1)\th<1$.
\item For the second term, we use the support property of $\breve{g}$ and \eqref{DH} and obtain
\begin{align*}
\l D\HH[ e ,0]\pth{\breve{g}}  +   \breve{g}^{ij} \De \wc{g}_{ij}  \r_{L^2_{-q-\de-2}} & \lesssim \l \breve{g} \r_{W^{2,\infty}}
\\&\lesssim \eta\e.
\end{align*}
\item For the third term, we have from \eqref{D2H}
\begin{align*}
\l D^2\HH[ e ,0]\pth{\pth{ \wc{g}  , \wc{\pi}  } , \pth{ \wc{g} , \wc{\pi}   }}  \r_{L^2_{-q-\de-2}} & \lesssim \l (1+r)^{q+\de + \half } \wc{g} \nab^2\wc{g} \r_{L^2}  + \l (1+r)^{q+\de + \half } (\nab\wc{g})^2 \r_{L^2} 
\\&\quad + \l (1+r)^{q+\de + \half } (\wc{\pi})^2 \r_{L^2}. 
\end{align*}
We use Hardy's inequality and the embedding $H^2_{-q-\de-2}\subset C^0_{-q-\de-2}$ from Lemma \ref{lem plongement} to get
\begin{align*}
\l (1+r)^{q+\de + \half } \wc{g} \nab^2\wc{g} \r_{L^2} & \lesssim \l r^{-1}\wc{g} \r_{L^2} \l (1+r)^{q+\de+\frac{3}{2}} \nab^2\wc{g} \r_{L^\infty}
\\&\lesssim \l \nab\wc{g}\r_{L^2} \l \nab^2\wc{g} \r_{H^2_{-q-\de-2}}
\\&\lesssim \eta \e.
\end{align*}
We use again the embedding $H^2_{-q-\de-1}\subset C^0_{-q-\de-1}$ from Lemma \ref{lem plongement} to get
\begin{align*}
\l (1+r)^{q+\de + \half } (\nab\wc{g})^2 \r_{L^2} &+ \l (1+r)^{q+\de + \half } (\wc{\pi})^2 \r_{L^2} 
\\& \lesssim \l \nab\wc{g} \r_{L^2} \l (1+r)^{q+\de+\half} \nab\wc{g} \r_{L^\infty}  + \l \wc{\pi} \r_{L^2} \l (1+r)^{q+\de+\half} \wc{\pi} \r_{L^\infty} 
\\&\lesssim  \l \nab\wc{g} \r_{L^2} \l \nab\wc{g} \r_{H^2_{-q-\de-1}}  + \l \wc{\pi} \r_{L^2} \l  \wc{\pi} \r_{H^2_{-q-\de-1}}
\\&\lesssim \eta\e. 
\end{align*}
We have proved that
\begin{align*}
\l D^2\HH[ e ,0]\pth{\pth{ \wc{g}  , \wc{\pi}  } , \pth{ \wc{g} , \wc{\pi}   }}  \r_{L^2_{-q-\de-2}} & \lesssim \eta\e.
\end{align*}
\item For the fourth term, we use \eqref{estim error} and the boundedness of $\wc{g}$, $\nab\wc{g}$ and $\wc{\pi}$ to obtain
\begin{align*}
\l  \mathrm{Err}_1(\p) \r_{L^2_{-q-\de-2}} & \lesssim \e m \l (1+r)^{-\frac{5}{2}} \r_{L^2}  + \e \l (1+r)^{\half} \pth{ |\nab\wc{g}|^2 + |\wc{\pi}|^2} \r_{L^2}  
\\&\quad +  \l \breve{g} \r_{W^{2,\infty}} + \l \breve{\pi} \r_{W^{1,\infty}} 
\\&\lesssim \eta\e.
\end{align*}
\end{itemize}
We have proved that $\l \HH(\bar{g}(\p),\bar{\pi}(\p)) \r_{L^2_{-q-\de-2}}\lesssim \eta\e$. Estimating $\l \MM(\bar{g}(\p),\bar{\pi}(\p)) \r_{L^2_{-q-\de-2}}$ is done in an identical way using \eqref{exp M} instead of \eqref{exp H}. We thus have
\begin{align}
\l \HH(\bar{g}(\p),\bar{\pi}(\p)) \r_{L^2_{-q-\de-2}} + \l \MM(\bar{g}(\p),\bar{\pi}(\p)) \r_{L^2_{-q-\de-2}} & \lesssim \eta \e. \label{L2 conclusion 1}
\end{align}
We now estimate the last two norms in \eqref{L2 bound} using \eqref{estim remainders}, the embedding $H^1_{-q-\de-1}\times H^1_{-q-\de-1} \subset L^2_{-q-\de-2}$ from Lemma \ref{lem plongement} and rough estimates like $(1+r)\pth{|\nab\wc{g}|+|\wc{\pi}|} \lesssim \e$:
\begin{align*}
&\l \RR_\HH\pth{\p,\wc{u},\wc{X}} \r_{L^2_{-q-\de-2}} + \l \RR_\MM\pth{\p,\wc{u},\wc{X}} \r_{L^2_{-q-\de-2}} 
\\&\qquad \lesssim \l \wc{u} \nab^2 \wc{g} \r_{L^2_{-q-\de-2}} + \l \pth{\nab \wc{X}}^2 \r_{L^2_{-q-\de-2}} +  \l \nab\wc{u}\nab\wc{X} \r_{L^2_{-q-\de-2}} 
\\&\qquad\quad +  \l   \pth{ \frac{m}{1+r} + |\wc{g}| }\nab^2\pth{\wc{u} + \wc{X}} \r_{L^2_{-q-\de-2}} 
\\&\qquad\quad +  \l \pth{ \frac{m}{(1+r)^2} + |\nab\wc{g}|+ |\wc{\pi}| } \nab\pth{ \wc{X} + \wc{u} } \r_{L^2_{-q-\de-2}} 
\\&\qquad\quad + \l  \pth{ \frac{m}{(1+r)^3}  + \pth{\frac{m}{(1+r)^2} + |\nab\wc{g}|+|\wc{\pi}|}^2}\wc{u} \r_{L^2_{-q-\de-2}}
\\&\qquad\quad + C\pth{ \l \breve{g} \r_{W^{2,\infty}}, \l \breve{\pi}\r_{W^{1,\infty}} }\pth{ \l \wc{u}\r_{H^2_{-q-\de}} + \l \wc{X} \r_{H^2_{-q-\de}} }
\\&\qquad \lesssim \e \l \wc{u}\r_{H^2_{-q-\de}} + \l \nab\wc{X}\r_{H^1_{-q-\de-1}}^2 + \l \nab\wc{X}\r_{H^1_{-q-\de-1}} \l \nab\wc{u}\r_{H^1_{-q-\de-1}}
\\&\qquad\quad + \pth{ \eta^2 + \e + \eta\e} \pth{ \l \wc{u}\r_{H^2_{-q-\de}}  + \l \wc{X}\r_{H^2_{-q-\de}} }
\\&\qquad \lesssim C(C_0) \eta\e^2.
\end{align*}
Together with \eqref{L2 bound} and \eqref{L2 conclusion 1} this implies 
\begin{align}
\l A\pth{ \p , \wc{u},\wc{X} } \r_{L^2_{-q-\de-2}} + \l B\pth{ \p , \wc{u},\wc{X} } \r_{L^2_{-q-\de-2}} & \lesssim   \eta\e \pth{ 1 + C(C_0)\e}.\label{L2 conclusion}
\end{align}
Since the parameter $\p$ has been constructed such that the orthogonality conditions \eqref{ortho 1}-\eqref{ortho 4} hold (recall Corollary \ref{coro param}), Proposition \ref{prop laplace} implies the existence of a unique couple 
\begin{align*}
\pth{ \Psi^{(s)}_1\pth{ \wc{u},\wc{X} },\Psi^{(s)}_2\pth{ \wc{u},\wc{X} }}\in H^2_{-q-\de}
\end{align*}
solving \eqref{eq psi s 1}-\eqref{eq psi s 2}. Moreover, thanks to \eqref{L2 conclusion} we have
\begin{align*}
\l \Psi^{(s)}_1\pth{ \wc{u},\wc{X} } \r_{H^2_{-q-\de}}  + \l \Psi^{(s)}_2\pth{ \wc{u},\wc{X} } \r_{H^2_{-q-\de}}  \lesssim  \eta\e \pth{ 1 + C(C_0)\e}.
\end{align*}
Therefore, if $C_0$ is large enough and if $\e$ is small compared to $C_0^{-1}$, we have
\begin{align*}
\pth{ \Psi^{(s)}_1\pth{ \wc{u},\wc{X} },\Psi^{(s)}_2\pth{ \wc{u},\wc{X} }}\in B_{\pth{ H^2_{-q-\de}}^4}(0,C_0\eta\e).
\end{align*}
This shows that the solution map $\Psi^{(s)}$ is well-defined. 
\end{proof}

\begin{lemma}\label{lem u X contraction}
If $C_0$ is large enough and $\e$ is small compared to $C_0^{-1}$, the solution map $\Psi^{(s)}$ is a contraction in $H^2_{-q-\de}$.
\end{lemma}

\begin{proof}
Let $\pth{\wc{u}_a,\wc{X}_a}$ and $\pth{\wc{u}_b,\wc{X}_b}$ both in $B_{\pth{ H^2_{-q-\de}}^4}(0,C_0\eta\e)$. We wish to estimate the difference 
\begin{align*}
\Psi^{(s)}_i\pth{ \wc{u}_a,\wc{X}_a } - \Psi^{(s)}_i\pth{ \wc{u}_b,\wc{X}_b }
\end{align*}
for $i=1,2$. For this, we write the equations satisfied by these differences:
\begin{align}
8\De\pth{ \Psi^{(s)}_1\pth{ \wc{u}_a,\wc{X}_a } - \Psi^{(s)}_1\pth{ \wc{u}_b,\wc{X}_b } } & = A\pth{ \p\pth{\wc{u}_a,\wc{X}_a} , \wc{u}_a,\wc{X}_a } - A\pth{ \p\pth{\wc{u}_b,\wc{X}_b} , \wc{u}_b,\wc{X}_b }, \label{eq diff 1}
\\ \De \pth{ \Psi^{(s)}_2\pth{ \wc{u}_a,\wc{X}_a } - \Psi^{(s)}_2\pth{ \wc{u}_b,\wc{X}_b } } & =  B\pth{ \p\pth{\wc{u}_a,\wc{X}_a} , \wc{u}_a,\wc{X}_a } - B\pth{ \p\pth{\wc{u}_b,\wc{X}_b} , \wc{u}_b,\wc{X}_b }.\label{eq diff 2}
\end{align}
For clarity, we denote $\p\pth{\wc{u}_a,\wc{X}_a}$ by $\p_a$ and $\p\pth{\wc{u}_b,\wc{X}_b}$ by $\p_b$. We wish to estimate the $L^2_{-q-\de-2}$ norm of the RHS of \eqref{eq diff 1} and \eqref{eq diff 2}. We have
\begin{align}
&\l A\pth{ \p_a , \wc{u}_a,\wc{X}_a } - A\pth{ \p_b , \wc{u}_b,\wc{X}_b } \r_{L^2_{-q-\de-2}} \label{L2 diff}
\\& \lesssim \l \HH\pth{ \bar{g}(\p_a),\bar{\pi}(\p_a)} - \HH\pth{ \bar{g}(\p_b),\bar{\pi}(\p_b)}  \r_{L^2_{-q-\de-2}} \nonumber
\\&\quad + \l \RR_\HH\pth{\p_a,\wc{u}_a,\wc{X}_a} - \RR_\HH\pth{\p_b,\wc{u}_b,\wc{X}_b}  \r_{L^2_{-q-\de-2}}.\nonumber
\end{align}
We start with the first term in \eqref{L2 diff} and use \eqref{exp H} to expand it:
\begin{align}
&\l \HH\pth{ \bar{g}(\p_a),\bar{\pi}(\p_a)} - \HH\pth{ \bar{g}(\p_b),\bar{\pi}(\p_b)}  \r_{L^2_{-q-\de-2}} \label{L2 diff 1}
\\&\qquad \lesssim \l \HH\pth{e + \chi_{\vec{p}_a}(g_{\vec{p}_a}-e) , \chi_{\vec{p}_a}\pi_{\vec{p}_a} } - \HH\pth{e + \chi_{\vec{p}_b}(g_{\vec{p}_b}-e) , \chi_{\vec{p}_b}\pi_{\vec{p}_b} } \r_{L^2_{-q-\de-2}} \nonumber
\\&\qquad\quad + \l D\HH[e,0]\pth{ \breve{g}_a-\breve{g}_b} + \De\wc{g}_{ij}\pth{ \breve{g}_a^{ij} - \breve{g}_b^{ij} } \r_{L^2_{-q-\de-2}} + \l \mathrm{Err}_1(\p_a) - \mathrm{Err}_1(\p_b) \r_{L^2_{-q-\de-2}} \nonumber
\end{align} 
We estimate each term in \eqref{L2 diff 1}.
\begin{itemize}
\item For the first term, we use \eqref{diff constraint kerr}
\begin{align*}
& \l \HH\pth{e + \chi_{\vec{p}_a}(g_{\vec{p}_a}-e) , \chi_{\vec{p}_a}\pi_{\vec{p}_a} } - \HH\pth{e + \chi_{\vec{p}_b}(g_{\vec{p}_b}-e) , \chi_{\vec{p}_b}\pi_{\vec{p}_b} } \r_{L^2_{-q-\de-2}}
\\&\qquad\lesssim \pth{ |m_a-m_b| + \eta^2 |v_a-v_b| } \l r^{q+\de-\frac{5}{2}} \r_{L^2\pth{\{\min(\la_1,\la_2)\leq r \leq 2 \max(\la_1,\la_2)\}}}
\\&\qquad\quad + \eta^2 \pth{ |y_a - y_b| + |a_a- a_b|} \l r^{q+\de-\frac{7}{2}} \r_{L^2\pth{\{\min(\la_1,\la_2)\leq r \leq 2 \max(\la_1,\la_2)\}}}
\\&\qquad \lesssim \pth{ |m_a-m_b| + \eta^2 |v_a-v_b| } \max(\la_a,\la_b)^{q+\de-1} 
\\&\qquad\quad + \eta^2 \pth{ |y_a - y_b| + |a_a- a_b|}\max(\la_a,\la_b)^{q+\de-2}
\\&\qquad \lesssim_{C_0} \eta^{1-2(q+\de-1)\th} \e^{2(q+\de-1)\th} \pth{ \l \wc{u}_a - \wc{u}_b \r_{H^2_{-q-\de}} + \l \wc{X}_a - \wc{X}_b \r_{H^2_{-q-\de}}   }
\\&\qquad \lesssim_{C_0}\e \pth{ \l \wc{u}_a - \wc{u}_b \r_{H^2_{-q-\de}} + \l \wc{X}_a - \wc{X}_b \r_{H^2_{-q-\de}}   },
\end{align*}
where we also used \eqref{estim diff m v lem}, \eqref{estim diff y a lem} and $2(q+\de-1)\th<1$.
\item For the second term, the support property of $\breve{g}_a-\breve{g}_b$ simply implies
\begin{align*}
&\l D\HH[e,0]\pth{ \breve{g}_a-\breve{g}_b} + \De\wc{g}_{ij}\pth{ \breve{g}_a^{ij} - \breve{g}_b^{ij} } \r_{L^2_{-q-\de-2}} 
\\&\quad \lesssim \l \breve{g}_a - \breve{g}_b \r_{W^{2,\infty}}
\\&\quad \lesssim_{C_0} \e \pth{ \l \wc{u}_a - \wc{u}_b \r_{H^2_{-q-\de}} + \l \wc{X}_a - \wc{X}_b \r_{H^2_{-q-\de}}   },
\end{align*}
where we used \eqref{estim diff breve lem}.
\item For the third term we use \eqref{diff error} and \eqref{estim T S}
\begin{align*}
\l \mathrm{Err}_1(\p_a) - \mathrm{Err}_1(\p_b) \r_{L^2_{-q-\de-2}} &\lesssim \pth{ \l \wc{g} \r_{H^2_{-q-\de}} +  \l \wc{\pi} \r_{H^1_{-q-\de}} } d(\p_a,\p_b) +  \eta d(\p_a,\p_b)
\\&\lesssim_{C_0} \e \pth{ \l \wc{u}_a - \wc{u}_b \r_{H^2_{-q-\de}} + \l \wc{X}_a - \wc{X}_b \r_{H^2_{-q-\de}}   },
\end{align*}
where we also used 
\begin{align}\label{estim dpp}
d(\p_a,\p_b) & \lesssim_{C_0} \e \pth{ \l \wc{u}_a - \wc{u}_b \r_{H^2_{-q-\de}} + \l \wc{X}_a - \wc{X}_b \r_{H^2_{-q-\de}}   },
\end{align}
which follows from \eqref{estim diff m v lem}, \eqref{estim diff y a lem} and \eqref{estim diff breve lem}.
\end{itemize}
Collecting the above estimates, we have proved that
\begin{align*}
\l \HH\pth{ \bar{g}(\p_a),\bar{\pi}(\p_a)} - \HH\pth{ \bar{g}(\p_b),\bar{\pi}(\p_b)}  \r_{L^2_{-q-\de-2}} \lesssim_{C_0} \e \pth{ \l \wc{u}_a - \wc{u}_b \r_{H^2_{-q-\de}} + \l \wc{X}_a - \wc{X}_b \r_{H^2_{-q-\de}}   }. 
\end{align*}
For the second term in \eqref{L2 diff} we use \eqref{diff remainders} and $\left| V_\OO^{abcd}\right| \lesssim 1$ to obtain
\begin{align*}
&\l \RR_\HH\pth{\p_a,\wc{u}_a,\wc{X}_a} - \RR_\HH\pth{\p_b,\wc{u}_b,\wc{X}_b}  \r_{L^2_{-q-\de-2}}
\\&\lesssim_{C_0} \pth{ \l \wc{u}_a \r_{H^2_{-q-\de}} + \l \wc{X}_a \r_{H^2_{-q-\de}} +  \l \wc{u}_b \r_{H^2_{-q-\de}} + \l \wc{X}_b \r_{H^2_{-q-\de}} } d(\p_a,\p_b)
\\&\quad + \eta^2  \pth{ \l \wc{u}_a - \wc{u}_b \r_{H^2_{-q-\de}} + \l \wc{X}_a - \wc{X}_b \r_{H^2_{-q-\de}}   }  
\\&\quad + \pth{ \l \wc{g}\r_{H^2_{-q-\de}} + \l \wc{\pi} \r_{H^1_{-q-\de-1}} + \l \wc{u}_a \r_{H^2_{-q-\de}} + \l \wc{X}_a \r_{H^2_{-q-\de}} +  \l \wc{u}_b \r_{H^2_{-q-\de}} + \l \wc{X}_b \r_{H^2_{-q-\de}} }
\\& \hspace{3cm}\times\pth{ \l \wc{u}_a - \wc{u}_b \r_{H^2_{-q-\de}} + \l \wc{X}_a - \wc{X}_b \r_{H^2_{-q-\de}}   }
\\&\quad + \eta\e \pth{ \l \wc{u}_a - \wc{u}_b \r_{H^2_{-q-\de}} + \l \wc{X}_a - \wc{X}_b \r_{H^2_{-q-\de}}   }  + \eta\e d(\p_a,\p_b)
\\&\lesssim_{C_0} \eta\e d(\p_a,\p_b) + \e \pth{ \l \wc{u}_a - \wc{u}_b \r_{H^2_{-q-\de}} + \l \wc{X}_a - \wc{X}_b \r_{H^2_{-q-\de}}   }
\\&\lesssim_{C_0} \e \pth{ \l \wc{u}_a - \wc{u}_b \r_{H^2_{-q-\de}} + \l \wc{X}_a - \wc{X}_b \r_{H^2_{-q-\de}}   },
\end{align*}
where we used the continuous embeddings $H^1_{-q-\de-1}\times H^1_{-q-\de-1}\subset L^2_{-q-\de-2}$ and $H^2_{-q-\de}\times L^2_{-q-\de-2}\subset L^2_{-q-\de-2}$ from Lemma \ref{lem plongement} and \eqref{estim dpp}. We have proved that
\begin{align*}
\l A\pth{ \p_a , \wc{u}_a,\wc{X}_a } - A\pth{ \p_b , \wc{u}_b,\wc{X}_b } \r_{L^2_{-q-\de-2}} & \lesssim_{C_0} \e \pth{ \l \wc{u}_a - \wc{u}_b \r_{H^2_{-q-\de}} + \l \wc{X}_a - \wc{X}_b \r_{H^2_{-q-\de}}   }.
\end{align*}
Repeating the argument, we can prove that
\begin{align*}
\l B\pth{ \p_a , \wc{u}_a,\wc{X}_a } - B\pth{ \p_b , \wc{u}_b,\wc{X}_b } \r_{L^2_{-q-\de-2}} & \lesssim_{C_0} \e \pth{ \l \wc{u}_a - \wc{u}_b \r_{H^2_{-q-\de}} + \l \wc{X}_a - \wc{X}_b \r_{H^2_{-q-\de}}   }.
\end{align*}
Applying Proposition \ref{prop laplace} to the system \eqref{eq diff 1}-\eqref{eq diff 2} and using the orthogonality conditions of Corollary \ref{coro param}, we obtain
\begin{align*}
\l \Psi^{(s)}\pth{ \wc{u}_a,\wc{X}_a} - \Psi^{(s)}\pth{ \wc{u}_b,\wc{X}_b} \r_{H^2_{-q-\de}} & \lesssim_{C_0} \e \pth{ \l \wc{u}_a - \wc{u}_b \r_{H^2_{-q-\de}} + \l \wc{X}_a - \wc{X}_b \r_{H^2_{-q-\de}}   }.
\end{align*}
Taking $\e$ small compared to $C_0^{-1}$ shows that the solution map $\Psi^{(s)}$ is a contraction for the distance in $H^2_{-q-\de}$.
\end{proof}

The Banach fixed point theorem then implies the following statement on the system \eqref{eq u X}.

\begin{corollary}\label{coro u X}
If $\e$ is small enough, there exists a unique solution $\wc{u}$ and $\wc{X}$ to the system \eqref{eq u X} with the bound
\begin{align*}
\l \wc{u} \r_{H^2_{-q-\de}} + \l \wc{X} \r_{H^2_{-q-\de}} \lesssim \eta \e.
\end{align*}
\end{corollary}

This concludes the proof of Theorem \ref{maintheorem}, since the initial data set 
\begin{align*}
\pth{ \pth{ 1+\wc{u}}^4 \bar{g}\pth{\p\pth{\wc{u},\wc{X}}} , \bar{\pi}\pth{\p\pth{\wc{u},\wc{X}}}+ L_e \wc{X}}
\end{align*}
solves the constraint equations on $\RRR^3$, and all the bounds stated in Theorem \ref{maintheorem} follow from Corollaries \ref{coro param} and \ref{coro u X}.

\appendix

\section{Appendix to Section \ref{section préliminaires}}\label{appendix section 2}

\subsection{Proof of Proposition \ref{prop laplace}}\label{appendix prop laplace}

We will use the following properties:
\begin{itemize}
\item[(i)] For $j\geq 1$, the functions $(w_{j,\ell})_{-(j-1)\leq \ell \leq j-1}$ form a basis of $\mathfrak{H}^{j-1}$ the set of harmonic polynomials defined on $\RRR^3$ which are homogeneous of degree $j-1$.
\item[(ii)] For $x=|x|\om$ and $y=|y|\om'$ with $\om,\om'\in\mathbb{S}^2$ and $|x|>|y|$ we have
\begin{align}\label{expansion potential}
\frac{1}{|x-y|} & =  \sum_{j=0}^\infty\frac{1}{2j+1} \frac{|y|^j}{|x|^{j+1}} \sum_{\ell=-j}^j (-1)^\ell  Y_{j,-\ell}(\om) Y_{j,\ell}(\om').
\end{align}
\end{itemize}

We start with the first part of Proposition \ref{prop laplace}, i.e. look at the scalar Laplace equation $\De u =f$. Note that the uniqueness of $\tilde{u}$ follows from the injectivity of $\De:H^2_{-q-\de}\longrightarrow L^2_{-q-\de-2}$ (see Theorem \ref{theo macowen}). It remains to show the existence of $\tilde{u}$. For each $j\geq 1$, we consider $G_{j}:\RRR_+\longrightarrow\RRR$ a smooth compactly supported (in $[0,1]$) function and vanishing in a neighborhood of zero and such that $\int_{\RRR_+}G_{j}(r) r^{j+1}\d r=\frac{1}{4\pi}$. By using the fact that $(w_{j',\ell'})_{1\leq j'\leq q, -(j'-1)\leq \ell'\leq j'-1}$ is a basis of $\bigcup_{i=0}^{q-1}\mathfrak{H}^i$, one can check that
\begin{align*}
f - \sum_{j=1}^q \sum_{\ell=-(j-1)}^{j-1} \langle f , w_{j,\ell} \rangle G_{j}(r)Y_{j-1,\ell}  \in \pth{ \bigcup_{i=0}^{q-1}\mathfrak{H}^i}^{\perp_{L^2}}.
\end{align*}
Moreover we have
\begin{align*}
\l f - \sum_{j=1}^q \sum_{\ell=-(j-1)}^{j-1} \langle f , w_{j,\ell} \rangle G_{j}(r)Y_{j-1,\ell} \r_{L^2_{-q-\de-2}} &\lesssim  \l f \r_{L^2_{-q-\de-2}} + \sum_{j=1}^q  \int_{\RRR^3} |f| (1+r)^{j-1}
\\& \lesssim \l f \r_{L^2_{-q-\de-2}},
\end{align*}
where we used 
\begin{align*}
\int_{\RRR^3} |f| (1+r)^{j-1} & \lesssim \int_{\RRR^3}|f| (1+r)^{q+\de+\half} (1+r)^{j- q-\de - \frac{3}{2}}
\\& \lesssim \l (1+r)^{j- q-\de - \frac{3}{2}} \r_{L^2} \l f \r_{L^2_{-q-\de-2}} 
\\& \lesssim  \l f \r_{L^2_{-q-\de-2}},
\end{align*}
if $j\leq q$ and $\de>0$. Therefore we can apply Theorem \ref{theo macowen} with $p=q$ and $\rho=-q-\de$ and obtain the existence of $v$ such that 
\begin{align*}
\De v = f - \sum_{j=1}^q \sum_{\ell=-(j-1)}^{j-1} \langle f , w_{j,\ell} \rangle G_{j}(r)Y_{j-1,\ell},
\end{align*}
with $\l v\r_{H^2_{-q-\de}}\lesssim\l f \r_{L^2_{-q-\de-2}}$. Now, for $j\geq 1$ and $-(j-1) \leq \ell \leq j-1$, we define
\begin{align*}
u_{j,\ell} = \frac{1}{4\pi}\pth{G_{j}Y_{j-1,\ell}}*\frac{1}{|\cdot|},
\end{align*}
which is such that $\De u_{j,\ell} = -G_{j}Y_{j-1,\ell}$. Let $x=r\om$ with $r>1$ and $\om\in\mathbb{S}^2$, since $G_j$ is supported in $[0,1]$, we can use \eqref{orthogonality} and \eqref{expansion potential} to obtain
\begin{align*}
u_{j,\ell}(x) & = \sum_{j'=0}^\infty \frac{1}{4\pi(2j'+1)r^{j'+1}} 
\\&\hspace{1cm} \times \sum_{\ell'=-j'}^{j'} \pth{ \int_0^1 G_{j}(r')(r')^{j'+2}\d r' } \pth{ \int_{\mathbb{S}^2} Y_{j-1,\ell} \pth{\om'}Y_{j',\ell'}(\om') } (-1)^{\ell'}  Y_{j',-\ell'}(\om) 
\\ & =   \frac{1}{(2j-1)r^{j}}  \pth{ \int_0^1 G_{j}(r')(r')^{j+1}\d r' }  (-1)^{\ell}  Y_{j-1,-\ell}(\om) 
\\& = -v_{j,\ell} ,
\end{align*}
where $v_{j,\ell}$ is introduced in Definition \ref{def tool laplace}. In particular, if we define $\tilde{u}_{j,\ell}=u_{j,\ell}+\chi(r)v_{j,\ell}$, then $\tilde{u}_{j,\ell}$ is actually compactly supported and we just need to show that it is twice differentiable with bounded second order derivatives. This simply follows from its definition and the facts that $\frac{1}{|\cdot|}$ is locally integrable and $G_j$ supported outside of 0. In conclusion, the function $u$ defined by
\begin{align*}
u = -  \sum_{j=1}^q \sum_{\ell=-(j-1)}^{j-1} \langle f , w_{j,\ell} \rangle u_{j,\ell} + v
\end{align*}
solves $\De u = f$ and rewrites as
\begin{align*}
u=  \chi(r) \sum_{j=1}^q \sum_{\ell=-(j-1)}^{j-1} \langle f , w_{j,\ell} \rangle v_{j,\ell}  + \tilde{u} ,
\end{align*}
where we defined
\begin{align*}
\tilde{u} =  - \sum_{j=1}^q \sum_{\ell=-(j-1)}^{j-1} \langle f , w_{j,\ell} \rangle  \tilde{u}_{j,\ell} + v.
\end{align*}
Also we have from $\l v\r_{H^2_{-q-\de}}\lesssim\l f \r_{L^2_{-q-\de-2}}$ and $|\langle f , w_{j,\ell} \rangle|\lesssim \l f \r_{L^2_{-q-\de-2}}$ that $\l \tilde{u} \r_{H^2_{-q-\de}}\lesssim \l f \r_{L^2_{-q-\de-2}}$. This concludes the proof of the first part of Proposition \ref{prop laplace}.

We now look at the second part of Proposition \ref{prop laplace}, i.e. the vector Laplace equation $\De X=Y$. Applying the first part of Proposition \ref{prop laplace} to the three equations $\De X_i  = Y_i$, we obtain the existence of a unique vector field $\tilde{X}$ satisfying $\l \tilde{X} \r_{H^2_{-q-\de}}\lesssim \l Y \r_{L^2_{-q-\de-2}}$ and such that 
\begin{align*}
X=  \chi(r) \sum_{j=1}^q \sum_{\ell=-(j-1)}^{j-1}\sum_{k=1,2,3} \langle Y , W_{j,\ell,k} \rangle  V_{j,\ell,k}  + \tilde{X} 
\end{align*}
solves $\De X = Y$. To conclude the proof, it remains to rearrange the terms corresponding to $j=2$. First, we recall the value of the $w_{2,\ell}$ functions
\begin{align*}
w_{2,-1} = \sqrt{3}x^2, \quad w_{2,0} = \sqrt{3}x^3, \quad w_{2,1} = \sqrt{3}x^1.
\end{align*}
Expressing each $x^i\dr_j$ in the basis 
\begin{align*}
\pth{ \Om^+_{11} , \Om^+_{22} , \Om^+_{33} , \Om^+_{12} , \Om^+_{13} , \Om^+_{23} , \Om^-_{12} , \Om^-_{13}, \Om^-_{23} }
\end{align*}
we obtain
\begin{align*}
W_{2,-1,1} & =  \frac{\sqrt{3}}{2} \pth{\Om^+_{12} - \Om^-_{12}},     & W_{2,0,1} & = \frac{\sqrt{3}}{2} \pth{\Om^+_{13} - \Om^-_{13}},  & W_{2,1,1} & = \frac{\sqrt{3}}{2} \Om^+_{11}, 
\\ W_{2,-1,2} & = \frac{\sqrt{3}}{2} \Om^+_{22},   & W_{2,0,2} & = \frac{\sqrt{3}}{2} \pth{\Om^+_{23} - \Om^-_{23}}, & W_{2,1,2} & = \frac{\sqrt{3}}{2} \pth{\Om^+_{12} + \Om^-_{12}},
\\ W_{2,-1,3} & =  \frac{\sqrt{3}}{2} \pth{\Om^+_{23} + \Om^-_{23}} , & W_{2,0,3} & = \frac{\sqrt{3}}{2} \Om^+_{33},  & W_{2,1,3} & = \frac{\sqrt{3}}{2} \pth{\Om^+_{13} + \Om^-_{13}}.
\end{align*}
This gives
\begin{align*}
\frac{2}{\sqrt{3}}&\sum_{\ell=-1,0,1}\sum_{k=1,2,3} \langle Y , W_{2,\ell,k} \rangle V_{2,\ell,k}
\\& = \langle Y ,  \Om^-_{12} \rangle \pth{- V_{2,-1,1} + V_{2,1,2}} +  \langle Y , \Om^-_{13} \rangle\pth{ - V_{2,0,1} + V_{2,1,3}} +  \langle Y , \Om^-_{23} \rangle \pth{ - V_{2,0,2} + V_{2,-1,3}  } 
\\&\quad + \langle Y , \Om^+_{12}  \rangle \pth{ V_{2,-1,1}  +  V_{2,1,2}} + \langle Y , \Om^+_{13}  \rangle \pth{ V_{2,0,1}  +  V_{2,1,3} } + \langle Y , \Om^+_{23}  \rangle \pth{ V_{2,0,2} +  V_{2,-1,3} }
\\&\quad + \langle Y , \Om^+_{11} \rangle V_{2,1,1} + \langle Y , \Om^+_{22} \rangle V_{2,-1,2}  + \langle Y , \Om^+_{33} \rangle V_{2,0,3} ,
\end{align*}
which concludes the proof of Proposition \ref{prop laplace}.

\subsection{Proof of Proposition \ref{prop KIDS}}\label{appendix proof KIDS}

\paragraph{Identification of the kernels.} We recall the main result of \cite{Moncrief1975}: if $(M,\g)$ is a solution to the Einstein vacuum equations and if $(g,\pi)$ are the induced data on a hypersurface $\Si$, then the kernel of $D\Phi[g,k]^*$ (with $\Phi=(\HH,\MM)$ being the full constraint operator) consists of all Killing fields of $(M,\g)$. More precisely, if $X$ is Killing field of $(M,\g)$, then $(h,Y)\in \ker \pth{D\Phi[g,k]^*} $ where the scalar $h$ and the vector field $Y$ are respectively the normal and tangential projection of $X$ to $\Si$. 

In the case of Minkowski spacetime, the 10-dimensional space of Killing fields is generated by the four translations, the three spatial rotations and the three Lorentz boosts. This proves that the kernels of $D\HH[e,0]^*$ and $D\MM[e,0]^*$ are given respectively by
\begin{align*}
\ker\pth{D\HH[e,0]^* } = \mathrm{Span}\pth{ w_{j,\ell},\; j=1,2, -(j-1)\leq \ell \leq j-1 }
\end{align*}
and
\begin{align*}
\ker\pth{D\MM[e,0]^* } = \mathrm{Span}\pth{ W_{1,0,1},W_{1,0,2},W_{1,0,3}, \Om^-_{12}, \Om^-_{13}, \Om^-_{23} }
\end{align*}
as stated.

\paragraph{Construction of the family $(h_{j,\ell})_{3\leq j \leq q, -(j-1)\leq \ell \leq j-1}$.} Let $\mu=\mu(r)$ be a smooth non-negative function compactly supported in $\{1\leq r \leq 10\}$. For $3\leq j \leq q$ and $-(j-1)\leq \ell \leq j-1$ we define
\begin{align*}
h_{j,\ell} & = \mu(r)D\HH[e,0]^*(w_{j,\ell}).
\end{align*}
By construction, $h_{j,\ell}$ is a traceless (since $\De w_{j,\ell}=0$) smooth and compactly supported symmetric 2-tensor. Assume by contradiction that the matrix 
\begin{align}\label{matrice h}
\pth{ \left\langle h_{j,\ell}, D\HH[e,0]^*(w_{j',\ell'}) \right\rangle }_{3\leq j,j'\leq q, -(j-1)\leq \ell \leq j-1,-(j'-1)\leq \ell' \leq j'-1}
\end{align}
is non-invertible. This implies that the columns of this matrix are linearly dependent, i.e. that there exists scalars $(\la_{j,\ell})_{3\leq j \leq q, -(j-1)\leq \ell \leq j-1}$ such that
\begin{align*}
\sum_{j'=3}^{q} \sum_{\ell'=-(j'-1)}^{j'-1} \la_{j',\ell'}  \left\langle \mu(r)D\HH[e,0]^*(w_{j,\ell}), D\HH[e,0]^*(w_{j',\ell'}) \right\rangle =0
\end{align*}
for all $(j,\ell)$ such that $3\leq j \leq q$ and $-(j-1)\leq \ell \leq j-1$ and such that they are not all vanishing. We multiply each of these equalities by $\la_{j,\ell}$ and sum over $(j,\ell)$ to obtain
\begin{align*}
\l \mu(r) \sum_{j=3}^{q} \sum_{\ell=-(j-1)}^{j-1}\la_{j,\ell} D\HH[e,0]^*(w_{j,\ell}) \r_{L^2} & = 0.
\end{align*}
Therefore, the sum $\sum_{j=3}^{q} \sum_{\ell=-(j-1)}^{j-1}\la_{j,\ell} D\HH[e,0]^*(w_{j,\ell})$ vanishes on $\{\mu\neq 0\}$ and since each $D\HH[e,0]^*(w_{j,\ell})$ is a polynomial, this implies that 
\begin{align*}
D\HH[e,0]^*\pth{ \sum_{j=3}^{q} \sum_{\ell=-(j-1)}^{j-1}\la_{j,\ell} w_{j,\ell} } & = 0.
\end{align*}
Since the kernel of $D\HH[e,0]^*$ is generated by $\{ w_{j,\ell},\; j=1,2, -(j-1)\leq \ell \leq j-1 \}$, the linear independence of the family $\{ w_{j,\ell},\; 1\leq j \leq q, -(j-1)\leq \ell \leq j-1 \}$ implies that $\la_{j,\ell}=0$ for all $3\leq j \leq q$ and $-(j-1)\leq \ell \leq j-1$, which contradicts our first assumption. This shows that the matrix \eqref{matrice h} is invertible. Moreover, since $\De w_{j,\ell}=0$ and $\nab\mu=\mu' \dr_r$ we have
\begin{align*}
(\div h_{j,\ell})_r & = \mu'(r) (\Hess w_{j,\ell})_{rr} 
\\& = (j-1)(j-2) \mu'(r) r^{j-3} Y_{j-1,\ell}. 
\end{align*}
The linear independence of the spherical harmonics thus implies the linear independence of the functions $(\div h_{j,\ell})_r $ and thus of the 1-forms $\div h_{j,\ell}$.

\paragraph{Construction of the family $(\varpi_Z)_{Z\in\mathcal{Z}_q}$.}  We simplify the notations and rename the elements of $\mathcal{Z}_q$ so that $\mathcal{Z}_q = \left\{ Z_1,\dots, Z_N \right\}$ with $N=3q^2-6$. We define some cutoffs functions:
\begin{itemize}
\item $\mu^{(\infty)}$ is simply defined by
\begin{align*}
\mu^{(\infty)} (r) = \mathbbm{1}_{[1,2]}(r).
\end{align*}
\item for $n\in\mathbb{N}^*$ and $1\leq k \leq N$ we define $\mu^{(n)}_k$ as a smooth function satisfying $0\leq \mu^{(n)}_k\leq 1$ and 
\begin{equation*}
\mu^{(n)}_k(r)  = \left\{
\begin{aligned}
&1 \qquad \text{if}\quad r\in \left[ 1+ \frac{1}{kn}, 2-\frac{1}{kn} \right],
\\& 0 \qquad \text{if}\quad r\notin \left[ 1+ \frac{1}{(k+1)n}, 2-\frac{1}{(k+1)n} \right].
\end{aligned}
\right. 
\end{equation*}
\end{itemize}
We also define the following matrices:
\begin{align*}
M^{(\infty)} &  = \pth{ \left\langle \mu^{(\infty)} D\MM[e,0]^*(Z_i), D\MM[e,0]^*(Z_j) \right\rangle  }_{1\leq i,j\leq N},
\\ M^{(n)} & = \pth{ \left\langle \mu^{(n)}_i D\MM[e,0]^*(Z_i), D\MM[e,0]^*(Z_j) \right\rangle  }_{1\leq i,j\leq N}.
\end{align*}
On the one hand, we can prove that $M^{(\infty)}$ is invertible as we did above for the matrix \eqref{matrice h} (the non-smoothness of $\mu^{(\infty)}$ doesn't affect the argument), ultimately using the description of the kernel of $D\MM[e,0]^*$ and the linear independence of the family $\left\{ W_{1,0,1},W_{1,0,2},W_{1,0,3}, \Om^-_{12}, \Om^-_{13}, \Om^-_{23}  \right\}\cup\mathcal{Z}_q$. On the other hand, by construction, for all $1\leq i \leq N$ we have
\begin{align*}
\mu^{(n)}_i \xrightarrow[n\to+\infty]{\enskip L^2\enskip} \mu^{(\infty)}.
\end{align*}
This implies the convergence
\begin{align*}
M^{(n)} \xrightarrow[n\to+\infty]{} M^{(\infty)}.
\end{align*}
The set of invertible matrices being an open set, this proves that $M^{(n)}$ is invertible for $n$ large enough. Thus we take $n$ large enough and define
\begin{align*}
\varpi_i & = \mu^{(n)}_i D\MM[e,0]^*(Z_i).
\end{align*}
Let us prove that the family $((\De\tr+\div\div)\varpi_i)_{1\leq i\leq N}$ is linearly independent. Consider scalar $\la_i$ such that
\begin{align}\label{nontrivial combinaison}
\sum_{i=1}^N\la_i(\De\tr+\div\div)\varpi_i =0.
\end{align}
Since the components in Cartesian coordinates of the vector fields $Z_i$ are harmonic functions, one can show 
\begin{align*}
(\De\tr+\div\div)\varpi_i & =  -(\mu^{(n)}_i)'\pth{ 3\dr_r\div Z_i +  \frac{3}{r} \div Z_i -  \frac{1}{r} \dr_r (Z_i)_r   }  -(\mu^{(n)}_i)''\pth{ \div Z_i +  \dr_r (Z_i)_r}.
\end{align*}
In particular, each $(\De\tr+\div\div)\varpi_i$ is supported in $\left\{(\mu^{(n)}_i)'\neq 0\right\}$, i.e. in 
\begin{align*}
S_i \vcentcolon = \left[ 1 + \frac{1}{(i+1)n},1+\frac{1}{in}\right] \cup\left[ 2- \frac{1}{in}, 2-\frac{1}{(i+1)n} \right].
\end{align*}
By construction, we have $\mathring{S}_i\cap \mathring{S}_j=\emptyset$ if $i\neq j$. Therefore, \eqref{nontrivial combinaison} implies that 
\begin{align}\label{relation}
\la_i \pth{ -(\mu^{(n)}_i)'\pth{ 3\dr_r\div Z_i +  \frac{3}{r} \div Z_i -  \frac{1}{r} \dr_r (Z_i)_r   }  -(\mu^{(n)}_i)''\pth{ \div Z_i +  \dr_r (Z_i)_r} } & = 0
\end{align}
identically on $S_i$ for all $1\leq i \leq N$. Assume that there exists $i$ such that $\la_i\neq 0$, then \eqref{relation} rewrites
\begin{align*}
Q_i(x)(\mu^{(n)}_i)''(r)+  P_i(x) (\mu^{(n)}_i)'(r) & = 0
\end{align*}
on $S_i$ with $P_i$ and $Q_i$ polynomials. Since $Q_i$ is a polynomial, there exists $\om_0\in\mathbb{S}^2$ such that $Q(r\om_0)\neq0$ for all $r>0$. Defining $f(r)=\frac{P_i(r\om_0)}{Q_i(r\om_0)}$ we obtain
\begin{align*}
(\mu^{(n)}_i)''(r)+  f(r) (\mu^{(n)}_i)'(r) & = 0
\end{align*}
on $S_i$ with $(\mu^{(n)}_i)'=0$ on $\dr S_i$ and $f$ smooth. By Cauchy-Lipschitz, this would imply $(\mu^{(n)}_i)'=0$ identically on $S_i$ which is a contradiction. Therefore we necessarily have
\begin{align*}
3\dr_r\div Z_i +  \frac{3}{r} \div Z_i -  \frac{1}{r} \dr_r (Z_i)_r & = 0,
\\ \div Z_i +  \dr_r (Z_i)_r & = 0,
\end{align*}
on $S_i$. Thus $\div Z_i$ solves the equation
\begin{align*}
\dr_r \pth{ r^{\frac{4}{3}} \div Z_i }  & = 0.
\end{align*}
Since $\div Z_i$ is a polynomial, this is only possible if $\div Z_i=0$ everywhere. From the second relation we then get $\dr_r (Z_i)_r=0$ which is a contradiction since
\begin{align*}
\dr_r (\Om^+_{ab})_r & = 2 \frac{x^a x^b}{r^2},
\\ \dr_r (W_{j,\ell,k})_r & =(j-1)r^{j-3} Y_{j-1,\ell} x^k.
\end{align*}
Therefore, $\la_i=0$ for all $1\leq i \leq N$ and the family $\pth{ (\De\tr+\div\div)\varpi_i }_{1\leq i \leq N}$ is linearly independent. This concludes the proof of Proposition \ref{prop KIDS}.

\subsection{Proof of Proposition \ref{prop kerr}}\label{appendix prop kerr}

\paragraph{Estimates on $\chi_{\vec{p}}$.} The cutoff $\chi_{\vec{p}}$ is defined in Definition \ref{def chi p} and it satisfies the following properties:
\begin{itemize}
\item[(i)] we have $\left| \chi_{\vec{p}} \right|\lesssim \mathbbm{1}_{\{\la\leq r \}}$ and if $n\geq 1$ we have
\begin{align}\label{estim cutoff}
r^n \left| \nab^n \chi_{\vec{p}} \right| \lesssim_n \mathbbm{1}_{\{ \la\leq r \leq 2 \la \}},
\end{align}
\item[(ii)] for all $n\geq 0$ we have
\begin{align}\label{diff cutoff}
r^n \left| \nab^n \pth{ \chi_{\vec{p}} - \chi_{\vec{p}\,'}} \right| \lesssim_n \frac{\left| \la - \la' \right|}{r}\mathbbm{1}_{\{ \min(\la,\la')\leq r \leq 2 \max(\la,\la') \}}.
\end{align}
\end{itemize}
We start with the proof of \eqref{estim cutoff}. From Faà di Bruno's formula and $|\nab^kr|\lesssim r^{1-k}$ we get
\begin{align*}
\left| \nab^n  \chi_{\vec{p}} \right| & \lesssim_n \sum_{k=1}^n r^{k-n} \la^{-n} \left| \chi^{(k)} \pth{\frac{r}{\la}} \right|.
\end{align*}
Since derivatives of $\chi$ are supported in $ \{ 1 \leq r \leq 2  \}$ we obtain
\begin{align*}
r^n\left| \nab^n  \chi_{\vec{p}} \right| & \lesssim_n  \pth{\frac{r}{\la}}^n \mathbbm{1}_{ \{ \la \leq r \leq 2 \la \} }.
\\&\lesssim_n \mathbbm{1}_{ \{ \la \leq r \leq 2 \la \} }.
\end{align*}

We now prove \eqref{diff cutoff}, assuming without loss of generality that $\la\leq \la'$. Let $n=0$. We can already assume that $\la\leq r \leq 2\la'$ otherwise $\chi_{\vec{p}} - \chi_{\vec{p}\,'}=0$. We distinguish several cases.
\begin{enumerate}
\item If $2\la \leq \la' \iff \frac{\la'}{2} \leq \la'-\la$. Thus, from $r\leq 2\la'$ we obtain $\frac{|\la-\la'|}{r}\geq \frac{\la'}{2r}\geq \frac{1}{4}$. Since $\left|  \chi_{\vec{p}} - \chi_{\vec{p}\,'}  \right|\leq 2$, we obtain $\left|  \chi_{\vec{p}} - \chi_{\vec{p}\,'}  \right|\lesssim \frac{|\la-\la'|}{r}$.
\item If $ \la'\leq 2 \la$. We have
\begin{align*}
\left|  \chi_{\vec{p}} - \chi_{\vec{p}\,'}  \right| & \lesssim \frac{r}{\la\la'}(\la'-\la)
\\&\lesssim \frac{|\la-\la'|}{r},
\end{align*}
where we used $\frac{r}{\la'}\leq 2$ and $\frac{r}{\la}\leq 4$.
\end{enumerate}
Since $ \chi_{\vec{p}} - \chi_{\vec{p}\,'} $ is supported in $\{ \min(\la,\la')\leq r \leq 2 \max(\la,\la') \}$, this concludes the proof of \eqref{diff cutoff} if $n=0$. For $n\geq 1$, we use again Faà di Bruno's formula and $|\nab^kr|\lesssim r^{1-k}$ to obtain
\begin{align*}
r^n\left| \nab^n\pth{ \chi_{\vec{p}}(r) - \chi_{\vec{p}\,'}(r) } \right| & \lesssim_n  \sum_{k=1}^n  \pth{\frac{r}{\la'}}^{k} \left| \chi^{(k)}\pth{\frac{r}{\la}} - \chi^{(k)}\pth{\frac{r}{\la'}} \right|
\\&\quad + \sum_{k=1}^n r^{k} \left| \la^{-k} - (\la')^{-k} \right| \left| \chi^{(k)}\pth{\frac{r}{\la}}\right|.
\end{align*}
For the first sum, we use the fact that $r\leq 2\la'$ and we apply the previous dichotomy argument to $\chi_{\vec{p}}^{(k)}(r) - \chi_{\vec{p}\,'}^{(k)}(r)$ instead of $\chi_{\vec{p}}(r) - \chi_{\vec{p}\,'}(r)$. We obtain
\begin{align*}
\left| \chi_{\vec{p}}^{(k)}(r) - \chi_{\vec{p}\,'}^{(k)}(r) \right| & \lesssim_k \frac{|\la-\la'|}{r}\mathbbm{1}_{\{ \min(\la,\la')\leq r \leq 2 \max(\la,\la') \}}.
\end{align*}
and thus
\begin{align*}
\sum_{k=1}^n  \pth{\frac{r}{\la'}}^{k} \left| \chi^{(k)}\pth{\frac{r}{\la}} - \chi^{(k)}\pth{\frac{r}{\la'}} \right| \lesssim_n  \frac{|\la-\la'|}{r}\mathbbm{1}_{\{ \min(\la,\la')\leq r \leq 2 \max(\la,\la') \}}.
\end{align*}
For the second sum, we use the fact that $\chi^{(k)}\pth{\frac{r}{\la}}$ with $k\geq 1$ is supported in $\{ \la \leq r \leq 2 \la \}$ to get
\begin{align*}
r^{k} \left| \la^{-k} - (\la')^{-k} \right|  & \lesssim_k \frac{|\la-\la'|}{r} \sum_{j=0}^{k-1}  \pth{\frac{r}{\la}}^j \pth{\frac{r}{\la'}}^{k-j}
\\& \lesssim_k \frac{|\la-\la'|}{r}.
\end{align*}
This concludes the proof of \eqref{diff cutoff}.

\paragraph{Estimates on $(g_{\vec{p}},\pi_{\vec{p}})$ and $(e+\chi_{\vec{p}}(g_{\vec{p}}-e),\chi_{\vec{p}}\pi_{\vec{p}})$.} For the proof of \eqref{estim kerr}, we start by estimating $\g_{m,|a|}-\m$ in the region $r\geq \la$. In this region, we have $\frac{|a|}{r}\leq \half$ and we deduce from the expression (see \cite{Visser2009})
\begin{align*}
\tilde{r} = \sqrt{ \frac{r^2-|a|^2 + \sqrt{(r^2-|a|^2)^2+4|a|^2z^2}}{2} }
\end{align*}
that $1 \lesssim \left| \nab^n\pth{ \frac{\tilde{r}}{r}}\right|\lesssim 1$. After projecting the metric on $\{t=0\}$ this implies
\begin{align}\label{estim kerr inter}
r^n \left| \nab^n (g_{\vec{p}_0}-e)  \right| + r^{n+1}\left| \nab^n  \pi_{\vec{p}_0} \right| & \lesssim_n \frac{m}{r},
\end{align}
where $\vec{p}_0=(m,0,|a|\dr_3,0)$. Considering now the change of coordinates \eqref{change of coord}, we use 
\begin{align*}
\ga^{-1} \lesssim \left| \,^{(v)}\La \right|_{\RRR^4\to \RRR^4} \lesssim \ga
\end{align*}
and that in the region $r\geq \la$ we have
\begin{align*}
\left| |x|^{-\ell} - |x-y|^{-\ell} \right| \lesssim |x|^{-\ell-1}. 
\end{align*}
This allows us to deduce from \eqref{estim kerr inter} that in the region $\{\la \leq r \}$ we have
\begin{align}\label{estim kerr preli}
r^n \left| \nab^n (g_{\vec{p}}-e)  \right| + r^{n+1}\left| \nab^n  \pi_{\vec{p}} \right| & \lesssim_{\ga,n} \frac{m}{r}.
\end{align}
Using moreover the expression of $\,^{(v)}\La$ (see (8.2) in \cite{Aretakis2023}) and the fact that $x\mapsto \frac{1}{\sqrt{1-x^2}}$ is locally Lipschitz we obtain
\begin{align*}
\left| \,^{(v)}\La - \,^{(v')}\La \right|_{\RRR^4\to \RRR^4}& \lesssim \left| \ga-\ga'\right| + \max(\ga,\ga') \left| v-v'\right|
\\& \lesssim_{\max(\ga,\ga')} \left| v-v'\right|.
\end{align*}
This implies that in the region $\{ \max(\la,\la')\leq r \}$ we have the following difference estimate
\begin{align}\label{diff kerr preli}
r^n \left| \nab^n \pth{ g_{\vec{p}} - g_{\vec{p}\,'}} \right| &+ r^{n+1} \left| \nab^n \pth{ \pi_{\vec{p}} - \pi_{\vec{p}\,'}} \right| 
\\& \lesssim_{\max(\ga,\ga'),n} \frac{\left| m -m' \right|}{r} + \max(m,m') \pth{ \frac{\left| y -y' \right| + \left| a -a' \right|}{r^2} + \frac{\left| v -v' \right|}{r}}.\nonumber
\end{align}
The estimate \eqref{estim kerr} then follows from putting \eqref{estim cutoff} and \eqref{estim kerr preli} together, and the estimate \eqref{diff kerr} then follows from putting \eqref{diff cutoff} and \eqref{diff kerr preli} together and using the bound
\begin{align*}
\left| \la- \la' \right| \lesssim \left| y- y' \right| + \left| a- a' \right|.
\end{align*}

\paragraph{The integral identities.} In order to prove \eqref{int kerr H 1}-\eqref{int kerr M x} we introduce the classical asymptotic charges for an asymptotically flat initial data set: 
\begin{align*}
\E(g,\pi) & \vcentcolon= \frac{1}{16\pi}\lim_{R\to+\infty}\int_{S_R} \nu\cdot \pth{ \div g - \nab\tr g} ,
\\ \P(g,\pi)_i & \vcentcolon= \frac{1}{8\pi} \lim_{R\to+\infty} \int_{S_R}  \pi_{\nu i} ,
\\ \J(g,\pi)_i & \vcentcolon= \frac{1}{8\pi} \lim_{R\to+\infty} \int_{S_R} \pi_{Y^i \nu} ,
\\ \C(g,\pi)_i & \vcentcolon=  \lim_{R\to+\infty} \int_{S_R} x^i \nu\cdot \pth{ \div g - \nab\tr g}  -  \lim_{R\to+\infty} \int_{S_R} (g_{\nu i} - \nu^i \tr g),
\end{align*}
where $S_r$ is the Cartesian sphere of radius $r$ and $\nu$ is the outward pointing unit normal to $S_r$. Following \cite{Chrusciel2003} and Section 8 of \cite{Aretakis2023} we find that the asymptotic charges of $(g_{\vec{p}},\pi_{\vec{p}})$ are given in terms of $\vec{p}$ by
\begin{align}\label{charges BH bis}
\pth{ \E,\P,\J,\C } (g_{\vec{p}},\pi_{\vec{p}}) =  \pth{ \ga m,-\ga m v, ma , my}.
\end{align}
We start with \eqref{int kerr H 1} and \eqref{int kerr M 1}, i.e. the identities involving vector fields growing like 1 at infinity. Thanks to Lemma \ref{lem Phi} and since $\Phi(e,0)=0$ we have
\begin{align}
\Phi\pth{ e +\chi_{\vec{p}}(g_{\vec{p}}-e),\chi_{\vec{p}}\pi_{\vec{p}}} & = D\Phi[e,0]\pth{ \chi_{\vec{p}}(g_{\vec{p}}-e),\chi_{\vec{p}}\pi_{\vec{p}} } + \RR_{\Phi,\mathrm{bh}}(\vec{p}), \label{def R bh}
\end{align}
where $\Phi$ denotes $\HH$ or $\MM$ and where the error term $\RR_{\Phi,\mathrm{bh}}(\vec{p})$ is schematically given by
\begin{align*}
\RR_{\Phi,\mathrm{bh}}(\vec{p}) & = \chi_{\vec{p}}(g_{\vec{p}}-e)\nab^2(\chi_{\vec{p}}(g_{\vec{p}}-e))  + (\nab (\chi_{\vec{p}}(g_{\vec{p}}-e)))^2   + (\chi_{\vec{p}}\pi_{\vec{p}})^2 
\\&\quad + \chi_{\vec{p}}(g_{\vec{p}}-e)\nab(\chi_{\vec{p}}\pi_{\vec{p}})     + \chi_{\vec{p}}(g_{\vec{p}}-e)(\nab (\chi_{\vec{p}}(g_{\vec{p}}-e)))^2 
\\&\quad  + (\chi_{\vec{p}}(g_{\vec{p}}-e)\nab (\chi_{\vec{p}}(g_{\vec{p}}-e)))^2  + (\chi_{\vec{p}}(g_{\vec{p}}-e)\chi_{\vec{p}}\pi_{\vec{p}})^2 
\\&\quad +  \chi_{\vec{p}}(g_{\vec{p}}-e) \chi_{\vec{p}}\pi_{\vec{p}} \nab (\chi_{\vec{p}}(g_{\vec{p}}-e))   + (\chi_{\vec{p}}(g_{\vec{p}}-e))^2 \chi_{\vec{p}}\pi_{\vec{p}} \nab (\chi_{\vec{p}}(g_{\vec{p}}-e)).
\end{align*}
Thanks to the assumption $m\leq 1$ and \eqref{estim kerr} we have
\begin{align}
\left| \RR_{\Phi,\mathrm{bh}}(\vec{p})\right| & \lesssim_\ga  \frac{m^2}{r^4}  . \label{estim R bh appendix}
\end{align}
Using $m,m'\leq 1$ and \eqref{diff kerr} we also obtain
\begin{align}
&\left| \RR_{\Phi,\mathrm{bh}}(\vec{p}) - \RR_{\Phi,\mathrm{bh}}(\vec{p}\,') \right| \label{diff R bh appendix}
\\&\qquad \lesssim_{\max(\ga,\ga')}  \frac{\max(m,m')}{r^3}\pth{  \frac{|m-m'|}{r} + \max(m,m')\pth{\frac{|y-y'|+|a-a'|}{r^2}  + \frac{ | v-v'|}{r}} }. \nonumber
\end{align}
From \eqref{def R bh} and the fact that $\Phi\pth{ e +\chi_{\vec{p}}(g_{\vec{p}}-e),\chi_{\vec{p}}\pi_{\vec{p}}}$ is supported in $\{ \la \leq r \leq 2 \la \}$ (since $\Phi(e,0)=0$ and $\Phi(g_{\vec{p}},\pi_{\vec{p}})=0$ away from the center of mass $y$) we obtain
\begin{align*}
\left\langle \HH\pth{ e + \chi_{\vec{p}}(g_{\vec{p}} - e),\chi_{\vec{p}} \pi_{\vec{p}}  }, 1 \right\rangle & = \left\langle D\HH[e,0]\pth{ \chi_{\vec{p}}(g_{\vec{p}}-e) } \mathbbm{1}_{\{\la\leq r \}} , 1 \right\rangle 
\\&\quad + \left\langle \RR_{\HH,\mathrm{bh}}(\vec{p})\mathbbm{1}_{\{\la\leq r \}}, 1 \right\rangle,
\\  \left\langle \MM\pth{ e + \chi_{\vec{p}}(g_{\vec{p}} - e),\chi_{\vec{p}} \pi_{\vec{p}}  }, W_{1,0,k} \right\rangle & =  \left\langle D\MM[e,0]\pth{ \chi_{\vec{p}}\pi_{\vec{p}} } \mathbbm{1}_{\{\la\leq r \}} , W_{1,0,k} \right\rangle 
\\&\quad +  \left\langle \RR_{\MM,\mathrm{bh}}(\vec{p}) \mathbbm{1}_{\{\la\leq r \}} , W_{1,0,k} \right\rangle.
\end{align*}
Thanks to \eqref{DH} and Stokes theorem we have
\begin{align*}
&\left\langle D\HH[e,0]\pth{ \chi_{\vec{p}}(g_{\vec{p}}-e) } \mathbbm{1}_{\{\la\leq r \}} , 1 \right\rangle 
\\&\quad = \lim_{R\to +\infty} \int_{\{ \la \leq r \leq R \}} \pth{ -\De\tr \pth{ \chi_{\vec{p}}(g_{\vec{p}} - e) } + \div\div\pth{ \chi_{\vec{p}}(g_{\vec{p}} - e) } }
\\&\quad = \lim_{R\to +\infty} \int_{\dr\{ \la \leq r \leq R \}}\nu\cdot \pth{ -\nab \tr \pth{ \chi_{\vec{p}}(g_{\vec{p}} - e) } + \div\pth{ \chi_{\vec{p}}(g_{\vec{p}} - e) } },
\end{align*} 
where $\nu$ denotes the outward pointing unit normal to $\dr\{ \la \leq r \leq R \}$. Now, note that $\dr\{ \la \leq r \leq R \}=S_\la \cup S_R$ and that $\chi_{\vec{p}}=0$ on $S_\la$, $\chi_{\vec{p}}=1$ on $S_R$ (since $R\geq 2\la$) and $\nu\chi_{\vec{p}}=0$ on $S_\la \cup S_R$. This implies
\begin{align*}
\left\langle D\HH[e,0]\pth{ \chi_{\vec{p}}(g_{\vec{p}}-e) } \mathbbm{1}_{\{\la\leq r \}} , 1 \right\rangle  & = \lim_{R\to +\infty} \int_{S_R}\nu\cdot \pth{ -\nab \tr  g_{\vec{p}}  + \div g_{\vec{p}}  }.
\end{align*}
We recognize the expression of $16\pi\E(g_{\vec{p}},\pi_{\vec{p}})$ and obtain 
\begin{align*}
\left\langle \HH\pth{ e + \chi_{\vec{p}}(g_{\vec{p}} - e),\chi_{\vec{p}} \pi_{\vec{p}}  }, 1 \right\rangle & = 16\pi \ga m +  \left\langle \RR_{\HH,\mathrm{bh}}(\vec{p})\mathbbm{1}_{\{\la\leq r \}}, 1 \right\rangle.
\end{align*}
We define $\RR_{\HH,1}(\vec{p})\vcentcolon =  \left\langle \RR_{\HH,\mathrm{bh}}(\vec{p})\mathbbm{1}_{\{\la\leq r \}}, 1 \right\rangle$. The estimates \eqref{estim R bh appendix}-\eqref{diff R bh appendix} imply 
\begin{align*}
\left| \RR_{\HH,1}(\vec{p}) \right|  \lesssim_\ga m^2
\end{align*}
and
\begin{align*}
&\left| \RR_{\HH,1}(\vec{p}) - \RR_{\HH,1}(\vec{p}\,')\right| 
\\&\quad \lesssim \left| \left\langle \pth{ \RR_{\HH,\mathrm{bh}}(\vec{p}) - \RR_{\HH,\mathrm{bh}}(\vec{p}\,')} \mathbbm{1}_{\{\la\leq r \}} , 1 \right\rangle \right|  + \left| \left\langle \RR_{\HH,\mathrm{bh}}(\vec{p}\,')\pth{\mathbbm{1}_{\{\la\leq r \}} - \mathbbm{1}_{\{\la'\leq r \}} }, 1 \right\rangle \right|
\\&\quad \lesssim_{\max(\ga,\ga')} \max(m,m')\pth{ |m-m'| + \max(m,m')\pth{ |y-y'| + |a-a'| + |v-v'| }} 
\\& \quad + (m')^2|\la-\la'|
\\&\quad \lesssim_{\max(\ga,\ga')} \max(m,m')\pth{ |m-m'| + \max(m,m')\pth{ |y-y'| + |a-a'| + |v-v'| }} , 
\end{align*}
where we used the fact that $(1+r)^{-4}$ is integrable on $\RRR^3$. Moreover, thanks to \eqref{DM} and to the following version of Stokes theorem
 \begin{align*}
\int_{\Om} X \cdot \div T & = \int_{\dr \Om} T_{X\nu} - \half \int_{\Om} \nabla\otimes X \cdot T,
\end{align*}
we have
\begin{align*}
&\left\langle D\MM[e,0]\pth{ \chi_{\vec{p}}\pi_{\vec{p}} } \mathbbm{1}_{\{\la\leq r \}} , W_{1,0,k} \right\rangle = \lim_{R\to +\infty}\int_{S_R} (\pi_{\vec{p}})_{k\nu},
\end{align*}
where we recall from Proposition \ref{prop KIDS} that $\nabla\otimes W_{1,0,k} = 0$. We recognize the expression of $8\pi\P(g_{\vec{p}},\pi_{\vec{p}})_k$ and obtain
\begin{align*}
\left\langle \MM\pth{ e + \chi_{\vec{p}}(g_{\vec{p}} - e),\chi_{\vec{p}} \pi_{\vec{p}}  }, W_{1,0,k} \right\rangle & = -8\pi \ga m v^k +   \left\langle \RR_{\MM,\mathrm{bh}}(\vec{p}) \mathbbm{1}_{\{\la\leq r \}} , W_{1,0,k} \right\rangle.
\end{align*}
We define $\RR_{\MM,W_{1,0,k}}(\vec{p})\vcentcolon =  \left\langle \RR_{\MM,\mathrm{bh}}(\vec{p}) \mathbbm{1}_{\{\la\leq r \}} , W_{1,0,k} \right\rangle$ and \eqref{estim R bh appendix}-\eqref{diff R bh appendix} again imply
\begin{align*}
\left| \RR_{\MM,W_{1,0,k}}(\vec{p}) \right| & \lesssim_\ga m^2 
\end{align*}
and
\begin{align*}
&\left| \RR_{\MM,W_{1,0,k}}(\vec{p}) - \RR_{\MM,W_{1,0,k}}(\vec{p}\,')\right| 
\\&\hspace{0.5cm} \lesssim_{\max(\ga,\ga')} \max(m,m')\pth{ |m-m'| + \max(m,m')\pth{ |y-y'| + |a-a'| + |v-v'| }} .
\end{align*}
This concludes the proof of \eqref{int kerr H 1} and \eqref{int kerr M 1}. 

We now turn to the proof of \eqref{int kerr H x} and \eqref{int kerr M x}, i.e. the identities involving vector fields growing like $r$ at infinity. Compared to the proof of \eqref{int kerr H 1} and \eqref{int kerr M 1}, we need to isolate the quadratic terms in the expansion of the constraint operators:
\begin{align}
&\left\langle \HH\pth{ e+ \chi_{\vec{p}}(g_{\vec{p}} -e),\chi_{\vec{p}}\pi_{\vec{p}}} , x^k \right\rangle \nonumber
\\& = \left\langle D\HH[e,0]\pth{\chi_{\vec{p}}(g_{\vec{p}} -e)} \mathbbm{1}_{\{\la\leq r\}} , x^k \right\rangle\label{lol}
\\&\quad + \underbrace{\left\langle \half D^2\HH[e,0]\pth{ \pth{ \chi_{\vec{p}}(g_{\vec{p}} -e),\chi_{\vec{p}}\pi_{\vec{p}} },\pth{ \chi_{\vec{p}}(g_{\vec{p}} -e),\chi_{\vec{p}}\pi_{\vec{p}} } } \mathbbm{1}_{\{\la\leq r\}}, x^k \right\rangle}_{=\vcentcolon Q^k(\vec{p})}\non
\\&\quad + \left\langle \mathrm{Err}\HH[e,0]\pth{ \chi_{\vec{p}}(g_{\vec{p}} -e),\chi_{\vec{p}}\pi_{\vec{p}} } \mathbbm{1}_{\{\la\leq r\}}, x^k \right\rangle\non.
\end{align}
As above we can show that
\begin{align*}
\left\langle D\HH[e,0]\pth{\chi_{\vec{p}}(g_{\vec{p}} -e)} \mathbbm{1}_{\{\la\leq r\}} , x^k \right\rangle & = m y^k,
\end{align*}
and using Lemma \ref{lem Phi}, \eqref{estim kerr} and \eqref{diff kerr} we can also show that
\begin{align*}
\left| \left\langle \mathrm{Err}\HH[e,0]\pth{ \chi_{\vec{p}}(g_{\vec{p}} -e),\chi_{\vec{p}}\pi_{\vec{p}} } \mathbbm{1}_{\{\la\leq r\}}, x^k \right\rangle \right| \lesssim m^3
\end{align*}
and
\begin{align*}
&\left| \left\langle \mathrm{Err}\HH[e,0]\pth{ \chi_{\vec{p}}(g_{\vec{p}} -e),\chi_{\vec{p}}\pi_{\vec{p}} } \mathbbm{1}_{\{\la\leq r\}} - \mathrm{Err}\HH[e,0]\pth{ \chi_{\vec{p}\,'}(g_{\vec{p}\,'} -e),\chi_{\vec{p}\,'}\pi_{\vec{p}\,'} } \mathbbm{1}_{\{\la'\leq r\}}, x^k \right\rangle  \right| 
\\&\hspace{2.5cm}\lesssim \max(m,m')^2\pth{ |m-m'| + \max(m,m')\pth{ |y-y'| + |a-a'| + |v-v'| }}.
\end{align*}
It remains to study the quadratic term $Q^k(\vec{p})$. A direct estimate with \eqref{estim kerr} and \eqref{D2H} would imply 
\begin{align*}
\left| \half D^2\HH[e,0]\pth{ \pth{ \chi_{\vec{p}}(g_{\vec{p}} -e),\chi_{\vec{p}}\pi_{\vec{p}} },\pth{ \chi_{\vec{p}}(g_{\vec{p}} -e),\chi_{\vec{p}}\pi_{\vec{p}} } } \right| \lesssim \frac{m^2}{r^4}
\end{align*}
so that the integral defining $Q^k(\vec{p})$ \text{a priori} does not converge (recall that $x^k\sim r$ at infinity). More precisely, we introduce 
\begin{align*}
\widehat{Q}^k_r(\vec{p}) \vcentcolon= \int_{S_r} \half D^2\HH[e,0]\pth{ \pth{ \chi_{\vec{p}}(g_{\vec{p}} -e),\chi_{\vec{p}}\pi_{\vec{p}} },\pth{ \chi_{\vec{p}}(g_{\vec{p}} -e),\chi_{\vec{p}}\pi_{\vec{p}} } } x^k
\end{align*}
so that $Q^k(\vec{p})=\int_{r\geq \la} \widehat{Q}^k_r(\vec{p}) \d r$. Using the expression of $(g_{\vec{p}},\pi_{\vec{p}})$ we obtain 
\begin{align*}
\widehat{Q}^k_r(\vec{p}) = \frac{c^k(\vec{p})}{r} + \GO{m^2 r^{-2}},
\end{align*}
for a constant $c^k(\vec{p})$ depending only on the parameter $\vec{p}$. If $c^k(\vec{p})\neq 0$, then $Q^k(\vec{p})=\pm\infty$ which contradicts the fact that all the other terms in \eqref{lol} are finite. We have thus obtained $c^k(\vec{p})=0$ and benefiting from this cancellation we can easily show that
\begin{align*}
\left| Q^k(\vec{p}) \right| \lesssim m^2
\end{align*}  
and
\begin{align*}
\left| Q^k(\vec{p}) - Q^k(\vec{p}\,') \right| \lesssim \max(m,m')\pth{ |m-m'| + \max(m,m')\pth{ |y-y'| + |a-a'| + |v-v'| }}.
\end{align*}
Defining 
\begin{align*}
\RR_{\HH,x^k}(\vec{p}) \vcentcolon = Q^k(\vec{p}) + \left\langle \mathrm{Err}\HH[e,0]\pth{ \chi_{\vec{p}}(g_{\vec{p}} -e),\chi_{\vec{p}}\pi_{\vec{p}} } \mathbbm{1}_{\{\la\leq r\}}, x^k \right\rangle,
\end{align*}
and collecting the above estimates, this concludes the proof of \eqref{int kerr H x}. The proof of \eqref{int kerr M x} is identical.

\paragraph{Pointwise estimates on the constraint operators.} Using the fact that $\Phi(e,0)=0$ and $\Phi(g_{\vec{p}},\pi_{\vec{p}})=0$ away from the center of mass $y$ and the schematic expression from Lemma \ref{lem Phi} we obtain
\begin{align*}
&\Phi\pth{e + \chi_{\vec{p}}(g_{\vec{p}}-e), \chi_{\vec{p}}\pi_{\vec{p}} } 
\\&\qquad =  \chi_{\vec{p}}g_{\vec{p}}(g_{\vec{p}} - e) \nab^2\chi_{\vec{p}} +  \chi_{\vec{p}}g_{\vec{p}} \nab\chi_{\vec{p}}\nab g_{\vec{p}}  + (1-\chi_{\vec{p}})e \nab^2(\chi_{\vec{p}}(g_{\vec{p}} - e)) 
\\&\qquad\quad + \pth{ (1-\chi_{\vec{p}})e  \nab( \chi_{\vec{p}}(g_{\vec{p}} - e))}^2  +  (1-\chi_{\vec{p}})e  \chi_{\vec{p}}g_{\vec{p}} \pth{\nab( \chi_{\vec{p}}(g_{\vec{p}} - e))}^2
\\&\qquad\quad + \chi_{\vec{p}}^2g_{\vec{p}}^2(g_{\vec{p}} - e)^2\pth{ \nab\chi_{\vec{p}}   }^2 +  \chi_{\vec{p}}^3g_{\vec{p}}^2(g_{\vec{p}} - e) \nab\chi_{\vec{p}}  \nab g_{\vec{p}}   
\\&\qquad\quad + (1-\chi_{\vec{p}})e \nab( \chi_{\vec{p}}\pi_{\vec{p}}) +  \chi_{\vec{p}}g_{\vec{p}}\pi_{\vec{p}} \nab\chi_{\vec{p}} 
\\&\qquad\quad + ((1-\chi_{\vec{p}})e)^2 \chi_{\vec{p}}\pi_{\vec{p}} \nab(\chi_{\vec{p}}(g_{\vec{p}} - e)) + (1-\chi_{\vec{p}})e \chi_{\vec{p}}^2g_{\vec{p}} \pi_{\vec{p}} \nab(\chi_{\vec{p}}(g_{\vec{p}} - e))
\\&\qquad\quad + \chi_{\vec{p}}^3g_{\vec{p}}^2 \pi_{\vec{p}} (g_{\vec{p}} - e)\nab\chi_{\vec{p}}  +  \chi_{\vec{p}}^2  ( \chi_{\vec{p}}^2 -1)\pth{ \pth{g_{\vec{p}} \nab g_{\vec{p}}  }^2  + g_{\vec{p}}^2 \pi_{\vec{p}} \nab g_{\vec{p}} } .
\end{align*}
Using \eqref{estim cutoff}, \eqref{estim kerr preli} and \eqref{estim kerr} we obtain
\begin{align*}
\left| \Phi\pth{(1-\chi_{\vec{p}})e + \chi_{\vec{p}}g_{\vec{p}}, \chi_{\vec{p}}\pi_{\vec{p}} }  \right| & \lesssim_\ga \pth{ \frac{m}{r^3} + \frac{m^2}{r^4} } \mathbbm{1}_{\{\la\leq r \leq 2\la\} }
\\&\lesssim_\ga \frac{m}{r^3} \mathbbm{1}_{\{\la\leq r \leq 2\la\} }.
\end{align*}
Using also \eqref{diff cutoff} \eqref{diff kerr preli} and \eqref{diff kerr} we obtain
\begin{align*}
&\left| \Phi\pth{e + \chi_{\vec{p}}(g_{\vec{p}} - e), \chi_{\vec{p}}\pi_{\vec{p}} } -  \Phi\pth{e + \chi_{\vec{p}\,'}(g_{\vec{p}\,'} - e), \chi_{\vec{p}\,'}\pi_{\vec{p}\,'} } \right|  
\\&\quad \lesssim_{\max(\ga,\ga')} \max(m,m') \frac{|\la-\la'|}{r^4}  \mathbbm{1}_{\{\min(\la,\la')\leq r \leq 2 \max(\la,\la')\}}
\\&\quad\quad + \pth{ \frac{|m-m'|}{r} + \max(m,m')\pth{ \frac{|y-y'|+|a-a'|}{r^2} + \frac{|v-v'|}{r}} }\pth{ \frac{1}{r^2} +  \frac{m}{r^3}}
\\&\hspace{3cm}\quad\times\mathbbm{1}_{\{\min(\la,\la')\leq r \leq 2 \max(\la,\la')\}}
\\&\quad \lesssim_{\max(\ga,\ga')}  \pth{ \frac{|m-m'|}{r^3} + \max(m,m')\pth{ \frac{|y-y'|+|a-a'|}{r^4} + \frac{|v-v'|}{r^3}} }
\\&\hspace{3cm}\times\mathbbm{1}_{\{\min(\la,\la')\leq r \leq 2 \max(\la,\la')\}}
\end{align*} 
where we used $\left| \la-\la'\right| \lesssim \left| y-y'\right| + \left| a-a'\right|$. This proves \eqref{constraint kerr} and \eqref{diff constraint kerr} and concludes the proof of Proposition \ref{prop kerr}.

\section{Appendix to Section \ref{section main results}}\label{appendix proof section 3}

\subsection{Proof of Lemma \ref{lem decomposition}} \label{section proof lem reduction}

Assuming the decomposition \eqref{decomposition}, i.e. that 
\begin{align*}
h= \u[h] e+ \nab\otimes \X[h] + h_{TT}
\end{align*}
with $h_{TT}$ a TT tensor, we have
\begin{equation}\label{tr div h}
\begin{aligned}
\tr h & = 3\u[h]+2\div\X[h],
\\ \div h & = \nab\u[h] + \nab\div\X[h] + \De\X[h].
\end{aligned}
\end{equation}
We apply the Laplacian to the first identity and compute the divergence of the second one:
\begin{align*}
\De \tr h & = 3\De \u[h]+2\De\div\X[h],
\\ \div\div h & = \De\u[h] + 2\De\div\X[h].
\end{align*}
Provided that we can define $\De^{-1}\div\div h$, this rewrites as
\begin{align*}
\u[h] & = \half \pth{ \tr h - \De^{-1}\div\div h} ,
\\ \div\X[h] & = \frac{3}{4}\De^{-1} \div\div h - \frac{1}{4} \tr h.
\end{align*}
Plugging these two relations into the second equation in \eqref{tr div h} finally gives the system
\begin{equation}\label{system u X h}
\begin{aligned}
\u[h] & = \frac{1}{2}\pth{ \tr h - \De^{-1}\div\div h },
\\ \De \X[h] & = \div h -\frac{1}{4}\nab \tr h -\frac{1}{4}\nab\De^{-1}\div\div h.
\end{aligned}
\end{equation}
Conversely, one can show that if $\u[h]$ and $\X[h]$ solve \eqref{system u X h} then $h_{TT}$ defined by \eqref{decomposition} is TT. Therefore, it only remains to show the existence of a unique solution $\pth{ \u[h],\X[h]}$ to \eqref{system u X h} with the appropriate bounds.

First, we need to make precise the meaning of $\De^{-1}\div\div h$. Since $h\in H^k_{-q-\de}$ we have in particular $\div\div h\in H^{k-2}_{-q-\de-2}$ and thanks to the first part of Proposition \ref{prop laplace} there exists a unique $\tilde{f}\in H^k_{-q-\de}$ with $\l \tilde{f}\r_{H^k_{-q-\de}}\lesssim \l h \r_{H^k_{-q-\de}}$ such that $f$ defined by
\begin{align*}
f =  \chi(r) \sum_{j=1}^q \sum_{\ell=-(j-1)}^{j-1} \langle \div\div h , w_{j,\ell} \rangle  v_{j,\ell}  + \tilde{f} 
\end{align*}
solves $\De f = \div\div h$. We then define $\De^{-1}\div\div h\vcentcolon = f $. Since $\langle \div\div h , w_{j,\ell} \rangle=0$ for $j=1,2$ we have $\De^{-1}\div\div h=\tilde{f}$ and thus $\l \De^{-1}\div\div h \r_{H^k_{-q-\de}}\lesssim \l h \r_{H^k_{-q-\de}}$. This also shows how one solve the first equation of \eqref{system u X h}, i.e. by simply defining 
\begin{align*}
\u[h]\vcentcolon = \half \pth{ \tr h - \De^{-1}\div\div h},
\end{align*}
which implies the bound $\l \u[h]\r_{H^k_{-q-\de}}\lesssim \l h \r_{H^k_{-q-\de}}$.

We now look at the second equation of \eqref{system u X h}. Its RHS 
\begin{align*}
Y = \div h -\frac{1}{4}\nab \tr h -\frac{1}{4}\nab\De^{-1}\div\div h
\end{align*}
satisfies $\l Y  \r_{H^{k-1}_{-q-\de-1}} \lesssim \l h\r_{H^k_{-q-\de}}$. We have the following dichotomy:
\begin{itemize}
\item if $q=1$, then we can use the isomorphism property of $\De$ (see Remark \ref{remark iso}) to deduce the existence of a unique $\X[h]\in H^{k+1}_{-q-\de+1}$ solving $\De \X[h]=Y$ with moreover $\l \X[h] \r_{H^{k+1}_{-q-\de+1}}\lesssim \l h \r_{H^k_{-q-\de}}$,
\item if $q=2$, then the second part of Proposition \ref{prop laplace} implies the existence of a unique $\widetilde{\X}[h]\in H^{k+1}_{-q-\de+1}$ with $\l \widetilde{\X}[h] \r_{H^{k+1}_{-q-\de+1}}\lesssim \l h\r_{H^k_{-q-\de}}$ and such that $\X[h]$ defined by
\begin{align*}
\X[h] & = \chi(r)\sum_{j=1}^q \sum_{\ell=-(j-1)}^{j-1} \sum_{k=1,2,3} \langle Y , W_{j,\ell,k} \rangle  V_{j,\ell,k} + \widetilde{\X}[h] 
\end{align*}
solves $\De \X[h]=Y$. Since $\langle Y , W_{1,0,k} \rangle=0$, $\left| \langle Y , W_{2,\ell,k} \rangle \right| \lesssim \l h\r_{H^k_{-q-\de}}$ and $\chi(r)r^{-2}\in H^{k+1}_{-q-\de+1}$, we obtain $\l \X[h] \r_{H^{k+1}_{-q-\de+1}}\lesssim \l h \r_{H^k_{-q-\de}}$.
\end{itemize}
In both cases we have constructed a solution $\X[h]$ to the second equation of \eqref{system u X h} with the bound $\l \X[h] \r_{H^{k+1}_{-q-\de+1}}\lesssim \l h \r_{H^k_{-q-\de}}$. The bound on $h_{TT}$ in \eqref{estim decomp} then simply follows from its definition and the bounds on $h$, $\u[h]$ and $\X[h]$. This concludes the proof of Lemma \ref{lem decomposition}.

\subsection{Proof of Proposition \ref{prop reduction}}\label{section proof reduction}

We set $\wc{v}=v-1$ and $\wc{\psi}=\psi - \mathrm{Id}$ and define the following symmetric 2-tensors:
\begin{align*}
\wc{g}\pth{\wc{v},\wc{\psi}} & \vcentcolon = (1+\wc{v})^{-4}\pth{ \pth{ \mathrm{Id} + \wc{\psi}}^*\g_{|_{\Si_0}}} - e,
\\ \wc{\pi}\pth{Y, \wc{\psi}} & \vcentcolon = \pi^{\pth{\pth{\mathrm{Id}+\wc{\psi}}^*\g,\Si_0}} - L_eY.
\end{align*}
We start by expanding $\wc{g}\pth{\wc{v},\wc{\psi}} $, already assuming $|\wc{v}|\leq \half$. We obtain
\begin{align}
\wc{g}\pth{\wc{v},\wc{\psi}} &  = \wc{g}(0,0) - 4\wc{v} e + \nab \otimes \wc{\psi}  + \mathrm{Err}_{\wc{g}}\pth{\wc{v},\wc{\psi}},\label{expansion wc g}
\end{align}
where the error term is given schematically by 
\begin{align*}
\mathrm{Err}_{\wc{g}}\pth{\wc{v},\wc{\psi}} & =  (\g-\m)\nab\wc{\psi}  + \pth{\nab\wc{\psi}}^2  + \wc{v}\pth{\g-\m + \nab\wc{\psi}_i} + \wc{v}^2,
\end{align*}
where $\psi_i$ denotes spatial components of $\psi$. Note that in the previous expression of $\mathrm{Err}_{\wc{g}}\pth{\wc{v},\wc{\psi}}$ we neglected all cubic and higher error terms since they behave better than the quadratic ones. This convention also applies to all schematic expressions in the rest of this section.

We now expand $\wc{\pi}\pth{Y, \wc{\psi}}$. We recall the general expression
\begin{align*}
\pi^{(\g,\Si_0)}_{ij} & = -\half \T^{(\g,\Si_0)} (\g_{ij}) + \half \g\pth{ \left[\T^{(\g,\Si_0)},\dr_{(i}\right],\dr_{j)}  } 
\\&\quad + \half g^{ab}   \T^{(\g,\Si_0)} (\g_{ab})  g_{ij} - g^{ab}\g\pth{ \left[\T^{(\g,\Si_0)},\dr_{a}\right],\dr_{b}  } g_{ij} ,
\end{align*}
where $g=\g_{|_{\Si_0}}$ and $\T^{(\g,\Si_0)}$ is the future directed unit normal to $\Si_0$ for $\g$. Using the schematic expansion of the unit normal for a spacetime metric close to Minkowski spacetime
\begin{align*}
\T^{(\g,\Si_0)} - \dr_t - \half (\g-\m)_{00}\dr_t + \de^{ij}(\g-\m)_{0i}\dr_j & = (\g-\m)^2
\end{align*}
and the definition of the pullback we first obtain the schematic expansion
\begin{align*}
\T^{\pth{\pth{\mathrm{Id}+\wc{\psi}}^*\g,\Si_0}} -  \T^{(\g,\Si_0)} -  \dr_t \wc{\psi}_0 \dr_t + \de^{ij} \pth{ \dr_t \wc{\psi}_i + \dr_i \wc{\psi}_0 } \dr_j  & =  (\g-\m)\nab\wc{\psi} + \pth{\nab\wc{\psi}}^2,
\end{align*}
and finally
\begin{align}\label{expansion wc pi}
\wc{\pi}\pth{Y, \wc{\psi}} & = \wc{\pi}(0,0) + \Hess \wc{\psi}_0 - \De \wc{\psi}_0 e  - L_e Y + \mathrm{Err}_{\wc{\pi}}\pth{\wc{\psi}},
\end{align}
where the error term is given schematically by 
\begin{align*}
\mathrm{Err}_{\wc{\pi}}\pth{\wc{\psi}} & = (\g-\m)\nab\wc{\psi} + \pth{ \nab\wc{\psi}}^2 + \dr_t (\g-\m)\nab\wc{\psi} 
\\&\quad + \nab (\g-\m)\nab\wc{\psi} + (\g-\m)\nab^2\wc{\psi} + \nab\wc{\psi}\nab^2\wc{\psi}.
\end{align*}
Note that in this error term, time derivatives of $\wc{\psi}$ \textit{a priori} appear but we impose $\dr_t \wc{\psi}=0$ (note however that in the main terms in \eqref{expansion wc pi}, time derivatives of $\wc{\psi}$ actually cancel out).

We want to find $\pth{\wc{v},\wc{\psi},Y}$ such that $\wc{g}\pth{\wc{v},\wc{\psi}} $ and $\wc{\pi}\pth{Y, \wc{\psi}}$ are TT tensors. According to the expansions \eqref{expansion wc g} and \eqref{expansion wc pi} and using the notation of Lemma \ref{lem decomposition} this is equivalent to the following system for $\pth{\wc{v},\wc{\psi},Y}$:
\begin{align}
\wc{v} & = \frac{1}{4}\u\left[ \wc{g}(0,0)\right] + \frac{1}{4}\u\left[\mathrm{Err}_{\wc{g}}\pth{\wc{v},\wc{\psi}} \right], \label{eq v app}
\\ \wc{\psi}_i & = - \X\left[ \wc{g}(0,0)\right] - \X\left[\mathrm{Err}_{\wc{g}}\pth{\wc{v},\wc{\psi}} \right]_i, \label{eq phi i}
\\ \De \wc{\psi}_0 & = 2\u\left[ \wc{\pi}(0,0)\right] + 2\u\left[\mathrm{Err}_{\wc{\pi}}\pth{\wc{\psi}} \right] + 2 \div\X\left[ \wc{\pi}(0,0)\right] +2\div\X\left[ \mathrm{Err}_{\wc{\pi}}\pth{\wc{\psi}} \right] \label{eq phi 0},
\\ Y & = \half \nab \wc{\psi}_0 + \X\left[ \wc{\pi}(0,0)\right] + \X\left[\mathrm{Err}_{\wc{\pi}}\pth{\wc{\psi}} \right] \label{eq Y},
\end{align}
where we used $\u[h_1+h_2]=\u[h_1]+\u[h_2]$ and $\X[h_1+h_2]=\X[h_1]+\X[h_2]$ which follows from the uniqueness part of Lemma \ref{lem decomposition}. Indeed, one can check that if $\pth{\wc{v},\wc{\psi},Y}$ satisfy \eqref{eq v app}-\eqref{eq Y} then we have
\begin{align*}
\wc{g}\pth{\wc{v},\wc{\psi}} & = \pth{ \wc{g}(0,0)   + \mathrm{Err}_{\wc{g}}\pth{\wc{v},\wc{\psi}}  }_{TT},
\\ \wc{\pi}\pth{Y, \wc{\psi}} & = \pth{\wc{\pi}(0,0) + \mathrm{Err}_{\wc{\pi}}\pth{\wc{\psi}}}_{TT}.
\end{align*}
It thus remains to solve the system \eqref{eq v app}-\eqref{eq phi 0} on $\Si_0$ via a Banach fixed point argument and then define $Y$ by \eqref{eq Y}. We don't write down all the details and simply state \textit{a priori} estimates which easily imply both the boundedness and contraction part of the fixed point argument.

First, note that the assumption \eqref{spacetime assumption} together with the bounds of Lemma \ref{lem decomposition} in the cases $(k,q)=(3,1)$ and $(k,q)=(2,2)$ implies
\begin{align}\label{terme principal}
\l \u[\wc{g}(0,0)]\r_{H^3_{-1-\de}}+\l \X[\wc{g}(0,0)]\r_{H^{4}_{-\de}} + \l \u[\wc{\pi}(0,0)]\r_{H^2_{-2-\de}}+\l \X[\wc{\pi}(0,0)]\r_{H^{3}_{-1-\de}} \lesssim \e.
\end{align}
We assume that there exists $A_0>0$ such that
\begin{align*}
\l \wc{v} \r_{H^3_{-1-\de}} + \l \wc{\psi} \r_{H^4_{-\de}} \leq A_0 \e.
\end{align*}
Using also \eqref{terme principal}, this implies the following estimates on the error terms
\begin{align*}
\l \mathrm{Err}_{\wc{g}}\pth{\wc{v},\wc{\psi}} \r_{H^3_{-1-\de}} + \l \mathrm{Err}_{\wc{\pi}}\pth{\wc{\psi}} \r_{H^2_{-2-\de}}  \lesssim C(A_0) \e^2,
\end{align*}
where we also used the embeddings $H^3_{-1-\de}\times H^3_{-1-\de}\subset H^3_{-1-\de}$, $H^3_{-1-\de}\times H^3_{-1-\de}\subset H^2_{-2-\de}$ and $H^3_{-1-\de}\times H^2_{-2-\de}\subset H^2_{-2-\de}$ from Lemma \ref{lem plongement}. Thanks to the bounds of Lemma \ref{lem decomposition} in the cases $(k,q)=(3,1)$ and $(k,q)=(2,2)$ we obtain
\begin{align}\label{terme erreur}
& \l \u\left[\mathrm{Err}_{\wc{g}}\pth{\wc{v},\wc{\psi}}\right]\r_{H^3_{-1-\de}}+\l \X\left[\mathrm{Err}_{\wc{g}}\pth{\wc{v},\wc{\psi}}\right]\r_{H^{4}_{-\de}} \nonumber
\\&\hspace{3cm}+ \l \u\left[\mathrm{Err}_{\wc{\pi}}\pth{\wc{\psi}}\right]\r_{H^2_{-2-\de}}+\l \X\left[\mathrm{Err}_{\wc{\pi}}\pth{\wc{\psi}}\right]\r_{H^{3}_{-1-\de}} \lesssim C(A_0) \e^2.
\end{align}
We now estimate $\pth{\wc{v},\wc{\psi}}$ with the equations \eqref{eq v app}-\eqref{eq phi 0}. For $\wc{v}$ and $\wc{\psi}_i$ we directly obtain
\begin{align*}
\l \wc{v} \r_{H^3_{-1-\de}} + \l \wc{\psi}_i \r_{H^4_{-\de}} & \lesssim \e + C(A_0) \e^2,
\end{align*}
where we used \eqref{terme principal} and \eqref{terme erreur}. For $\wc{\psi}_0$, we estimate the RHS of \eqref{eq phi 0}:
\begin{align*}
\l \text{RHS of \eqref{eq phi 0}} \r_{H^2_{-2-\de}} \lesssim \e + C(A_0) \e^2,
\end{align*}
where we used \eqref{terme principal} and \eqref{terme erreur}. We then use the isomorphism property of $\De$ (see Remark \ref{remark iso}) to deduce the bound
\begin{align*}
\l \wc{\psi}_0 \r_{H^4_{-\de}} \lesssim \e + C(A_0) \e^2.
\end{align*}
Thanks to these \text{a priori} estimates we can solve the nonlinear system \eqref{eq v app}-\eqref{eq phi 0} and obtain the bound
\begin{align*}
\l \wc{v} \r_{H^3_{-1-\de}} + \l \wc{\psi} \r_{H^4_{-\de}} \lesssim \e.
\end{align*}
Defining $Y$ by \eqref{eq Y}, \eqref{terme principal} and \eqref{terme erreur} imply
\begin{align*}
\l Y \r_{H^3_{-1-\de}} \lesssim \e.
\end{align*}
This concludes the proof of Proposition \ref{prop reduction}.

\subsection{Proof of Proposition \ref{prop F}}\label{section proof prop F}

The first part of Proposition \ref{prop F} is proved as follows. We define $c_k=\int_{\RRR^3} \varpi^{ij} \dr_k h_{ij}$. We apply Cauchy-Schwarz in $L^2$ and then in $\ell^2$ and get $|c_k|^2\leq \l \varpi \r_{L^2}^2 \l \dr_k h \r_{L^2}^2$ which gives
\begin{align*}
\sqrt{\sum_{k=1,2,3} |c_k|^2} & \leq \sqrt{\l \varpi \r_{L^2}^2\sum_{k=1,2,3}  \l \dr_k h \r_{L^2}^2}
\\& = \l \varpi \r_{L^2}\l \nabla h \r_{L^2}.
\end{align*}
We now apply the inequality $ab\leq \frac{a^2}{2}+\frac{b^2}{2}$ with $a=\frac{1}{\sqrt{2}}\l \nabla h \r_{L^2}$ and $b=\sqrt{2}\l \varpi \r_{L^2}$ and obtain
\begin{align*}
\sqrt{\sum_{k=1,2,3} |c_k|^2} & \leq \frac{1}{4}\l \nabla h \r_{L^2}^2 + \l \varpi \r_{L^2}^2,
\end{align*}
i.e. $J(h,\varpi)\leq 1$.

If we assume $J(h,\varpi)=1$, then 
\begin{align*}
\sum_{k=1,2,3} \pth{  \l \varpi \r_{L^2}^2 \l \dr_k h \r_{L^2}^2 - |c_k|^2 }= 0.
\end{align*}
Since these three quantities are nonnegative, they must all vanish and for all $k=1,2,3$ we have $|c_k|^2= \l \varpi \r_{L^2}^2 \l \dr_k h \r_{L^2}^2$. Therefore, we have the equality case in the Cauchy-Schwarz inequality, i.e. there exists $\la_k$ for $k=1,2,3$ such that $ \dr_k h_{ij}  = \la_k  \varpi_{ij} $ for all $i,j=1,2,3$. Moreover, the equality case in the inequality $ab\leq \frac{a^2}{2}+\frac{b^2}{2}$ implies that $\l \nabla h \r_{L^2}=2\l \varpi \r_{L^2}$, which in turn implies that $\sum_{k=1,2,3}\la_k^2=4$. Therefore, the vector $(\la_1,\la_2,\la_3)$ is non-zero and there exists a non-zero vector $X\in\RRR^3$ such that $X^k\la_k=0$. This implies that each $h_{ij}$ satisfies the transport equation $X(h_{ij})=0$. This is incompatible with $h$ being in $\dot{H}^1$ and non-zero. Therefore the equality case is impossible and we have proved that 
\begin{align*}
J(h,\varpi)<1.
\end{align*}

We now turn to the second part of Proposition \ref{prop F}. As formally shown above, up to a linear change of coordinates, the upper bound 1 is reached when $\wc{\pi}=\half\dr_3\wc{g}$. Therefore, in order to saturate the bound $J<1$, we will construct a sequence $(\wc{g}_n,\wc{\pi}_n)_{n\in\mathbb{N}}$ of TT tensors with $\wc{\pi}_n=\half \dr_3 \wc{g}_n$. We first note that if $(\wc{g}_n)_{n\in\mathbb{N}}$ is a sequence of symmetric 2-tensors such that
\begin{align}\label{condition}
\lim_{n\to +\infty} \frac{\l\dr_1 \wc{g}_n \r_{L^2} }{\l \dr_3 \wc{g}_n \r_{L^2}}=\lim_{n\to +\infty} \frac{ \l \dr_2 \wc{g}_n \r_{L^2}}{\l \dr_3 \wc{g}_n \r_{L^2}} =0,
\end{align}
then we have
\begin{align*}
\lim_{n\to+\infty}J\pth{\sqrt{2}\wc{g}_n,\frac{1}{\sqrt{2}}\dr_3\wc{g}_n}=1.
\end{align*}
Indeed we have:
\begin{align*}
J\pth{\sqrt{2}\wc{g}_n,\frac{1}{\sqrt{2}}\dr_3\wc{g}_n} & = \frac{ \sqrt{\l \dr_3\wc{g}_n\r_{L^2}^4+ \sum_{k=1,2}\pth{\int_{\RRR^3} \dr_3 \wc{g}_n^{ij} \dr_k (\wc{g}_n)_{ij} }^2}}{  \l \dr_3\wc{g}_n\r_{L^2}^2+\frac{1}{2} \l \dr_1\wc{g}_n\r_{L^2}^2 + \frac{1}{2} \l \dr_2\wc{g}_n\r_{L^2}^2 }
\\&\geq \frac{ 1}{  1+ \frac{ \l \dr_1\wc{g}_n\r_{L^2}^2}{2\l \dr_3\wc{g}_n\r_{L^2}^2} + \frac{ \l \dr_2\wc{g}_n\r_{L^2}^2}{2\l \dr_3\wc{g}_n\r_{L^2}^2}  }.
\end{align*}
Thanks to the assumption \eqref{condition} this last quantity converges to 1, and thanks to the first part of Proposition \ref{prop F} we have $J\pth{\sqrt{2}\wc{g}_n,\frac{1}{\sqrt{2}}\dr_3\wc{g}_n}<1$ so that
\begin{align*}
\lim_{n\to+\infty}J\pth{\sqrt{2}\wc{g}_n,\frac{1}{\sqrt{2}}\dr_3\wc{g}_n}=1.
\end{align*}

In view of this fact, in order to prove the second part of Proposition \ref{prop F} it is enough to exhibit a sequence $(\wc{g}_n)_{n\in\mathbb{N}}$ of TT tensors satisfying \eqref{condition}. To this end, we first consider $\phi:\RRR\longrightarrow\RRR$ a non-zero compactly supported function, $\chi:\RRR\longrightarrow\RRR$ a smooth cutoff function such that $\chi=1$ on $[0,1]$ and $\chi=0$ on $[2,+\infty)$ and define
\begin{align*}
g_n \vcentcolon = \phi''(x^3) \chi\pth{\frac{\rho}{n}}M,
\end{align*}
where $\rho=\sqrt{(x^1)^2+(x^2)^2}$ and $M=(\d x^1)^2 - (\d x^2)^2$. This defines a traceless symmetric 2-tensor on $\RRR^3$ with
\begin{align*}
(\div g_n)_1 & = \frac{1}{n} \phi''(x^3) \chi'\pth{\frac{\rho}{n}}\frac{x^1}{\rho},
\\ (\div g_n)_1 & = - \frac{1}{n} \phi''(x^3) \chi'\pth{\frac{\rho}{n}}\frac{x^2}{\rho},
\\ (\div g_n)_3 & = 0.
\end{align*}
Moreover, the sequence $g_n$ satisfies
\begin{align*}
\l \dr_3 g_n \r_{L^2} & = c_0 n ,
\\ \l \div g_n \r_{L^2} + \l \dr_1 g_n\r_{L^2} + \l\dr_2 g_n \r_{L^2} & \lesssim 1, 
\end{align*}
with 
\begin{align*}
c_0 \vcentcolon = \sqrt{4\pi \pth{ \int_{\RRR}(\phi'''(z))^2\d z}\pth{\int_{\RRR_+}\chi(\rho)\rho\d\rho}}>0,
\end{align*}
and where the implicit constant in $\lesssim$ depends on $\phi$ and $\chi$. We need to correct the fact that $g_n$ is not divergence free. We define a symmetric 2-tensor $h_n$ directly in Cartesian coordinates:
\begin{itemize}
\item we first define the $(h_n)_{3i}$ components for $i=1,2$ by $ \dr_3 (h_n)_{3i} = - (\div g_n)_i$ with the requirement $\lim_{x^3\to-\infty}(h_n)_{3i}=0$,
\item we then define the $(h_n)_{33}$ component by $\dr_3 (h_n)_{33}=-\dr_1 (h_n)_{31} - \dr_2 (h_n)_{32}$ with the requirement $\lim_{x^3\to-\infty}(h_n)_{33}=0$,
\item we finally define the remaining components $(h_n)_{ij}$ for $i,j=1,2$ as the solution to
\begin{equation}\label{system 2d}
\begin{aligned}
\div^{(2)}h_n & = 0,
\\ \tr^{(2)}h_n & = -(h_n)_{33},
\end{aligned}
\end{equation}
where $\div^{(2)}$ and $\tr^{(2)}$ denote the divergence and trace operator in the $(x^1,x^2)$ variables.
\end{itemize}
By construction $\wc{g}_n\vcentcolon=g_n+h_n$ is a TT tensor. It remains to show that the sequence $(\wc{g}_n)_{n\in\mathbb{N}}$ satisfies \eqref{condition}, and before that to properly define and estimate $h_n$.

\paragraph{The components $(h_n)_{3i}$ for $i=1,2,3$.} From the definition of $(h_n)_{31}$ and $(h_n)_{32}$ and the expression of $\div g_n$ we have
\begin{align*}
(h_n)_{31} & = - \frac{1}{n} \phi'(x^3) \chi'\pth{\frac{\rho}{n}} \frac{x^1}{\rho},
\\ (h_n)_{32} & = \frac{1}{n} \phi'(x^3) \chi'\pth{\frac{\rho}{n}} \frac{x^2}{\rho}.
\end{align*}
This implies that
\begin{align*}
\dr_3 (h_n)_{33} & = -\dr_1 (h_n)_{31} - \dr_2 (h_n)_{32}
\\& = \frac{1}{n} \phi'(x^3)\pth{ \dr_1 \pth{ \chi'\pth{\frac{\rho}{n}} \frac{x^1}{\rho} } - \dr_2 \pth{  \chi'\pth{\frac{\rho}{n}} \frac{x^2}{\rho} } }
\\& = \frac{1}{n} \phi'(x^3) \frac{(x^1)^2-(x^2)^2}{\rho^2} \pth{  \frac{1}{n} \chi''\pth{\frac{\rho}{n}} -  \frac{1}{\rho}    \chi' \pth{\frac{\rho}{n}}    },
\end{align*}
which in turn yields
\begin{align*}
(h_n)_{33} & =\frac{1}{n} \phi(x^3) \frac{(x^1)^2-(x^2)^2}{\rho^2} \pth{  \frac{1}{n} \chi''\pth{\frac{\rho}{n}} -  \frac{1}{\rho}    \chi'  \pth{\frac{\rho}{n}}   }.
\end{align*}
For $i,j=1,2$, we obtain the estimates
\begin{align*}
\l \dr_3 (h_n)_{3i} \r_{L^2} & \lesssim 1,
\\ \l \dr_j (h_n)_{3i} \r_{L^2} +\l \dr_3 (h_n)_{33} \r_{L^2}  & \lesssim \frac{1}{n},
\\ \l \dr_j (h_n)_{33} \r_{L^2} & \lesssim \frac{1}{n^2}.
\end{align*}

\paragraph{The components $(h_n)_{ij}$ for $i,j=1,2$.} We first define a scalar function on $\RRR^2$ by
\begin{align*}
f(x^1,x^2) & =  \frac{(x^1)^2-(x^2)^2}{\rho^2} \pth{  \chi''(\rho)-  \frac{1}{\rho}    \chi'(\rho)   }
\end{align*}
so that $(h_n)_{33}=\frac{1}{n^2}\phi(x^3)f\pth{\frac{x^1}{n},\frac{x^2}{n}}$. We now search $(h_n)_{ij}$ under the form 
\begin{align*}
(h_n)_{ij}(x^1,x^2,x^3)=\frac{1}{n^2}\phi(x^3)\tilde{h}_{ij}\pth{\frac{x^1}{n},\frac{x^2}{n}},
\end{align*}
where $\tilde{h}$ is a symmetric 2-tensor defined on $\RRR^2$. We also define $\hat{h}=\tilde{h} + \half f e^{(2)}$ (where $e^{(2)}$ is the Euclidean metric on $\RRR^2$). The system \eqref{system 2d} is then equivalent to the following system for $\hat{h}$:
\begin{equation}\label{system 2d bis}
\begin{aligned}
\div^{(2)}\hat{h} & = \half \nab^{(2)} f,
\\ \tr^{(2)}\hat{h} & = 0,
\end{aligned}
\end{equation}
where $\nab^{(2)}=(\dr_1,\dr_2)$. We now use a consequence of Corollary 2.11 from \cite{Huneau2016a} on the divergence equation on $\RRR^2$ for traceless symmetric 2-tensor: since $f$ is smooth and compactly supported and since $\int_{\RRR^2} \dr_i f=0$, there exists a traceless symmetric 2-tensor $\hat{h}$ solution of \eqref{system 2d bis} and such that
\begin{align*}
\l \hat{h} \r_{L^2(\RRR^2)} + \l \nab\hat{h} \r_{L^2(\RRR^2)} \lesssim 1,
\end{align*}
where the implicit constant depends on $f$, i.e. on $\chi$ and $\phi$. We thus get the estimates for $k,i,j=1,2$
\begin{align*}
\l \dr_3 (h_n)_{ij} \r_{L^2(\RRR^2)} & \lesssim \frac{1}{n},
\\ \l \dr_k (h_n)_{ij} \r_{L^2(\RRR^2)} & \lesssim \frac{1}{n^{2}}.
\end{align*}

\paragraph{Verification of \eqref{condition}.} Thanks to $\l \dr_3 g_n\r_{L^2}=c_0n$ and the various estimates on $\l \dr_3 h_n \r_{L^2}$ above we have
\begin{align*}
\l \dr_3 \wc{g}_n \r_{L^2} \geq c_0n - c_1 - \frac{c_2}{n},
\end{align*}
for some $c_1,c_2>0$. Therefore there exists $n_0$ such that if $n\geq n_0$ then $\l \dr_3 \wc{g}_n \r_{L^2} \geq \frac{c_0}{2}n$. For $n\geq n_0$ we then have
\begin{align*}
\frac{\l \dr_1 \wc{g}_n \r_{L^2}+\l \dr_2 \wc{g}_n \r_{L^2} }{\l \dr_3 \wc{g}_n \r_{L^2}} & \lesssim   \frac{ 1 + \frac{1}{n} + \frac{1}{n^{2}} }{n}.
\end{align*}
This shows that the sequence of TT tensors $(\wc{g}_n)_{n\geq n_0}$ satisfies \eqref{condition}. This concludes the proof of Proposition \ref{prop F}.

\subsection{Proof of Proposition \ref{prop uniqueness}}\label{appendix uniqueness}

From the decay assumptions on $\wc{g}_i$, $\wc{\pi}_i$, $u_i-1$ and $X_i$, the support properties of $\breve{g}_i$ and $\breve{\pi}_i$, and \eqref{uniq1}-\eqref{uniq2} we deduce that for large $r$ we have
\begin{align*}
g & = e + \chi_{\vec{p}_i} (g_{\vec{p}_i} - e ) + \GO{r^{-q-\de}},
\\ \pi & = \chi_{\vec{p}_i} \pi_{\vec{p}_i}  + \GO{r^{-q-\de-1}},
\end{align*}
for $i=1,2$. Therefore the asymptotics charges of $(g,\pi)$ are given by the ones of $(g_{\vec{p}_i},\pi_{\vec{p}_i})$. Since the asymptotic charges completely determine the black hole parameter $\vec{p}$, we have $\vec{p}_1=\vec{p}_2=\vcentcolon\vec{p}$. We then compute the trace of $g$ with respect to the Euclidean metric from \eqref{uniq1}. Since $\wc{g}_i$ and $\breve{g}_i$ are traceless by assumption (recall that tensors in $\AA_q$ are traceless, see Proposition \ref{prop KIDS}) this gives  
\begin{align*}
u^4_1 \tr\pth{e + \chi_{\vec{p}} (g_{\vec{p}} - e )} = u^4_2 \tr\pth{e + \chi_{\vec{p}} (g_{\vec{p}} - e )}
\end{align*}
and implies $u_1=u_2$ and furthermore
\begin{align*}
\wc{g}_1 + \breve{g}_1 = \wc{g}_2 + \breve{g}_2.
\end{align*}
Since $\wc{g}_i$ is divergence-free we can take the divergence of this identity and get $\div (\breve{g}_1-\breve{g}_2)=0$. Remember now that $\breve{g}_1-\breve{g}_2\in\AA_q$ is a linear combination of $h_{j,\ell}$ (recall Definition \ref{def Aq Bq}) and that the 1-forms $\div h_{j,\ell}$ are linearly independent (see Proposition \ref{prop KIDS}). Therefore we obtain $\breve{g}_1=\breve{g}_2$ and then $\wc{g}_1=\wc{g}_2$. 

We now look at the reduced second fundamental form, \eqref{uniq2} gives
\begin{align*}
\wc{\pi}_1 + \breve{\pi}_1 + L_{e }X_1  & = \wc{\pi}_2 + \breve{\pi}_2 + L_{e}X_2. 
\end{align*}
We apply the operator $\De\tr + \div\div$ to this equality and use the fact that $(\De\tr + \div\div)(L_e X)=0$ for all vector field $X$ and the fact that $\wc{\pi}_i$ is TT, we obtain
\begin{align*}
(\De\tr + \div\div)( \breve{\pi}_1-\breve{\pi}_2) = 0.
\end{align*}
Remember now that $\breve{\pi}_1-\breve{\pi}_2\in\BB_q$ is a linear combination of $\varpi_Z$ (recall Definition \ref{def Aq Bq}) and that the functions $(\De\tr + \div\div)\varpi_Z$ are linearly independent (see Proposition \ref{prop KIDS}). Therefore we obtain $\breve{\pi}_1=\breve{\pi}_2$ and
\begin{align*}
\wc{\pi}_1 + L_{e }X_1  & = \wc{\pi}_2 + L_{e}X_2. 
\end{align*}
We take the divergence of this equality and get $\De(X_1-X_2)=0$ (since $\wc{\pi}_i$ is divergence-free). Thanks to the decay property of $X_i$, this implies $X_1=X_2$ and finally $\wc{\pi}_1=\wc{\pi}_2$.

%

\end{document}